\documentclass[15pt]{article}

\usepackage[utf8]{inputenc}
\usepackage{amssymb}
\usepackage{amsmath}
\usepackage{amsthm}
\usepackage{mathtools}
\usepackage{caption}
\usepackage{hyperref}
\usepackage{algorithm, algorithmicx, algpseudocode}
\usepackage{subcaption}
\usepackage[square,numbers]{natbib}
\usepackage{comment}
\usepackage{booktabs}
\usepackage{multirow}
\usepackage{graphicx}
\newtheorem{theorem}{Theorem}

\newtheorem{definition}{Definition}

\DeclareMathOperator*{\argmin}{argmin}
\DeclareMathOperator*{\argmax}{argmax}

\newcommand\myPart[2]{^{#1}_{#2}}

\title{The Lawson-Hanson Algorithm with Deviation Maximization: \\ Finite Convergence and Sparse Recovery}

\usepackage{authblk}
\author[1]{Monica Dessole}
\author[1]{Marco Dell'Orto}
\author[1]{Fabio Marcuzzi}
\affil[1]{Department of Mathematics “Tullio Levi Civita”, University of Padova, Via Trieste 63, 35121 Padova, Italy}
\date{August 2021}                     
\begin{document}

\maketitle

\begin{abstract}
In this work we devise a new version of the Lawson-Hanson with Deviation Maximization (LHDM) algorithm, a block method for the solution of NonNegative Least Squares (NNLS). 
It is shown that the new algorithm here presented convergences in a finite number of steps under suitable conditions. An extensive campaign of experiments is performed in order to evaluate the performance gain with respect to the standard Lawson-Hanson algorithm. We also explore the sparse recovery ability of LHDM, comparing it against several $\ell_1$-minimization solvers in terms of solution quality and time-to-solution on a large set of dense instances.
\end{abstract}


\medskip

AMS:  65F25, 65K99

\section{Introduction}

NonNegative Least Squares (NNLS) problems arise in many applications where data points can be represented as nonnegative linear combinations of some meaningful components. Such problems are frequently encountered in signal and image processing and they are core problems in more complex computations, such as nonnegative matrix and tensor decompositions, see e.g. \cite{Lee-Seung, NMFbookGillis}. Moreover, when dealing with underdetermined systems of equations, the nonnegativity constraint is known to naturally enhance sparsity of the solution, that is the solution attained has few nonzeros, see e.g. \cite{Elad2008_uniqueness_nonnegative, FKo14, WTa09, WXT11}. An important outcome of this body of work is that nonnegativity alone may attain a satisfactory sparse recovery. Moreover, NNLS solvers can be adopted to solve unconstrained least squares problems with minor adjustments. Over the last two decades, sparsity has become one of the most relevant topics in signal processing. In general, sparsity in signals describes the phenomenon where high dimensional data can be expressed with few measurements and it results in finding a sparse solution to underdetermined systems of equations. This problem is highly nonconvex and its solution is NP-hard in general, although it is well known that $\ell_1$-minimization leads to the sparsest solution for a restricted class of matrices. 
This fact is known as $\ell_0-\ell_1$ equivalence and it has been found empirically \cite{CSD01} and theoretically \cite{Donoho-Huo_2001, EBr02}. 

The first algorithm devised for NNLS is the Lawson-Hanson algorithm \cite{LH95}. 
Since this seminal work, many modifications have been proposed in order to improve the standard Lawson-Hanson algorithm:
\citet{BJo97} have proposed a variation specifically designed for use in nonnegative tensor decompositions; their algorithm, called ``fast NNLS'' (FNNLS), reduces the execution time by avoiding redundant computations in nonnegative matrix factorization problems arising in tensor decompositions and performs well with multiple right-hand-sides problems, which is not the case here discussed, thus we omit a comparison. \citet{VBK04} presented a different NNLS solution algorithm, namely ``fast combinatorial NNLS'' (FCNNLS), also designed for the specific case of a large number of right-hand sides. 
Principal block pivoting method introduced by \citet{PJV94} is an active set type algorithm which differs from the standard Lawson-Hanson algorithm, since the sequence of iterates produced does not necessarily fall into the feasible region. The convergence is ensured provided the objective function is strictly convex, while when we deal with underdetermined matrices it is simply convex, therefore this algorithm may fail in sparse recovering. Surprisingly, the Lawson-Hanson algorithm \cite{LH95} does not suffer from this drawback. 

The Lawson-Hanson algorithm is an active set method which iteratively builds an optimal solution by solving at each iteration an unconstrained least squares subproblem, e.g. by computing a QR decomposition, on a subset of columns of the initial matrix, in such a way that the objective value strictly decreases while the iterates are kept within the feasible region. 
The algorithm performs a column pivoting by selecting one column at a time and therefore it relies only on BLAS--2 operations, resulting in poor performance. In an actual implementation, the columns are not permuted, rather a vector of indices is used to keep trace of columns ordering. In general, column pivoting makes it more difficult to achieve high performances in QR computation, see \cite{BHa92,BQu98,BQu98b,QSB98,Bis89}.

In this work, we address the NNLS problem in which the objective function matrix is dense and we present a modified Lawson-Hanson algorithm in which multiple column selection is performed in order to exploit BLAS--3 operations for efficiency. 
The column selection is based on the ``deviation maximization" pivoting presented in \cite{DMa21} to obtain rank-revealing QR factorizations. The main difference is that the column selection here does not aim at revealing the rank of the matrix, but is driven from the active-set method to solve the non-negativity constrained problem. The proposed algorithm is named Lawson-Hanson with Deviation Maximization (LHDM). 
A first try of combining the two techniques was presented in \cite{DMV20DRNA}, where we observed good sparsity recovery capacity and a significant performance gain with respect to the standard Lawson-Hanson algorithm when dealing with underdetermined systems. However, in this preliminary version the algorithm may fall in an infinite loop without converging to a solution: indeed, it is known that manipulating multiple columns at once may allow cycling. In the present paper, we introduce a new mechanism which allows a similar performance gain, but the termination in a finite number of steps is ensured. Moreover, we explore its sparsity recovery ability against several $\ell_1$-minimization solvers, extending a comparison that has been done in the recent literature \cite{AAVV_Solving_Basis_Pursuit-Heuristic_Optimality_Check_and_Solver_Comparison}.

The rest of the paper is organized as follows. In Section \ref{sec:nnls}, we formally introduce the NNLS problem and the Lawson-Hanson algorithm. We present the deviation maximization algorithm for column selection and the new version of the Lawson-Hanson with Deviation Maximization algorithm. We provide a theoretical result about finite convergence of LHDM by generalizing the analogous result for the standard Lawson-Hanson. In Section \ref{sec:sparse_recovery}, we introduce the $\ell_0$-minimization problem as sparse recovery and the convex $\ell_1$-minimization problem. We recall results that ensure the so-called $\ell_0-\ell_1$ equivalence and conditions under which the underdetermined NNLS problem has a unique solution. Last, we recall how arbitrary signed sparse recovery can be attained by NNLS solvers. In Section \ref{sec:numerical_experiments}, we present a comparison of LHDM against several $\ell_1$-minimization solvers in terms of performance and solution found. The results are shown with an extensive campaign of experiments. Finally, Section \ref{sec:conclusions} concludes the paper.

\section{Active set algorithms for nonnegative least squares} \label{sec:nnls}
 
Consider the nonnegative least squares problem 
\begin{equation}
    \min_{\mathbf{x}} \Vert A \mathbf{x} -\mathbf{b} \Vert_2, \qquad \text{s.t. }{\mathbf{x}\geq 0}, \label{eq:NNLS}
\end{equation}
where $A = (\mathbf{a}_1\ \dots\ \mathbf{a}_n)$ is a matrix of size $m \times n$, $\mathbf{b}$ is a vector of length $m$. In this work, we address the case in which $A$ is a dense matrix, i.e. it is stored as a full array with $mn$ entries. The case in which $A$ is sparse would require the design of a specific implementation in order to exploit the presence of many null entries.
Trivially, a solution ${\mathbf{x}}^{\star}$ of \eqref{eq:NNLS} also solves
\begin{equation*}
    \min_{\mathbf{x}} \frac{1}{2} \Vert A \mathbf{x} -\mathbf{b} \Vert_2^2, \qquad \text{s.t. }{\mathbf{x}\geq 0}. \label{eq:NNLS_squared}
\end{equation*}
Define $\phi( \mathbf{x}) := \Vert A \mathbf{x} -\mathbf{b} \Vert_2^2$.
Note that the gradient of $\phi$ at a point $\mathbf{x}$, i.e. steepest descend direction with opposite sign, is 
\begin{equation*}
    \nabla \phi ( \mathbf{x}) = A^TA \mathbf{x} - A^T \mathbf{b} = -A^T( \mathbf{b} - A \mathbf{x} ) = -A^T \mathbf{r}( \mathbf{x}),
\end{equation*}
where $\mathbf{r}$ is the residual function $ \mathbf{r}( \mathbf{x}) = \mathbf{b} - A \mathbf{x}$. The set of points satisfying the constraints is called feasible region, 
here $\Omega = \left\{ \mathbf{x} \in \mathbb{R}^n :\  \mathbf{x}\geq 0 \right\}$, and a point $\mathbf{x}$ is said to be feasible if it belongs to the feasible region. Let us recall the following basic definition.

\begin{definition}
Let $\mathbf{x}\in\mathbb{R}^n$ be a feasible point, i.e. $\mathbf{x}\geq 0$. A vector $\mathbf{s}\in\mathbb{R}^n$ is a \textbf{feasible direction} at  $\mathbf{x}$ if there exists $\varepsilon > 0$ such that
\begin{equation*}
    \mathbf{x} + \varepsilon \mathbf{s} \geq 0.
\end{equation*}
\end{definition}
Notice that if $\mathbf{s}$ is a feasible direction at $\mathbf{x}$ and $s_i<0$, then we must have $x_i>0$.
The following theorem characterizes the solutions of the convex problem \eqref{eq:NNLS}. For a proof see e.g. \cite{LH95}.
\begin{theorem}{(Karush-Kuhn-Tucker conditions for NNLS)} \label{thm:KKT_NNLS}
A point $\mathbf{x}^{\star} \in \mathbb{R}^n$ is a solution of problem \eqref{eq:NNLS} if and only if there exists  $\mathbf{w}^{\star} \in \mathbb{R}^n$ and a partition $Z^{\star} \cup P^{\star} = \left\{1,\dots,n \right\}$ such that
\begin{align}
    &\mathbf{w}^{\star} = A^T( \mathbf{b} - A \mathbf{x}^{\star} ),\\
    &x^{\star}_i = 0, \ i \in Z^{\star}, \quad x^{\star}_i > 0, \ i \in P^{\star}, \label{eq:KKT_1} \\
    &w^{\star}_i \leq 0, \ i \in Z^{\star}, \quad w^{\star}_i = 0, \ i \in P^{\star}. \label{eq:KKT_2}
\end{align}
\end{theorem}
Notice that equations (\ref{eq:KKT_1}-\ref{eq:KKT_2}) imply
\begin{equation}
    x^{\star}_i w^{\star}_i = 0,\ i=1,\dots,n,
\end{equation}
which are the well known complementarity conditions. The points $\mathbf{x}$ and $\mathbf{w}$ are referred as primal and dual variables, respectively. Notice that the primal and dual variables are related as follows $\mathbf{w} = A^T\mathbf{r} = -\nabla \phi ( \mathbf{x})$, i.e. the dual vector $\mathbf{w}$ is the opposite of the gradient direction, i.e. the steepest descend direction, at $\mathbf{x}$.
 
Active set methods are in the family of descend algorithms, that is they look for the solution by decreasing the objective function value, in particular at each iteration the new solution is a feasible point in a feasible descend direction with respect to the current solution. 
Recall that the $i$-th constraint is said to be active if it holds with equality, here $x_i=0$, it is said to be passive when it holds with strict inequality, here $x_i>0$, otherwise it is violated, namely $x_i<0$.

\subsection{The Lawson-Hanson algorithm}

The first algorithm to solve \eqref{eq:NNLS} is due to Lawson and Hanson \cite{LH95}, and it is presented in Algorithm \ref{alg:LH}, in an efficient formulation that uses Householder and Givens orthogonal transformations, which are then detailed here below and in sec. \ref{sed:modifiedQR}. It is an active set method and a particular case of the algorithm introduced in \cite{Sto71} for the least squares problem with linear inequality constraints.
\begin{algorithm}[ht]
\footnotesize
\caption{LH$(A, \mathbf{b})$}
\label{alg:LH}
\begin{algorithmic}[1]
\State $ {P}_0 = \emptyset$, $ {Z}_0=\left\{ 1,\dots,M \right\}$, $\mathbf{x}\myPart{(0)}{} = 0$, $\mathbf{w}\myPart{(0)}{} = A^T \mathbf{b}$, $s=0$
\While{$ {Z}_s \neq \emptyset$ and $\max(\mathbf{w}\myPart{(s)}{Z})>0$}  \Comment{outer loop}
\State $j_s \in \argmax_i {w}\myPart{(s)}{i} $ \label{alg:pivoting_step}
\State set $ {P}_{s+1} =  {P}_{s} \cup \left\{ j_s \right\}$, $ {Z}_{s+1} =  {Z}_{s} \setminus \left\{ j_s \right\}$
\If{$s=0$}
\State {set $R^{(1)}$ as the Householder reflector related to $\mathbf{a}_{j_s}$ and compute $\mathbf{b}^{(1)}$}
\Else
\State {use rank-1 modification to update $R^{(s+1)}$ and update $\mathbf{b}^{(s+1)}$}
\EndIf
\State compute $\mathbf{y}\myPart{(s+1)}{P}$ as the solution of $R\myPart{(s+1)}{} \mathbf{y} = \mathbf{b}^{(s+1)}$, set $\mathbf{y}\myPart{(s+1)}{Z} = 0$ \label{alg:LH_step_update_LS}
\While{$\min\left( \mathbf{y}\myPart{(s+1)}{P} \right) \leq 0$}  \Comment{inner loop}
\State $s = s+1$
\State $\alpha = \min \left\{ {x\myPart{(s-1)}{i}}\left({x\myPart{(s-1)}{i}-y\myPart{(s)}{i}}\right)^{-1}:\ i \in P_s,\ y\myPart{(s)}{i}\leq 0 \right\}$
\State $\mathbf{x}\myPart{(s)}{} = \mathbf{x}\myPart{(s-1)}{} + \alpha \left( \mathbf{y}\myPart{(s)}{}- \mathbf{x}\myPart{(s-1)}{}\right)$
\State $ {Q}=  \left\{i\ :\ x\myPart{(s)}{i}\leq 0 \right\}$
\State $ {P}_{s+1} =  {P}_s \setminus Q$ 
\State $ {Z}_{s+1} =  {Z}_s \cup Q $ 
\State $ R^{(s+1)} = R^{(s)}$
\For{$j \in Q$}
\State delete column corresponding to index $j$ from $R^{(s+1)}$
\State reduce $R^{(s+1)}$ to triangular form by Givens transformation and update $\mathbf{b}^{(s+1)}$
\EndFor 
\State compute $\mathbf{y}\myPart{(s+1)}{P}$ as the solution of $R\myPart{(s+1)}{} \mathbf{y} = \mathbf{b}^{(s+1)}$, set $\mathbf{y}\myPart{(s+1)}{Z} = 0$ \label{alg:LH_step_downdate_LS}
\EndWhile
\State $\mathbf{x}\myPart{(s+1)}{}=\mathbf{y}\myPart{(s+1)}{}$ \label{alg:LH_step_new_x}
\State $\mathbf{w}\myPart{(s+1)}{Z} = \left(A\myPart{(s+1)}{Z}\right)^T\mathbf{b}$, $\mathbf{w}\myPart{(s+1)}{P} = 0$
\State $s = s+1$
\EndWhile \\
\Return $ {P}^{\star} =  {P}_s$, $ {Z}^{\star} =  {Z}_s$, $\mathbf{x}^{\star} = \mathbf{x}\myPart{(s)}{}$
\end{algorithmic}
\end{algorithm}

Let us now describe the main features of the Lawson-Hanson algorithm.
Let $ {P}_s$ denote the passive set and $ {Z}_s$ denote the active set at $k$-th algorithmic step, with cardinality respectively $n_s$ and $n-n_s$. Let us define the following submatrices  
\begin{align*}
    A\myPart{(s)}{P} &= \left(\mathbf{a}_{i_1} \dots \mathbf{a}_{i_{n_s}} \right), && i_1, \dots i_{n_s} \in  {P}_s, \\
    {A}\myPart{(s)}{Z} &= \left(\mathbf{a}_{j_1} \dots \mathbf{a}_{j_{n-n_s}} \right), && j_1, \dots j_{n-n_s} \in  {Z}_s. \\
\end{align*}
We define the $k$-th iterate as $\mathbf{x}\myPart{(s)}{}=\left(x\myPart{(s)}{1}, \dots, x\myPart{(s)}{n} \right)^T \in \mathbb{R}^{n}$. With an analogous notation to the one introduced above for matrices, we have
\begin{align*}
    \mathbf{x}\myPart{(s)}{P} &= \left(x\myPart{(s)}{i_1}, \dots, x\myPart{(s)}{i_{n_s}} \right)^T \in \mathbb{R}^{n_s}, && i_1, \dots, i_{n_s} \in  {P}_s, \\
    {\mathbf{x}}\myPart{(s)}{Z} &= \left(x\myPart{(s)}{j_1}, \dots, x\myPart{(s)}{j_{n-n_s}} \right)^T \in \mathbb{R}^{n-n_s}, && j_1, \dots, j_{n-n_s} \in  {Z}_s.\\
\end{align*}
Since the iterates produced by Lawson-Hanson algorithm do not leave the feasible region, we always have ${\mathbf{x}}\myPart{(s)}{Z} = 0$. The vector $\mathbf{y}\myPart{(s)}{P}$ is the least squares solution of the following unconstrained subproblem
\begin{equation}
    \min_{\mathbf{y}} \left\Vert A\myPart{(s)}{P} \mathbf{y}-\mathbf{b}\right\Vert, \label{eq:LH_LS_subproblem}
\end{equation}
and $\mathbf{y}\myPart{(s)}{Z}=0$. Notice that these vectors change size at each iteration.

With the notation introduced, we have
\begin{align*}
    \mathbf{r}\myPart{(s)}{} &=  \mathbf{b} - A \mathbf{x}\myPart{(s)}{} = \mathbf{b} -  A\myPart{(s)}{P}  \mathbf{x}\myPart{(s)}{P}, \\
    \mathbf{w}\myPart{(s)}{} &= A^T \left( \mathbf{b} -A \mathbf{x}\myPart{(s)}{}  \right)= A^T \left( \mathbf{b} - A\myPart{(s)}{P}  \mathbf{x}\myPart{(s)}{P} \right)= A^T \mathbf{r}\myPart{(s)}{},\\
    \mathbf{w}\myPart{(s)}{Z} &=  \left(A\myPart{(s)}{Z}\right)^T \left( \mathbf{b} - A\myPart{(s)}{P}  \mathbf{x}\myPart{(s)}{P} \right) =  \left(A\myPart{(s)}{Z}\right)^T  \mathbf{r}\myPart{(s)}{}, \\
    \mathbf{w}\myPart{(s)}{P} &=  \left(A\myPart{(s)}{P}\right)^T \left( \mathbf{b} - A\myPart{(s)}{P} \mathbf{x}\myPart{(s)}{P}  \right) = 0.
\end{align*}
where the last identity is a consequence of normal equations for \eqref{eq:LH_LS_subproblem}, since at the end of the outer loop we have $\mathbf{x}\myPart{(s)}{P} = \mathbf{y}\myPart{(s)}{P}$.
Clearly we have $\mathbf{w}\myPart{(s)}{} \cdot \mathbf{x}\myPart{(s)}{} = 0$ for every $k$. As shown in \cite{LH95}, Algorithm \ref{alg:LH} terminates in a finite number of steps and $\mathbf{x}\myPart{(s)}{}, \mathbf{w}\myPart{(s)}{}, P_s, Z_s$ satisfy KKT conditions (\ref{eq:KKT_1}-\ref{eq:KKT_2}) on termination.

Since the columns of $A\myPart{(s)}{P}$ are linearly independent, the vector $\mathbf{y}\myPart{(s)}{P}$ is the unique solution of the linear system $R\myPart{(s)}{} \mathbf{y}\myPart{(s)}{P} = \mathbf{b}\myPart{(s)}{}$, where $R\myPart{(s)}{}$ is the square upper triangular matrix of order $n_s = |P_s|$ in the QR decomposition of $A\myPart{(s)}{P}$, that is 
\begin{equation*}
    A\myPart{(s)}{P} = Q\myPart{(s)}{}\left(\begin{array}{c}
         R\myPart{(s)}{} \\
         0
    \end{array}\right).
\end{equation*}
Here the passive set $P_s$ is an ordered set (the order in which the column indices have been inserted) and therefore  $P_s$ also imposes the column ordering in QR computation. Similarly, the updated right-hand side is given by
\begin{equation*}
    \mathbf{b}\myPart{(s)}{} = \left(Q\myPart{(s)}{}\right)^T\mathbf{b}.
\end{equation*}
With respect to the original matrix $A$, we obtain the following partial QR decomposition with column pivoting
\begin{equation*}
    \left( Q\myPart{(s)}{} \right)^T A\ \Pi\myPart{(s)}{} = \left(\begin{array}{cc}
         R\myPart{(s)}{} & B\myPart{(s)}{} \\
         0 & C\myPart{(s)}{}
    \end{array}\right),
\end{equation*}
where $\Pi\myPart{(s)}{}$ is the permutation matrix that permuted the elements in $P_s$ in the leftmost positions. In an actual implementation \cite{LH95}, the columns are not physically permuted, rather a permutation vector is used to keep track of the column ordering. Here, $B\myPart{(s)}{}$ has size $n_s \times (n-n_s)$ and $C\myPart{(s)}{}$ has size $(m-n_s) \times (n-n_s)$.

Let us spend the rest of this section by giving some details about the implementation of Algorithm \ref{alg:LH}.
Notice that the matrix $R\myPart{(s)}{}$ is never computed from scratch, instead it is modified by means of rank-$1$ updates or downdated by means of Givens rotations.
Let us briefly detail these strategies. At the beginning of the outer loop, the new passive set is incremented of exactly one index, therefore $n_{s+1} = |P_{s+1}| = n_{s}+1$. Then, the corresponding Householder vector is computed and applied to the right-hand side and to the trailing matrix, i.e. 
\begin{equation}
    \left(I- \mathbf{v}_s^T\mathbf{v}_s\right)C\myPart{(s)}{} = \left( \begin{array}{cc}
        r_{n_s+1} & b_{n_s+1,2}\ \dots\ b_{n_s+1,n-n_s}  \\
        0 & C\myPart{(s+1)}{}
    \end{array}\right),
\end{equation}
where we supposed that the selected column was in the leading position. The new block $C\myPart{(s+1)}{}$ has size $(m-n_s-1) \times (n-n_s-1)$. The partial factorization is updated as follows
\begin{equation*}
    R\myPart{(s+1)}{} = \left( \begin{array}{cc}
        R\myPart{(s)}{} & \mathbf{b}_1^{(s)} \\
         & r_{n_s+1}
    \end{array}\right), \quad
    B\myPart{(s+1)}{} = \left( \begin{array}{ccc}
        \mathbf{b}_2^{(s)} & \dots & \mathbf{b}_{n-n_s}^{(s)} \\
         b_{n_s+1,2} & \dots & b_{n_s+1,n-n_s}
    \end{array}\right).
\end{equation*}
On the other hand, when a column is removed from the passive set, the R factor is almost an Hessenberg matrix, i.e. it has some nonzero entries in its first lower diagonal, and it has to be reduced to triangular form. This can be done by computing a sequence of Givens rotations, which have to be applied to the whole matrix $A$ and to the right-hand side. In this way, the orthogonal factor $Q\myPart{(s)}{}$ is never computed nor stored. We detail the modified QR decomposition procedure in Section \ref{sed:modifiedQR}.

\subsection{The Lawson-Hanson algorithm with Deviation Maximization}

In this section we propose a new active set method in which multiple columns are manipulated at a time and for which finite termination can be established in exact arithmetic. Before coming to the description of the main algorithm, let us discuss a strategy, for columns selection, based upon the deviation maximization \cite{DMa21}.

Recall that for an $m\times k$ matrix $C = (\mathbf{c}_{1}\ \dots \ \mathbf{c}_{k})$ whose columns $\mathbf{c}_{j}$ are non-null, the \textit{correlation matrix} $\Theta$ has entries
\begin{equation}
    \theta_{ij} = \frac{\mathbf{c}_i^T \mathbf{c}_j}{\Vert \mathbf{c}_{i} \Vert\Vert \mathbf{c}_{j} \Vert}, \quad 1 \leq i,j \leq k.
    \label{eq:def_cosine_matrix}
\end{equation}
In particular, we have $\Theta = \left( C D^{-1} \right)^T C D^{-1} = D^{-1} C^T C D^{-1}$, where $D$ is the diagonal matrix with entries $d_{i} = \Vert \mathbf{c}_{i} \Vert\ ,\ 1\leq i \leq k$. It is immediate to see that $\Theta$ is symmetric positive semidefinite, it has only ones on the diagonal, and its entries range from $-1$ to $1$. 
Notice that $\theta_{ij}$ is the cosine of $\alpha_{ij} = \alpha(\mathbf{c}_i,\mathbf{c}_j) \in [0,\pi)$, the angle (modulo $\pi$) between $\mathbf{c}_i$ and $\mathbf{c}_j$. 
In order to emphasize this geometric interpretation, from now on we refer to $\Theta$ as the \textit{cosine matrix}. 

Algorithm \ref{alg:DM} presents the deviation maximization algorithm for column selection. Here, the procedure for computing the candidate set $I$ has been adapted to fit the Lawson-Hanson algorithm and it is slightly different from the one presented in \cite{DMa21}.

\begin{algorithm}[h]
\footnotesize
\caption{Deviation Maximization: $[J]$=DM$(A, \mathbf{u}_1, \mathbf{u}_2, \tau_1, \tau_2, \delta, k_{\max})$}
\label{alg:DM}
\begin{algorithmic}[1]
\State ${J} = \left\{j\ :\ j \in \argmax \mathbf{u}_1 \right\}$ \label{alg:DM_step_initializeJ}
\State ${I} = \left\{i\ :\ \mathbf{u}_{l}(i) \geq \tau_l \max \mathbf{u}_l,\ i\neq j, l=1,2 \right\} $ \label{alg:DM_step_candidate}
\If {$|I|>k_{\max}$}
\State $I= \left\{ i_s \in I \ :\ k=1,\dots,k_{\max} \right\}$   
\EndIf 
\State compute the cosine matrix $\Theta$ associated to $[A]_{:,{I}}$ 
\For{$i \in {I}$}
\If{$|\theta_{i,j}| < \delta, \forall j \in {J}$}
\State ${J} = {J} \cup \left\{ i \right\}$
\EndIf
\EndFor \\
\Return $J$
\end{algorithmic}
\end{algorithm}

The resulting procedure, namely the Lawson-Hanson algorithm with Deviation Maximization (LHDM), is presented in Algorithm \ref{alg:LHDM}, and it terminates in a finite number of steps as shown in Theorem \ref{thm:main_LHDM_thm} below.

\begin{theorem} \label{thm:sigma_min_of_R}
Let $A\myPart{(s)}{P}$ be a full column rank matrix, and let $R^{(s+1)}$ be the R factor of the QR decomposition of $A\myPart{(s+1)}{P} = \left(A\myPart{(s)}{P}\ A\myPart{(s)}{J}\right)$, where $A\myPart{(s)}{J}$ is the submatrix of $A$ that is obtained by taking columns with indices in the set ${J_s}$, which is chosen by the deviation maximization with $0 < \tau_1, \tau_2 < 1$, $k_{\max} > 1$, and
\begin{equation}
    \delta \leq \frac{{\tau}_2}{k_{\max}-1}. \label{eq:delta_static}
\end{equation}
Then
\begin{equation}
    \bar{\sigma}^{(s+1)} \geq \sigma_{n_{s+1}}(A)  \frac{\bar{\sigma}^{(s)}}{\sigma_1(A)} \frac{1}{\sqrt{2(n-n_{s+1})n_{s+1}}}  \frac{\tau_2 \sqrt{1-\tau_2}}{k_{s}^2 n_s},
\end{equation}
where $k_{s}=|J_s|$ and $\bar{\sigma}^{(s+1)}$ is the smallest singular value of $R^{(s+1)}$, which is in particular nonsingular.
\end{theorem}
\begin{proof}
Apply Theorem 2 of \cite{DMa21} to conclude.
\end{proof}

It is known, it has been showed in \cite{DMa21}, that even if the matrix $A\myPart{(s)}{P}$ is well conditioned, there could be a potentially dramatic increase in the condition number of  $A\myPart{(s+1)}{P}$ when using the deviation maximization strategy, as well as for the other column selection strategy like the standard column pivoting. A famous example is the Kahan matrix, however it is a very unlikely occurrence in practice.

The choice \eqref{eq:delta_static} is not the only one possible: in fact, 
the values of $\delta$ and $\tau_2$ can be set independently, as shown in \cite{DMa21}.

\begin{algorithm}[h!]
\footnotesize
\caption{LHDM$(A, \mathbf{b}, \tau_1, \tau_2, \tau_c, k_{\max})$}
\label{alg:LHDM}
\begin{algorithmic}[1]
\State $ {P}_0 = \emptyset$, $ {Z}_0=\left\{ 1,\dots,M \right\}$, $\mathbf{x}\myPart{(0)}{} = 0$, $\mathbf{w}\myPart{(0)}{} = A^T \mathbf{b}$, $\mathbf{y}\myPart{(0)}{} = 0$, $s=0$
\State $R^{(0)} = 0$, $B^{(0)} = 0$, $C^{(0)} = A$
\State initialize $\mathbf{u}^{(0)}$ as the vector of column norms of $A$ 
\While{$ {Z}_s \neq \emptyset$ and $\max\left(\mathbf{w}\myPart{(s)}{}\right)>0$} \Comment{outer loop}
\State $[J_s]$=DM$(C^{(s)}, \mathbf{w}^{(s)}, \mathbf{u}^{(s)}, \tau_1, \tau_2, \delta, k_{\max})$
\State $ {P}_{s+1} =  {P}_{s} \cup J_{s} $
\State $ {Z}_{s+1} =  {Z}_{s} \setminus J_{s} $
\If{$s=0$}
\State{ set $R^{(1)}$ as the R factor of the QR decomposition of $A\myPart{(1)}{J}$ and compute $\mathbf{b}^{(1)}$, $B^{(1)}$ and $C^{(1)}$}
\Else
\State { use rank-$k$ update to get $R^{(s+1)}$, $B^{(s+1)}$, $C^{(s+1)}$, and update $\mathbf{b}^{(s+1)}$}
\EndIf
\State compute $\mathbf{y}\myPart{(s+1)}{P}$ as the solution of $R\myPart{(s+1)}{} \mathbf{y} = \mathbf{b}^{(s+1)}$, set $\mathbf{y}\myPart{(s+1)}{Z} = 0$ \label{alg:LHDM_step_update_LS}
\While{$\min\left( \mathbf{y}\myPart{(s+1)}{J} \right) \leq 0$}  \Comment{ensure a feasible descend direction}
\State delete the last element from $J_s$
\State $ {P}_{s+1} =  {P}_{s} \cup J_{s} $
\State $ {Z}_{s+1} =  {Z}_{s} \setminus J_{s} $
\State delete the last column of  $R\myPart{(s+1)}{}$ and update $\mathbf{b}^{(s+1)}$, $B^{(s+1)}$, $C^{(s+1)}$,
\State compute $\mathbf{y}\myPart{(s+1)}{P}$ as the solution of $R\myPart{(s+1)}{} \mathbf{y} = \mathbf{b}^{(s+1)}$, set $\mathbf{y}\myPart{(s+1)}{Z} = 0$
\label{alg:LHDM_step_update_LS_new}
\EndWhile
\While{$\min\left( \mathbf{y}\myPart{(s+1)}{P} \right) \leq 0$} \Comment{inner loop}
\State $s = s+1$
\State $\alpha = \min \left\{ {x\myPart{(s-1)}{i}}\left({x\myPart{(s-1)}{i}-y\myPart{(s)}{i}}\right)^{-1}:\ i \in P_s,\ y\myPart{(s)}{i}\leq 0 \right\}$
\State $\mathbf{x}\myPart{(s)}{} = \mathbf{x}\myPart{(s-1)}{} + \alpha \left( \mathbf{y}\myPart{(s)}{}- \mathbf{x}\myPart{(s-1)}{}\right)$
\State $ {Q}=  \left\{i\ :\ x\myPart{(s)}{i}\leq 0 \right\}$
\State $ {P}_{s+1} =  {P}_s \setminus Q$ 
\State $ {Z}_{s+1} =  {Z}_s \cup Q $ 
\State $ R^{(s+1)} = R^{(s)}$
\For{$j \in Q$}
\State delete column corresponding to index $j$ from $R^{(s+1)}$
\State reduce $R^{(s+1)}$ to triangular form by Givens transformation and update $\mathbf{b}^{(s+1)}$
\EndFor 
\State compute $\mathbf{y}\myPart{(s+1)}{P}$ as the solution of $R\myPart{(s+1)}{} \mathbf{y} = \mathbf{b}^{(s+1)}$, set $\mathbf{y}\myPart{(s+1)}{Z} = 0$ \label{alg:LHDM_step_downdate_LS}
\EndWhile
\State $\mathbf{x}\myPart{(s+1)}{}=\mathbf{y}\myPart{(s+1)}{}$ \label{alg:LHDM_step_new_x}
\State $\mathbf{w}\myPart{(s+1)}{} = A^T\left(\mathbf{b}-A\mathbf{x}\myPart{(s+1)}{}\right)$
\State update $\mathbf{u}^{(s)}$ as the vector of column norms of $C^{(s+1)}$
\State $s = s+1$
\EndWhile \\
\Return $ {P}^{\star} =  {P}_s$, $ {Z}^{\star} =  {Z}_s$, $\mathbf{x}^{\star} = \mathbf{x}\myPart{(s)}{}$
\end{algorithmic}
\end{algorithm}

Let us now prove the main result of this work. 

\begin{theorem} \label{thm:main_LHDM_thm} Consider the NNLS problem
\begin{equation*}
    \min \Vert A \mathbf{x} -\mathbf{b} \Vert^2_2, \qquad \text{s.t. } {\mathbf{x}\geq 0}.
\end{equation*}
$k_{\max} > 1$, Algorithm \ref{alg:LHDM} terminates in a finite number of steps. 
On termination, $\mathbf{x}$ will be a solution vector and $\mathbf{w}$ will be a dual solution vector.
\end{theorem}
\begin{proof}
In order to ensure finite convergence of Algorithm \ref{alg:LHDM}, we have to ensure the following
\begin{enumerate}
    \item termination of the outer loop, by showing that $\mathbf{s}\myPart{(s+1)}{}= \mathbf{y}\myPart{(s+1)}{} - \mathbf{x}\myPart{(s)}{}$ is a feasible descent direction;
    \item termination of the inner loop with a nonempty passive set. 
\end{enumerate}
Let us address the first task. Theorem \ref{thm:sigma_min_of_R} ensures that $A\myPart{(s+1)}{P}= \left(A\myPart{(s)}{P}\ A\myPart{(s)}{J}\right)$ 
has numerically linearly independent columns. 
Let us claim the following statement: if $\mathbf{y}\myPart{(s+1)}{J}>0$, then $\mathbf{s}\myPart{(s+1)}{}= \mathbf{y}\myPart{(s+1)}{} - \mathbf{x}\myPart{(s)}{}$ is a feasible descend direction at $\mathbf{x}\myPart{(s)}{}$. Then, if $\mathbf{y}\myPart{(s+1)}{j}\leq0$ for some $j \in J_s$, the last column of $A\myPart{(s)}{J}$ is dropped and the corresponding index $j_{k_s}$ is moved back to the active set $Z_{s+1}$. We then look for a new value of $\mathbf{y}\myPart{(s+1)}{}$ such that 
\begin{equation*}
    R\myPart{(s+1)}{1} \mathbf{y}\myPart{(s+1)}{P} = \mathbf{b}\myPart{(s+1)}{1},
\end{equation*}
where $R\myPart{(s+1)}{1}$ is equal to $R\myPart{(s+1)}{}$ where the last column is missing, and the right hand side $\mathbf{b}\myPart{(s+1)}{1}$ has been properly updated, e.g. by applying the Givens rotation that zeros out the last diagonal element of $R\myPart{(s+1)}{}$, corresponding to the deleted column. We continue dropping a column index at a time from ${J_s}$ until $\mathbf{y}\myPart{(s+1)}{j}>0$ for all indices $j$ left in $J_s$. It should be noticed that this procedure terminates at most when only the first column of $A\myPart{(s)}{J}$ is left, that is the same column index that would be selected by the standard Lawson-Hanson algorithm, i.e. it would guarantee the positivity of the inserted component.

For what concerns the second task, the inner loop terminates each time in at most $n_{s+1}-1$ steps. Set $s=s+1$ and notice that the first time the inner loop is entered, we have $\mathbf{y}\myPart{(s)}{J}>0$, implying $\mathbf{x}\myPart{(s)}{J} = \alpha \mathbf{y}\myPart{(s)}{J} >0$ and thus $J_s$ and $Q$ do not intersect. Therefore, indices in $J_s$ are not removed before computing $\mathbf{y}\myPart{(s+1)}{}$. Moreover, at least the first index $j_1$ of $J_s$ is left, because of Lemma 23.17 in \cite{LH95}, and thus the inner loop must terminate at most in $n_{s+1}-1$ step with a nonempty passive set $P_{s+1} \supseteq \{ j_1 \}$.

Let us now prove the claim, i.e. if $\mathbf{y}\myPart{(s+1)}{J}>0$, then $\mathbf{s}\myPart{(s+1)}{}= \mathbf{y}\myPart{(s+1)}{} - \mathbf{x}\myPart{(s)}{}$ is a feasible descend direction at $\mathbf{x}\myPart{(s)}{}$.
Consider the least squares solution $\mathbf{y}$ computed at step \ref{alg:LHDM_step_update_LS} of Algorithm \ref{alg:LH} with deviation maximization pivoting, that is $\mathbf{y}$ solves
\begin{equation}
     \mathbf{y}_P = \argmin_{\mathbf{y}}  \left\Vert A\myPart{(s)}{P} \mathbf{y}-\mathbf{b}\right\Vert \quad \mathbf{y}_Z=0. 
     \label{eq:min_LQ_b}
\end{equation}
where $Z = Z_{s}$ and $P = P_{s}$. The additional inner loop ensures that $y\myPart{(s)}{j}>0$ for $j \in J_s$ on termination. This, together with $\mathbf{x}\myPart{(s-1)}{}\geq0$, means that $\mathbf{s}\myPart{(s)}{} := \mathbf{y}\myPart{(s)}{} -\mathbf{x}\myPart{(s-1)}{}$ is a feasible direction. Let us now prove that $\mathbf{s}\myPart{(s)}{}$ is also a descend direction.
Said $\mathbf{r}\myPart{(s-1)}{} = \mathbf{b} - A \mathbf{x}\myPart{(s-1)}{}$, let us define $\hat{\mathbf{z}}$ as follows
\begin{equation}
        \hat{\mathbf{z}} = \argmin_{\mathbf{z}_Z=0} \| A \mathbf{z}-\mathbf{r}\myPart{(s-1)}{}\| = \argmin_{\mathbf{z}_Z=0}  \|A (\mathbf{x}\myPart{(s-1)}{} + \mathbf{z})-\mathbf{b}\|, \label{eq:min_LQ_r}
\end{equation}
where $Z = Z_{s}$ and $P = P_{s}$. Notice that $\hat{\mathbf{z}}$ can be obtained as
\begin{equation}
     \hat{\mathbf{z}}_P = \argmin_{\mathbf{z}}  \|A\myPart{(s)}{P} \mathbf{z}-\mathbf{r}\myPart{(s-1)}{}\|, \quad \hat{\mathbf{z}}_Z=0.
\end{equation}
Then we have
\begin{equation}
\begin{aligned}
    \|A \hat{\mathbf{z}}-\mathbf{r}\myPart{(s-1)}{}\| &= \|A  \hat{\mathbf{z}} -(\mathbf{b} - A\mathbf{x}\myPart{(s-1)}{})\| = \|A  (\hat{\mathbf{z}} + \mathbf{x}\myPart{(s-1)}{} ) - \mathbf{b} \| \\
    & \geq  \|A \mathbf{y}\myPart{(s)}{}-\mathbf{b}\|,
\end{aligned}    
\end{equation}
since $ \hat{\mathbf{z}} + \mathbf{x}\myPart{(s-1)}{} $ vanishes on $Z_{s}$ and the minimum is reached at  $\mathbf{y}\myPart{(s)}{}$. On the other hand
\begin{equation}
\begin{aligned}
     \|A \mathbf{y}\myPart{(s)}{}-\mathbf{b}\| &=  \|A \mathbf{y}\myPart{(s)}{} - A\mathbf{x}\myPart{(s-1)}{} -\mathbf{b} + A\mathbf{x}\myPart{(s-1)}{}\| =  
    \|A (\mathbf{y}\myPart{(s)}{} - \mathbf{x}\myPart{(s-1)}{}) -\mathbf{r}\myPart{(s-1)}{}\| \\
    &\geq \|A \hat{\mathbf{z}}-\mathbf{r}\myPart{(s-1)}{}\|,
\end{aligned}    
 \label{eq:chain_ineq}
\end{equation}
since $ \mathbf{y}\myPart{(s)}{} - \mathbf{x}\myPart{(s-1)}{}$ vanishes on $Z_{s}$ and the minimum is reached at $\hat{\mathbf{z}}$. Therefore 
\begin{equation*}
     \|A \hat{\mathbf{z}}-\mathbf{r}\myPart{(s-1)}{}\| = \|A \mathbf{y}\myPart{(s)}{}-\mathbf{b}\|,
\end{equation*}
and the inequality chain \eqref{eq:chain_ineq} collapses in an equality chain.
As a consequence of the uniqueness of minimizers of \eqref{eq:min_LQ_b} and \eqref{eq:min_LQ_r}, we have 
\begin{equation}
    \hat{\mathbf{z}} = \mathbf{s}\myPart{(s)}{} = \mathbf{y}\myPart{(s)}{} -\mathbf{x}\myPart{(s-1)}{}. \label{eq:z_equal_s}
\end{equation}
Consider the function
\begin{equation*}
    \varphi(\varepsilon) = \frac{1}{2} \Vert A (\mathbf{x}\myPart{(s-1)}{} + \varepsilon \mathbf{s}\myPart{(s)}{}) - \mathbf{b} \Vert^2.
\end{equation*}
We have 
\begin{equation}
  \begin{aligned}
\varphi'(0) &=  (\mathbf{s}\myPart{(s)}{})^T A^T (A\mathbf{x}\myPart{(s-1)}{} - \mathbf{b}) = -(\mathbf{s}\myPart{(s)}{})^T \mathbf{w}\myPart{(s-1)}{} = -(\mathbf{y}\myPart{(s)}{})^T \mathbf{w}\myPart{(s-1)}{} \\
& = \sum_{i \in J}y\myPart{(s)}{i}(-w\myPart{(s-1)}{i}) < 0, \label{eq:strict_ineq}    
\end{aligned}  
\end{equation}
since $w\myPart{(s-1)}{i}>0$ for all ${i \in J}$ and $y\myPart{(s)}{i} > 0$ for all ${i \in J}$.
Therefore there exists $\varepsilon>0$ such that
\begin{equation*}
    \Vert A (\mathbf{x}\myPart{(s-1)}{} + \varepsilon \mathbf{s}\myPart{(s)}{}) - \mathbf{b} \Vert^2 < \Vert A \mathbf{x}\myPart{(s-1)}{} - \mathbf{b} \Vert^2 = \phi (\mathbf{x}\myPart{(s-1)}{}).
 \end{equation*}
By definition \eqref{eq:min_LQ_r} and relation \eqref{eq:z_equal_s}, we have
\begin{equation}
    \|  A (\mathbf{x}\myPart{(s-1)}{} + \mathbf{s}\myPart{(s)}{}) - \mathbf{b} \|^2 \leq  \Vert A (\mathbf{x}\myPart{(s-1)}{} + \varepsilon \mathbf{s}\myPart{(s)}{}) - \mathbf{b} \Vert^2 < \Vert A \mathbf{x}\myPart{(s-1)}{} - \mathbf{b} \Vert^2. \label{eq:non_opt}
\end{equation}
Thus, we take $ \varepsilon =1$ (i.e. $\mathbf{x}\myPart{(s)}{} = \mathbf{y}\myPart{(s)}{}$), if $\mathbf{y}\myPart{(s)}{}$ is feasible, otherwise the algorithm enters the inner loop and we choose a smaller step length $\alpha$, more precisely the largest possible to keep the new iterate within the feasible region. It's easy to check that
\begin{align*}
    \Vert A (\mathbf{x}\myPart{(s-1)}{} + \alpha \mathbf{s}\myPart{(s)}{}) - \mathbf{b} \Vert^2 &= \min_{\varepsilon} \left\{ \Vert A (\mathbf{x}\myPart{(s-1)}{} + \varepsilon \mathbf{s}\myPart{(s)}{}) - \mathbf{b} \Vert^2:\ \mathbf{x}\myPart{(s-1)}{} + \varepsilon \mathbf{s}\myPart{(s)}{} \geq 0 \right\} \\
    & \leq \Vert A \mathbf{x}\myPart{(s-1)}{} - \mathbf{b} \Vert^2.
\end{align*}
We deduce that the objective function is always non increasing. Equation \eqref{eq:non_opt} also implies that every time step \ref{alg:LHDM_step_new_x} is reached, we have a strictly smaller value of the objective function. Hence at every iteration of the outer loop, the current objective value is strictly smaller than the preceding one.
Since the solution solves a least squares subproblem like those of steps \ref{alg:LHDM_step_update_LS} and \ref{alg:LHDM_step_downdate_LS}, the corresponding passive set must be different from the preceding one. As the number of possible combinations of active and passive sets is finite, namely it is equal to $2^n$, the algorithm must terminate in at most $2^n$ outer loop repetitions. 
\end{proof}
The proposed strategy may look naive. In order to give an idea about the reasons why this should work well, let us prove the following.
\begin{theorem}
If $\tau_1 = 1$ and $\tau_2\leq 1/2$, then in the outer loop we have $\mathbf{y}\myPart{(s)}{J}>0$ for all $s>0$.
\end{theorem}
\begin{proof}
The proof is an extension of Lemma 23.17 in \cite{LH95}, in which we end up showing 
\begin{equation*}
     \mathbf{y}\myPart{(s)}{J}= (C^TC)^{-1}\mathbf{w}\myPart{(s+1)}{J},
\end{equation*}
where $C = A[n_{s}+1:m,n_{s}+k_s]$ and $\mathbf{w}\myPart{(s+1)}{J} = (\omega, \dots, \omega)$, where $\omega = \max \mathbf{w}\myPart{(s+1)}{}>0$, because $\tau_1 = 1$. The matrix $C^TC$ is $\tau_2$-scaled diagonally dominant because it has been selected by deviation maximization, with $\tau_2\leq 1/2$, hence its inverse is also diagonally dominant with a positive diagonal \cite{Rad16}, yielding   
\begin{equation*}
    \mathbf{y}\myPart{(s)}{J}>0.
\end{equation*}
\end{proof}

This result shows that it is possible to analytically prove that cycling is avoided without the need of an additional loop to ensure a feasible descend direction. However, $\tau_1 = 1$ may be a too demanding choice and the algorithm may loose in performance. 
In Section \ref{sec:numerical_experiments}, we give experimental evidence of the significant performance gain that this strategy can lead.

\section{Sparse recovery} \label{sec:sparse_recovery}
In this section we present the problem of sparse recovery and its connections with nonnegative least squares and with the deviation maximization technique. Sparse recovery is a fundamental problem in compressed sensing, signal denoising, statistical model selection and related fields, and it consists in finding the sparsest solution to an underdetermined system of equations, see e.g. \cite{DEl03, Don06, libro_Elad}.  
There is a wide literature in the field of signal processing, in which a matrix is usually referred as a dictionary and its columns are called atoms. Here, we abandon this terminology in favour of the classical linear algebra one. 

Given a vector $\mathbf{x} \in \mathbb{R}^n$, we define its support  ${S} = {S}(\mathbf{x}) = \left\{i :\ x_i \neq 0 \right\}$. 
We aim at identifying the solution $\mathbf{x}^{\star}$ of the linear system $A x = b$ with the sparsest support ${S}^{\star}$. Formally, we consider the optimization problem 
\begin{equation}
    \min \Vert \mathbf{x} \Vert_0, \qquad \text{s.t. } \ A \mathbf{x} = \mathbf{b}. \label{eq:minl0_problem}
\end{equation}
Recall the ``$\ell_0$-norm'' of a vector $\mathbf{x}$ is defined as
\begin{equation*}
    \Vert \mathbf{x} \Vert_0 =  \left\vert {S}(\mathbf{x}) \right\vert =
    \left\vert \left\{ i: x_i \neq 0\right\} \right\vert,
\end{equation*}
and it is not an actual norm because it does not hold that for any scalar $s$ we have $\Vert s \mathbf{x} \Vert = |s|\ \Vert \mathbf{x} \Vert$. Actually, problem \eqref{eq:minl0_problem} of combinatorial search is NP-hard, and in general we seek an approximate solution. Two well known approaches are: “greedy algorithms", that focus on the determination of the support and then a simple LS solution on it, and “relaxation methods", that focus on smoothing the $l_0$ norm and then adopt smooth optimization. A famous greedy algorithm is the Orthogonal-Matching-Pursuit (OMP), while a popular convex relaxation technique is the so-called \textit{Basis Pursuit} (BP) \cite{CDo94} problem
\begin{equation}
    \min \Vert \mathbf{x} \Vert_1, \qquad \text{s.t. } \ A \mathbf{x} = \mathbf{b}, \label{eq:minl1_problem}
\end{equation}
which is a convex optimization problem, since $\Vert \mathbf{x} \Vert_1$ is a convex function of $x$ and the solutions of $A \mathbf{x} = \mathbf{b}$ form a convex set. We will concentrate on this second approach. Indeed, the two problems \eqref{eq:minl0_problem} and \eqref{eq:minl1_problem} yield the same solution provided it is sparse enough. This result is known as $\ell_0-\ell_1$ equivalence, and it has been found empirically \cite{CSD01} and theoretically \cite{Donoho-Huo_2001, EBr02}.

\subsection{Exact recovery} \label{sec:exact_recovery}

There are different conditions under which we have $\ell_0-\ell_1$ equivalence, i.e. an exact recovery. In the literature, these conditions are usually found by showing the uniqueness of solution to problems \eqref{eq:minl0_problem} and \eqref{eq:minl1_problem} and their coincidence. As a consequence of these results, efficient algorithms to solve \eqref{eq:minl1_problem} can be used to find the sparsest solution to an underdetermined system of equations. 

\subsubsection{Uniqueness based on RIP}
Let us recall the following definition.
\begin{definition}
Consider a matrix $A\in \mathbb{R}^{m\times n}$, and suppose there exists $\delta_s \in (0,1)$ such that for every $m \times s$ submatrix $A_s$ of $A$ and for every $y\in \mathbb{R}^s$ we have
\begin{equation}
    (1-\delta_s) \Vert y \Vert_2^2 \leq \Vert A_s y \Vert_2^2  \leq (1-\delta_s) \Vert y \Vert_2^2. \label{eq:RIP}
\end{equation}
We say that $A$ has the $s$-order \textbf{Restricted Isometry Property} (RIP) with constant $\delta_s$.
\end{definition}
For a given $s$, one is usually interested in the smallest constant $\delta_s$ for which \eqref{eq:RIP} holds, which is referred as the  $s$-order restricted isometry constant.
 
A popular result \cite{CANDES2008589} states that if $\delta_{2s} < 1$, then problem \eqref{eq:minl0_problem} has a unique solution $\mathbf{x}$ with support size obeying $\Vert \mathbf{x} \Vert_0 \leq s$. Moreover, if $\delta_{2s} < \sqrt{2} -1$, then the solution to problem \eqref{eq:minl1_problem} also solves problem \eqref{eq:minl0_problem}. However, the problem of establishing whether a given matrix $A$ fulfills the $s$-order RIP is NP-hard in general \cite{TPf14}.

\subsubsection{Uniqueness based on ERC}
Let us state the following result from \cite{Tropp_2004}.
\begin{theorem}
Consider a linear system $A\mathbf{x} = \mathbf{b}$. If there exists a solution $\bar{\mathbf{x}}$ with support $\bar{{S}} = {S}(\bar{\mathbf{x}})$ such that
\begin{equation}
    \max_{i\not\in \bar{{S}}} \left\Vert A^{\dagger}_{\bar{{S}}} \mathbf{a}_i \right\Vert_1 < 1, 
    \label{eq:ERC}
\end{equation}
then $\bar{\mathbf{x}}$ is the unique solution to the minimum $\ell_0$ problem \eqref{eq:minl0_problem}, which can be recovered by solving  the minimum $\ell_1$ problem \eqref{eq:minl1_problem}.
\end{theorem}
We refer to \eqref{eq:ERC} as \textbf{Exact Recovery Condition} (ERC). This provides an easy sufficient optimality check for a given support ${S}$, but finding a support ${S}$ satisfying \eqref{eq:ERC} is a combinatorial problem.


\subsection{Sparse recovery by NNLS}
The nonnegativity constraint is known to naturally produce sparse solutions, see e.g. \cite{Elad2008_uniqueness_nonnegative, FKo14, WTa09, WXT11}. An important outcome of this body of work is that nonnegativity alone may attain a satisfactory sparse recovery. 
There are also important similarities. For example, \cite{FKo14} shows that the before mentioned Orthogonal-Matching-Pursuit (OMP) resembles the Lawson-Hanson algorithm except for the internal loop, which has indeed an evident benefit on nonnegative recovery. They note also that this internal loop acts as a shrinkage: it brings to zero the component which is relatively more negative and brings to positive with a reduced magnitude the other components that have become negative. So, it modifies all the components and, more importantly, reduces the $l_1$-norm of the solution.
In synthesis, the Lawson-Hanson algorithm finds monotonically the solution if the growing support is correct, otherwise enters in a inner loop that reduces the support and diminishes the $l_1$-norm of the solution.

Notice also that arbitrary signed sparse recovery is easily achievable. Given $\bar{\mathbf{x}} \in \mathbb{R}^n$, decompose it as  $\bar{\mathbf{x}} = \bar{\mathbf{x}}^+ - \bar{\mathbf{x}}^- $, where $\bar{\mathbf{x}}^+ \geq 0$ and $\bar{\mathbf{x}}^- \geq 0$. Then the solutions of the linear system $A\bar{\mathbf{x}} = \mathbf{b}$ can be attained as the solutions of the nonnegative least squares problem 
\begin{equation}
    \min_{\mathbf{x}} \left\Vert (A\ -A) {\mathbf{x}} - \mathbf{b} \right\Vert_2^2, \quad {\mathbf{x}} \geq 0, \label{eq:positivity_trick}
\end{equation}
where ${\mathbf{x}} = \left( \begin{array}{l}
     \bar{\mathbf{x}}^+ \\
     \bar{\mathbf{x}}^-
\end{array} \right) \in \mathbb{R}^{2n}$. This has been shown e.g. in \cite{FKo14} and it is sometimes referred as ``positivity trick''.

\subsubsection{Uniqueness for NNLS problem }
This topic has been treated in \cite{Elad2008_uniqueness_nonnegative, Slawski-Hein-2013, FKo14}.

Let us define the following class of matrices
\begin{equation}
    {\cal M}^+ = \left\{A\in \mathbb{R}^{m\times n} : \ \exists\ \mathbf{h}\in \mathbb{R}^m \text{ s.t. } A^T\mathbf{h} = \mathbf{w}>0 \right\}. \label{eq:M+}
\end{equation}
A necessary condition in order to have a unique solution to the NNLS problem \eqref{eq:NNLS} is that $A$ belongs to ${\cal M}^+$ \cite{WTa09}. 
\begin{definition}
The \textbf{coherence} or \textbf{mutual coherence} or \textbf{two-sided coherence} of a matrix $A$ is defined as
\begin{equation}
    \mu(A) = \max_{i < j} \vert \theta_{ij} \vert \left( = \max_{i > j} \vert \theta_{ij} \vert \right). \label{eq:mutual_coherence}
\end{equation}
\end{definition}

In \cite{Tropp_2004} it is demonstrated the following uniqueness result, under the necessary condition \eqref{eq:M+}.
\begin{theorem}
Consider a linear system $A\mathbf{x} = \mathbf{b}$ such that $A \in {\cal M}^+$. If there exists a nonnegative solution $\bar{\mathbf{x}}$ with $\Vert \bar{\mathbf{x}} \Vert_0 \leq s$ and
\begin{equation}
    \mu(A) < \frac{1}{2s-1}, 
\end{equation} 
then $\bar{\mathbf{x}}$ is the unique nonnegative solution and the minimum $\ell_1$-norm solution to $A\mathbf{x} = \mathbf{b}$.
\end{theorem}


\subsection{Sparsity enhancing methods and approximate measurements} \label{sec:lambda}

Let us now consider the more general situation in which the linear system does not have to be exactly solved, i.e. we rather look for an approximate solution satisfying $A\mathbf{x} = \mathbf{b} +  {\boldsymbol\varepsilon}$, where ${\boldsymbol\varepsilon}$ is the vector containing the error in the measurements. A famous method addressing this problem is the so-called LASSO \cite{Tib96_lasso}. Following \cite{FKo14}, we address to the solution of the following closely related problem
\begin{equation}
    \min \Vert \mathbf{x} \Vert_1^2 + \lambda \Vert A \mathbf{x} -\mathbf{b} \Vert^2_2. \label{eq:l1_LS} 
\end{equation}
As we already pointed out that no theoretical result states that NNLS solvers ensure sparse recovery, problem \eqref{eq:l1_LS} can be interpreted as a sparsity enhancing or squared $\ell_1$-regularization method consisting in adding a penalization term to the objective function. This problem can be thought as a weighted-sum form of a multiobjective problem where the parameter $\lambda$ controls the tradeoff between the two objectives. This technique suffers from an evident drawback that is the choice of $\lambda$, which is not explicitly related norm to the sparsity, nor to the linear system.

In  \cite{FKo14}, problem \eqref{eq:l1_LS} is recast as a nonnegative least squares problem. By making use of the positivity trick, problem \eqref{eq:l1_LS} becomes
\begin{equation}
    \min \left\Vert  \left( \begin{array}{c}
         1\ \dots\ 1  \\
         \lambda {A}\ -\lambda A 
    \end{array} \right) \left( \begin{array}{l}
     \mathbf{x}^+ \\
     \mathbf{x}^-
\end{array} \right) - 
    \left( \begin{array}{c}
         0  \\
        \lambda \mathbf{b}  
    \end{array}\right) \right\Vert^2_2, \qquad \text{s.t. }\mathbf{x}^+, \mathbf{x}^- \geq 0 \label{eq:l1_NNLS} 
\end{equation}
which will be referred as $\ell_1$-NNLS. A nice feature of this form is that the matrix involved clearly belongs to ${\cal M}^+$ and uniqueness of the solution is possible for any $\lambda$. The equivalence between problem \eqref{eq:minl1_problem} and problem \eqref{eq:l1_NNLS} above can be established for $\lambda \rightarrow \infty$, meaning that the squared residual is much more important then the $\ell_1$-norm, which becomes more and more vacuous as $\lambda$ gets larger. This suggests it could be sufficient to solve a simple NNLS problem in order to get a sparse solution.

\section{Upgrading the QR decomposition} \label{sed:modifiedQR}

%
We have already mentioned that in the methods here presented it is not necessary to compute the QR decomposition from scratch, but instead one can upgrade the current QR decomposition. In this section we detail the “updating" and “downdating" procedures to be adopted in the outer and inner loop, respectively. In the next subsection we see the general case, while in the following we see a simplified downdating applicable with the positivity trick.

\subsection{The general case}
Let us first consider the problem of updating the QR decomposition in the outer loop. Let $A$ be an $m \times n$ matrix and $B$ an $m \times k$ matrix, both full rank by columns, such that $k+n \leq m$, and suppose moreover that the columns in $B$ are linearly independent from those in $A$. We want to exploit the knowledge of the decomposition $A = QR$ in order to compute $\widetilde{Q}, \widetilde{R}$ such that
\begin{equation*}
    \widetilde{A} \coloneqq (A\ B) = \widetilde{Q} \widetilde{R}.
\end{equation*}
The following procedure is based on Householder transformations, and it is simply the Block Recursive Householder decomposition (without computing the initial QR decomposition $A=QR$, that has already been computed in this context).

\begin{algorithm}[ht]
\caption{Updating}
\label{alg:up}
\begin{algorithmic}[1]
\State $A = QR$, $B$
\State $B = Q^T B$ \label{alg:up_step:QtB}
\State $\widetilde{Q} = Q$, $\widetilde{R} = (R\ B)$
\For{$i = 1,\dots, k$}
\State $[\beta, v] = $house$(B(n+i:m,i))$ \Comment{compute Householder reflector}
\State $\widetilde{R}(n+i:m,n+i:m) = (I_{m-(n+i)+1} - \beta v\ v^T) \widetilde{R}(n+i:m,n+i:m)$
\State ${H}_{n+i} = I_{m}$  \label{alg:up_step:Q_init}
\State ${H}_{n+i}(n+i:m,n+i:m) = (I_{m-(n+i)+1} - \beta v\ v^T) $
\State $\widetilde{Q} = {H}_{n+i} \widetilde{Q}$ \label{alg:up_step:Q_end}
\EndFor
\end{algorithmic}
\end{algorithm}

Note that explicit computation of $\widetilde{Q}$ is not necessary if we express it as a product of Householder matrices
\begin{equation*}
    \widetilde{Q} = Q {H}_{n+1} \cdots {H}_{n+k}.
\end{equation*}
Operation count for the above algorithm (without taking into account the explicit computation of $\widetilde{Q}$ in steps \ref{alg:up_step:Q_init}--\ref{alg:up_step:Q_end}) if that of an Householder $QR$ decomposition of a matrix of size $m-(n+1) \times k$, i.e. about $2k^2(m-(n+1) -k/3)$ flops, plus the operation count in step \ref{alg:up_step:QtB}. If we suppose that $Q$ itself is expressed as a product of Householder matrices, i.e. 
\begin{equation*}
    {Q} = {H}_{n} \cdots {H}_{n},
\end{equation*}
we get about $kn(2m-n)$ flops. The whole procedure involves about $2k^2(m-(n+1) -k/3) + kn(2m-n) = \mathcal{O}(k^2m + knm)$ flops.

Let us now consider the downdating problem. Consider a set of $k$ column indices $\left\{ j_1, \dots, j_k \right\}$, with $1 \leq j_1 < \dots < j_k < n $, and let $\widetilde{A}$ be the $m \times (n-k)$ matrix whose columns are the columns of $A$ with index not in $\left\{ j_1, \dots, j_k \right\}$. Again, we want to exploit the knowledge of the decomposition $A = QR$ in order to compute $\widetilde{Q}, \widetilde{R}$ such that
\begin{equation*}
    \widetilde{A} = \widetilde{Q} \widetilde{R}.
\end{equation*}
The matrix $\widetilde{A}$ has a structure close to the Hessenberg form. Suppose $m = 8$, $n = 7$, $k=2$, $j_1 = 2$ and $j_2 = 5$. Then we have
\begin{equation*}
    \widetilde{A} = 
    \begin{bmatrix}
    \begin{array}{ccccc}
         \times & \times & \times & \times & \times \\
                & \times & \times & \times & \times \\
                & \times & \times & \times & \times \\
                &        & \times & \times & \times \\
                &        &        & \times & \times \\
                &        &        & \times & \times \\
                &        &        &        & \times \\
                &        &        &        &  \\
    \end{array}
    \end{bmatrix}.
\end{equation*}
The following procedure is based on Givens rotations. We do this in a column-major fashion, using Givens rotation to annihilate the highlighted entries one at a time:
\begin{align*}
    &\begin{bmatrix}
    \begin{array}{ccccc}
         \times & \times & \times & \times & \times \\
                & \times & \times & \times & \times \\
                & \pmb{\times} & \times & \times & \times \\
                &        & \times & \times & \times \\
                &        &        & \times & \times \\
                &        &        & \times & \times \\
                &        &        &        & \times \\
                &        &        &        &  \\
    \end{array}
    \end{bmatrix} \rightarrow
    \begin{bmatrix}
    \begin{array}{ccccc}
         \times & \times & \times & \times & \times \\
                & \times & \times & \times & \times \\
                &        & \times & \times & \times \\
                &        & \pmb{\times} & \times & \times \\
                &        &        & \times & \times \\
                &        &        & \times & \times \\
                &        &        &        & \times \\
                &        &        &        &  \\
    \end{array}
    \end{bmatrix} \rightarrow
    \begin{bmatrix}
    \begin{array}{ccccc}
         \times & \times & \times & \times & \times \\
                & \times & \times & \times & \times \\
                &        & \times & \times & \times \\
                &        &        & \times & \times \\
                &        &        & \times & \times \\
                &        &        & \pmb{\times} & \times \\
                &        &        &        & \times \\
                &        &        &        &  \\
    \end{array}
    \end{bmatrix} \rightarrow \\
    &\begin{bmatrix}
    \begin{array}{ccccc}
         \times & \times & \times & \times & \times \\
                & \times & \times & \times & \times \\
                &        & \times & \times & \times \\
                &        &        & \times & \times \\
                &        &        & \pmb{\times} & \times \\
                &        &        &        & \times \\
                &        &        &        & \times \\
                &        &        &        &  \\
    \end{array}
    \end{bmatrix} \rightarrow
    \begin{bmatrix}
    \begin{array}{ccccc}
         \times & \times & \times & \times & \times \\
                & \times & \times & \times & \times \\
                &        & \times & \times & \times \\
                &        &        & \times & \times \\
                &        &        &        & \times \\
                &        &        &        & \times \\
                &        &        &        & \pmb{\times} \\
                &        &        &        &  \\
    \end{array}
    \end{bmatrix} \rightarrow
    \begin{bmatrix}
    \begin{array}{ccccc}
         \times & \times & \times & \times & \times \\
                & \times & \times & \times & \times \\
                &        & \times & \times & \times \\
                &        &        & \times & \times \\
                &        &        &        & \times \\
                &        &        &        & \pmb{\times} \\
                &        &        &        &  \\
                &        &        &        &  \\
    \end{array}
    \end{bmatrix}.
\end{align*}
Thus, columns with indices $j_l < j < j_{l+1}$ have exactly $l$ subdiagonal entries to annihilate. Note that one could also zero out one subdiagonal at a time, with exactly the same number of floating point operations.
The following algorithm describes the resulting procedure (in column-major fashion). 
\begin{algorithm}[ht]
\caption{Downdating}
\label{alg:down}
\begin{algorithmic}[1]
\State $A = QR$, $j_1 < \dots < j_k $
\State $\widetilde{Q} = Q$, $\widetilde{R} = R$, $j_{k+1} = n+1$ 
\For{$l = 1,\dots, k$}
\For{$j=j_l+1,\dots,j_{l+1}-1$}
\For{$p=1,\dots,l$}
\State $i = j - p +1$
\State $\widetilde{G} = $givens$(\widetilde{A}(i-1:i,j))$ \Comment{compute Givens rotation}
\State $\widetilde{R}(i-1:i,j+1:m) = \widetilde{G} \widetilde{R}(i-1:i,j+1:m)$
\State $\widetilde{Q}(:,i-1:i) = \widetilde{Q}(:,i-1:i) \widetilde{G}^T$ 
\EndFor
\EndFor
\EndFor
\State $\widetilde{R}(:,j_1, \dots, j_k) = [\ ]$ \Comment{delete columns $j_1, \dots, j_k$}
\end{algorithmic}
\end{algorithm}

Again, the explicit computation of $\widetilde{Q}$ is not necessary if we express it as a product of Givens rotation matrices
\begin{equation*}
    \widetilde{Q} = Q {G}_{1}^T \cdots {G}_{t}^T,
\end{equation*}
where $t$ is the total number of computed rotations and $G_l$ is the $m\times m$ matrix obtained enlarging the corresponding $\widetilde{G}_l$.

The operation count can be computed by observing that this procedure boils down to the annihilation of $k$ subdiagonals of different length, as in the second exemplification.
After removing column $j_l$, we have to zero out the last element of each column of the remaining $m \times (n-j_l - (k-l))$ matrix: in fact the columns from $j_l$-th (not included) until the last one are $n-j_l$, but there are also $k-l$ columns we don't have to downdate, those with indices $j_{l+1},\dots,j_{k}$, that have already been removed. Define $n_l = (n-j_l - (k-l))$. Then, this operation takes $\sum_{i=0}^{n_l} 6 (n_l -i) = \mathcal{O}(3n_l^2) = \mathcal{O}(3(n-j_l - (k-l))^2)$. We compute the worst case of the total operation count, in which the first $k$ columns are removed: we have $j_l = l$ with $1 \leq l \leq k$, thus we have to zero out the last $k$ subdiagonals of the remaining $m \times (n-k)$ matrix. This involves $\mathcal{O}(\sum_{l=1}^k 3(n-k)^2 )= \mathcal{O}(3n^2k+3k^3-6nk^2)$ flops.

The updating and downdating procedures allow to modify the current QR decomposition with a remarkable computational saving. In a practical implementation, one can apply each time the orthogonal transformations to the whole matrix, or only to the columns whose indices are in the current passive set. In the first case, it is not necessary to keep trace of the transformations themselves. In the second case, we need to keep in memory the sequence of orthogonal transformations computed so far; moreover, an additional vector is needed to keep trace of the number of orthogonal transformations that have been applied to each column.

\subsection{The positivity trick case}
\label{sec:positivity-trick}

Let us now consider the particular case in which we deal with the NNLS problem \eqref{eq:positivity_trick}. In such a case, the objective function's matrix is $[A\ ;\ -A]$. The updating procedure by Householder triangularization is unchanged, while the downdating procedure by Givens rotations in the inner loop can be avoided with a simple “sign flip", as we show here.

Suppose that $\mathbf{y}\myPart{(s)}{}$ is not a feasible point, thus $\mathbf{y}\myPart{(s)}{k} <0$, for some $k \in P_s$. Since $\mathbf{s}\myPart{(s)}{}:= \mathbf{y}\myPart{(s+1)}{} - \mathbf{x}\myPart{(s-1)}{}$ is a feasible descend direction, we have $k \not \in J_s$. If $k\leq n$, the index $k$ corresponds to a column of $A$, otherwise if $n < k\leq 2n$, $k$ corresponds to a column of $-A$. Let us denote by $\bar{k}$ the twin index of $k$, i.e. $\bar{k} = k+n$ if $k\leq n$, and $\bar{k} = k-n$ if $n < k\leq 2n$. Let us show that if we let $P\myPart{(s+1)}{} = (P\myPart{(s)}{} \cup \left\{ \bar{k} \right\}) \setminus \left\{ {k} \right\}$, then $\mathbf{y}\myPart{(s+1)}{}$ is a feasible point and $\mathbf{s}\myPart{(s+1)}{}:= \mathbf{y}\myPart{(s+1)}{} - \mathbf{x}\myPart{(s-1)}{}$ is a feasible descend direction.

We first show that $\mathbf{y}\myPart{(s+1)}{}$ is a feasible point. With a little abuse of notation, we will denote by $k$ both the index itself and its position within the passive set $P_s$.
If we write by columns $A\myPart{(s)}{P} = (\mathbf{a}_1\ \dots \mathbf{a}_{k-1}\ \mathbf{a}_k\ \mathbf{a}_{k+1} \dots \mathbf{a}_{n_s} )$, then we have $A\myPart{(s+1)}{P} = (\mathbf{a}_1\ \dots \mathbf{a}_{k-1}\ -\mathbf{a}_k\ \mathbf{a}_{k+1} \dots \mathbf{a}_{n_s} )$. 
Consider the QR decomposition
\begin{equation*}
    A\myPart{(s)}{P} = Q\myPart{(s)}{}\ R\myPart{(s)}{} = Q\myPart{(s)}{}\ (\mathbf{r}_1\ \dots \mathbf{r}_{k-1}\ \mathbf{r}_k\ \mathbf{r}_{k+1} \dots \mathbf{r}_{n_s} ),
\end{equation*}
then we can write
\begin{equation*}
    A\myPart{(s+1)}{P} = Q\myPart{(s)}{}\ R\myPart{(s+1)}{} = Q\myPart{(s)}{}\ (\mathbf{r}_1\ \dots \mathbf{r}_{k-1}\ -\mathbf{r}_k\ \mathbf{r}_{k+1} \dots \mathbf{r}_{n_s} ).
\end{equation*}
Indeed, if $H$ is the Householder reflector of a vector $\mathbf{x}$, i.e. $H\mathbf{x} = \Vert \mathbf{x} \Vert\ \mathbf{e}_1$, then $H(-\mathbf{x}) = -\Vert \mathbf{x} \Vert\ \mathbf{e}_1$. Now,  $\mathbf{y}\myPart{(s)}{}$ is the solution of the least squares problem 
\begin{equation}
     \mathbf{y}\myPart{(s)}{P} = \argmin_{\mathbf{y}}  \left\Vert A\myPart{(s)}{P} \mathbf{y}-\mathbf{b}\right\Vert, \quad \mathbf{y}\myPart{(s)}{Z}=0,
\end{equation}
which can be obtained as the solution of 
\begin{equation}
    R\myPart{(s)}{} \mathbf{y}\myPart{(s)}{P} = \mathbf{b}\myPart{(s)}{P}:= (Q\myPart{(s)}{})^T\mathbf{b}, \quad \mathbf{y}\myPart{(s)}{Z}=0.
\end{equation}
Let us split $\mathbf{y}\myPart{(s)}{P} = (\mathbf{y}\myPart{(s)}{1}, y\myPart{(s)}{k}, \mathbf{y}\myPart{(s)}{2})^T$, where $\mathbf{y}\myPart{(s)}{1}\in \mathbb{R}^{k-1}$, $\mathbf{y}\myPart{(s)}{2}\in \mathbb{R}^{n_s - k}$, and $y\myPart{(s)}{k}<0$. Similarly, we write $\mathbf{b}\myPart{(s)}{P} = (\mathbf{b}\myPart{(s)}{1}, b\myPart{(s)}{k}, \mathbf{b}\myPart{(s)}{2})^T$, where $\mathbf{b}\myPart{(s)}{1}\in \mathbb{R}^{k-1}$, $\mathbf{b}\myPart{(s)}{2}\in \mathbb{R}^{m - k}$, and $\mathbf{y}\myPart{(s+1)}{P} = (\mathbf{y}\myPart{(s+1)}{1}, y\myPart{(s+1)}{\bar{k}}, \mathbf{y}\myPart{(s+1)}{2})^T$. The triangular matrices can be split as
\begin{equation}
    R\myPart{(s)}{} = \left( \begin{array}{lll}
         R_{11} & \mathbf{r}_{k}(1:k-1) & R_{12}(1:k-1,:)  \\
         0^T &  \mathbf{r}_{k}(k) & R_{12}(k,:)\\
         0 & 0 & R_{22}
    \end{array}\right),
\end{equation}
and 
\begin{equation}
    R\myPart{(s+1)}{} = \left( \begin{array}{lll}
         R_{11} & \mathbf{-r}_{k}(1:k-1) & R_{12}(1:k-1,:)  \\
         0^T &  \mathbf{-r}_{k}(k) & R_{12}(k,:)\\
         0 & 0 & R_{22}
    \end{array}\right),
\end{equation}
where $R_{11}$ is upper triangular of order $k-1$, $R_{12}$ has size $k \times (n_s-k)$, and $R_{22}$ is upper triangular of size $(m-k)\times (n_s-k)$.

Since $\mathbf{y}\myPart{(s)}{2}$ solves $\min \| R_{22}\mathbf{y} - \mathbf{b}_2 \|$, we have $\mathbf{y}\myPart{(s)}{2} = \mathbf{y}\myPart{(s+1)}{2}$.
By backsubstitution, we obtain
\begin{equation}
   y\myPart{(s)}{k} = \frac{1}{\mathbf{r}_{k}(k)}\left( b\myPart{(s)}{k} - R_{12}(k,1:n_s-k)^T \mathbf{y}\myPart{(s)}{2} \right) <0 ,
\end{equation}
and similarly 
\begin{equation}
   y\myPart{(s+1)}{\bar{k}} = -\frac{1}{\mathbf{r}_{k}(k)}\left( b\myPart{(s)}{k} - R_{12}(k,1:n_s-k)^T \mathbf{y}\myPart{(s+1)}{2} \right) = - y\myPart{(s)}{k}  > 0,
\end{equation}
where we used the fact that $\mathbf{y}\myPart{(s+1)}{2} = \mathbf{y}\myPart{(s)}{2}$.
Last, we have
\begin{equation}
   \mathbf{y}\myPart{(s)}{1} = R_{11}^{-1}\left(\mathbf{b}\myPart{(s)}{1} - y\myPart{(s)}{k} \mathbf{r}_{k}(1:k-1) - R_{12}\mathbf{y}\myPart{(s)}{2}\right),
\end{equation}
and
\begin{equation}
\begin{aligned}
   \mathbf{y}\myPart{(s+1)}{1} &= R_{11}^{-1}\left(\mathbf{b}\myPart{(s)}{1} + y\myPart{(s+1)}{\bar{k}} \mathbf{r}_{k}(1:k-1) - R_{12}(1:k-1,:)\mathbf{y}\myPart{(s+1)}{2}\right) \\
   &= R_{11}^{-1}\left(\mathbf{b}\myPart{(s)}{1} - y\myPart{(s)}{k} \mathbf{r}_{k}(1:k-1) - R_{12}(1:k-1,:)\mathbf{y}\myPart{(s)}{2}\right) = \mathbf{y}\myPart{(s)}{1}.
\end{aligned}
\end{equation}
Therefore, $\mathbf{y}\myPart{(s+1)}{P} = (\mathbf{y}\myPart{(s)}{1}, -y\myPart{(s)}{k}, \mathbf{y}\myPart{(s)}{2})^T>0$, meaning that $\mathbf{y}\myPart{(s+1)}{}$ is a feasible point and $\mathbf{s}\myPart{(s+1)}{} = (\mathbf{y}\myPart{(s+1)}{} - \mathbf{x}\myPart{(s-1)}{})$ is a feasible direction.

Let us show that $\mathbf{s}\myPart{(s+1)}{}$ is also a descend direction.  Suppose now without loss of generality that $k\leq n$ ($\bar{k}>n$), i.e. $k$ corresponds to a column of $A$, and let $A=(\mathbf{a}_1\ \dots \ \mathbf{a}_n)$. Recall that, at the beginning of the inner loop, the vector $\mathbf{s}\myPart{(s)}{} = (\mathbf{y}\myPart{(s)}{} - \mathbf{x}\myPart{(s-1)}{})$ is a descend direction, i.e.
\begin{equation*}
    \|\mathbf{r}\myPart{(s)}{}\| = \|(A\ -A)\mathbf{y}\myPart{(s)}{} - \mathbf{b}\| 
    < \|(A\ -A)\mathbf{x}\myPart{(s-1)}{} - \mathbf{b}\|.
\end{equation*}
 We have 
\begin{equation*}
\begin{aligned}
    \mathbf{r}^{(s+1)} &= \mathbf{b} - (A\ -A)\mathbf{y}\myPart{(s+1)}{}\\ 
    &= 
    \mathbf{b} - \sum_{\substack{i \in P \\ i \leq n}}\mathbf{a}_i y\myPart{(s+1)}{i} + \sum_{\substack{i \in P \\ i > n, i \neq \bar{k}}}\mathbf{a}_i y\myPart{(s+1)}{i} + \mathbf{a}_i y\myPart{(s+1)}{\bar{k}}\mathbf{a}_{\bar{k}}\\
    &= 
    \mathbf{b} - \sum_{\substack{i \in P \\ i \leq n, i \neq k}}\mathbf{a}_i y\myPart{(s)}{i} + \sum_{\substack{i \in P \\ i > n}}\mathbf{a}_i y\myPart{(s)}{i} + \mathbf{a}_i y\myPart{(s)}{k}\mathbf{a}_{k} \\
    &=
    \mathbf{r}^{(s)},
\end{aligned}
\end{equation*}
therefore $\mathbf{s}\myPart{(s+1)}{}$ is a feasible descend direction with $\mathbf{y}\myPart{(s+1)}{}$ feasible. We can set $\mathbf{x}\myPart{(s+1)}{} = \mathbf{y}\myPart{(s+1)}{}$. Notice that once computed $\mathbf{y}\myPart{(s)}{}$, no further computation is needed to obtain $\mathbf{y}\myPart{(s+1)}{}$, except for the sign flip.

\section{Numerical experiments}
\label{sec:numerical_experiments}
In this section, we present a few results obtained by applying LHDM and other methods for sparse recovery, to a quite large set of linear systems. 


We compared LHDM with LH, both written by us in Matlab, with the other convex relaxation techniques tested in \cite{AAVV_Solving_Basis_Pursuit-Heuristic_Optimality_Check_and_Solver_Comparison} and with the Orthogonal Matching Pursuit algorithm (OMP), whose implementation is publicly available online:
\begin{itemize}
    \item “$l1\_magic$", a primal-dual interior-point method for a LP reformulation of Basis-Pursuit \eqref{eq:minl1_problem}, from \verb@www.l1-magic.org@;
    \item “$l1\_homotopy$", from \texttt{users.ece.gatech.edu/\textasciitilde sasif/ homotopy};
    \item “$solveBP$", from \textit{SparseLab} (\texttt{sparselab.stanford.edu});
    \item “$spgl1$", from \texttt{www.cs.ubc.ca/labs/scl/spgl1};
    \item “$isal1$", from \texttt{wwwopt.mathematik.tu-darmstadt.de/spear};
    \item “$solveOMP$", from \textit{SparseLab} (\texttt{sparselab.stanford.edu}).
\end{itemize}

As done in \cite{AAVV_Solving_Basis_Pursuit-Heuristic_Optimality_Check_and_Solver_Comparison}, we used the default setting for all optional parameters of these solvers.

The comparisons are made in terms of solution quality, i.e. the $\ell_2$ norm of the error (“distance to optimum"), and execution times. Note that the distance to optimum is computable because the unicity of the solution is guaranteed, under the ERC property \eqref{eq:ERC}.

Our purpose is not to provide an exhaustive experimental comparison between the methods but, instead, to assess the substantial performance gain attained by the Deviation Maximization to the Lawson-Hanson algorithm, i.e. by the LHDM algorithm, and its competitiveness with the other methods publicly available for sparse recovery of dense, underdetermined systems. Actually, numerical experiments reveal that LHDM is a good choice for sparse recovery, for a wide class of instances. 

Numerical tests have been carried out on a dataset freely available online and fully described in \cite{AAVV_Solving_Basis_Pursuit-Heuristic_Optimality_Check_and_Solver_Comparison}. For brevity, we refer to the cited article for details about the $548$ matrices that make up the database. 
It contains many instances satisfying the ERC \eqref{eq:ERC}, thus representing a favorable situation for compressed sensing in terms of recoverability of sparse solutions via $\ell_1$ minimization. It also includes instances not satisfying the ERC in order to assess solvers' performance outside the so-called $\ell_0-\ell_1$ equivalence.
In our test we have used $444$ instances over $548$, since we restrict our investigation to dense matrices from this dataset. In fact, dealing with sparse matrices would require a dedicated implementation of LHDM which can be addressed in the future. In order to simplify the comparison of the results here presented with the ones in \cite{AAVV_Solving_Basis_Pursuit-Heuristic_Optimality_Check_and_Solver_Comparison}, 
we keep the same indexing from $1$ to $548$ to show the results. Results concerning sparse instances are left blank. Instances satisfying the ERC are indexed in two intervals $[1,200]$ and $[275,474]$. This helps the understanding of the Figures.

We used the positivity trick presented in \eqref{eq:positivity_trick} in order to achieve an arbitrary signed solution by means of LHDM, which succeeded in exact recovery for all tests verifying the ERC, meaning that the solution with smallest support was found.

We used a fixed choice of the thresholds used in LHDM, that is $\tau_1=0.6, \delta=0.9, \tau_2=0.15$ and $k_{\max}=32$, that are the same values used in \cite{DMa21}, hence giving evidence of the wide applicability of the method proposed without the need of choosing problem dependent parameters.
However, a fine tuning of the parameters can give slight improvements in specific problem families. For example, for very well-conditioned matrices, $\delta$ can be substantially smaller, to get a more efficient choice of columns. This fine tuning is out of the scope of this paper since here we want to exploit the general applicability of the method.

The results are grouped into subsections, each devoted to a single experimental aspect.

\subsection{Efficiency of the “sign flip"}

We have seen in sec. \ref{sec:positivity-trick} that, in the case of the positivity trick, the inner loop of LH can be substituted by a much simpler “sign flip" of the involved R columns. In Figure \ref{fig:results_1} we see in LH the difference between running the inner loop and the sign-flip: there is an interesting speedup, around $1.5\times$, for the more expensive tests.

\begin{figure}[H]
\centering
    \begin{tabular}{ccc}
      \includegraphics[height=0.22\textwidth]{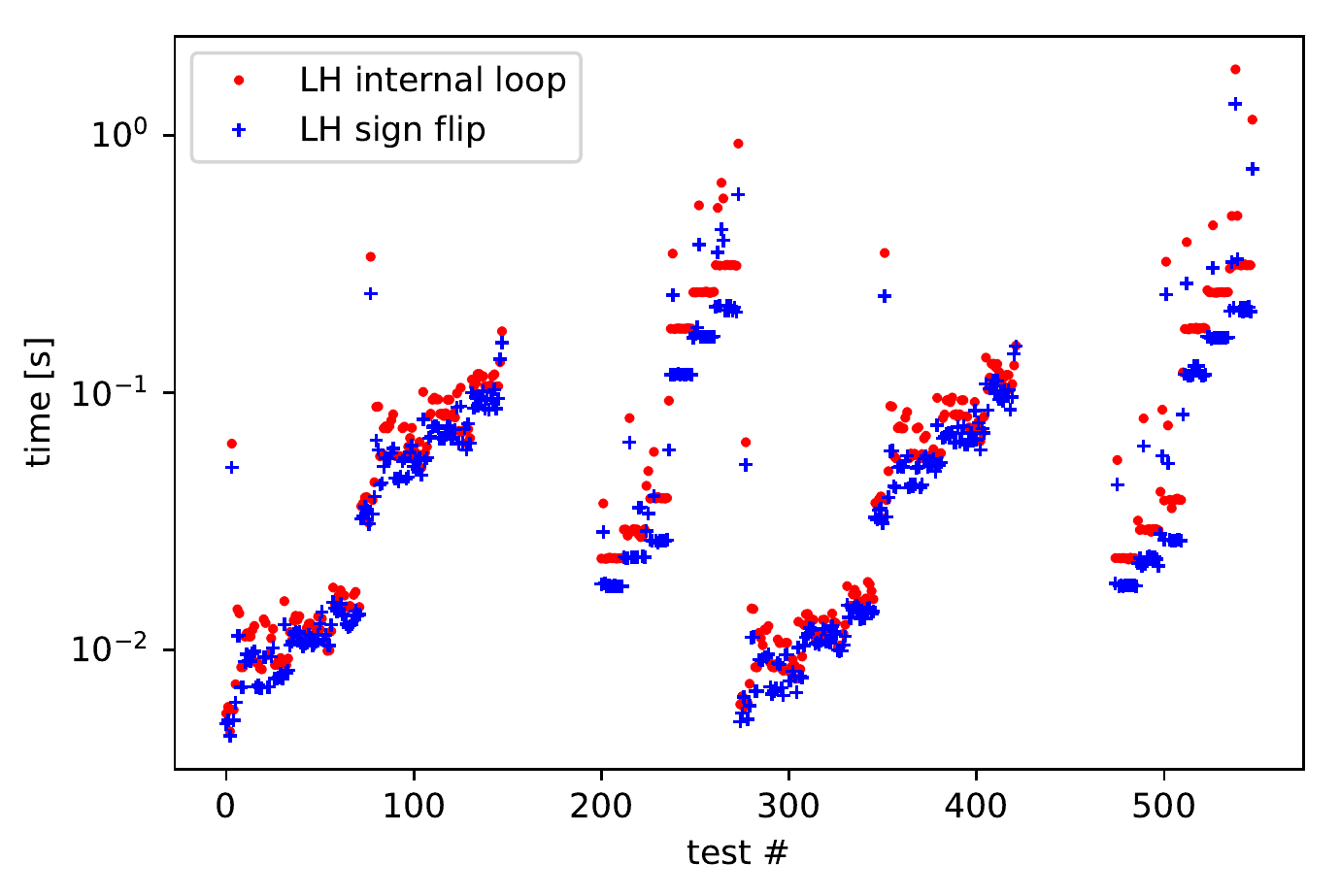}&
      \includegraphics[height=0.22\textwidth]{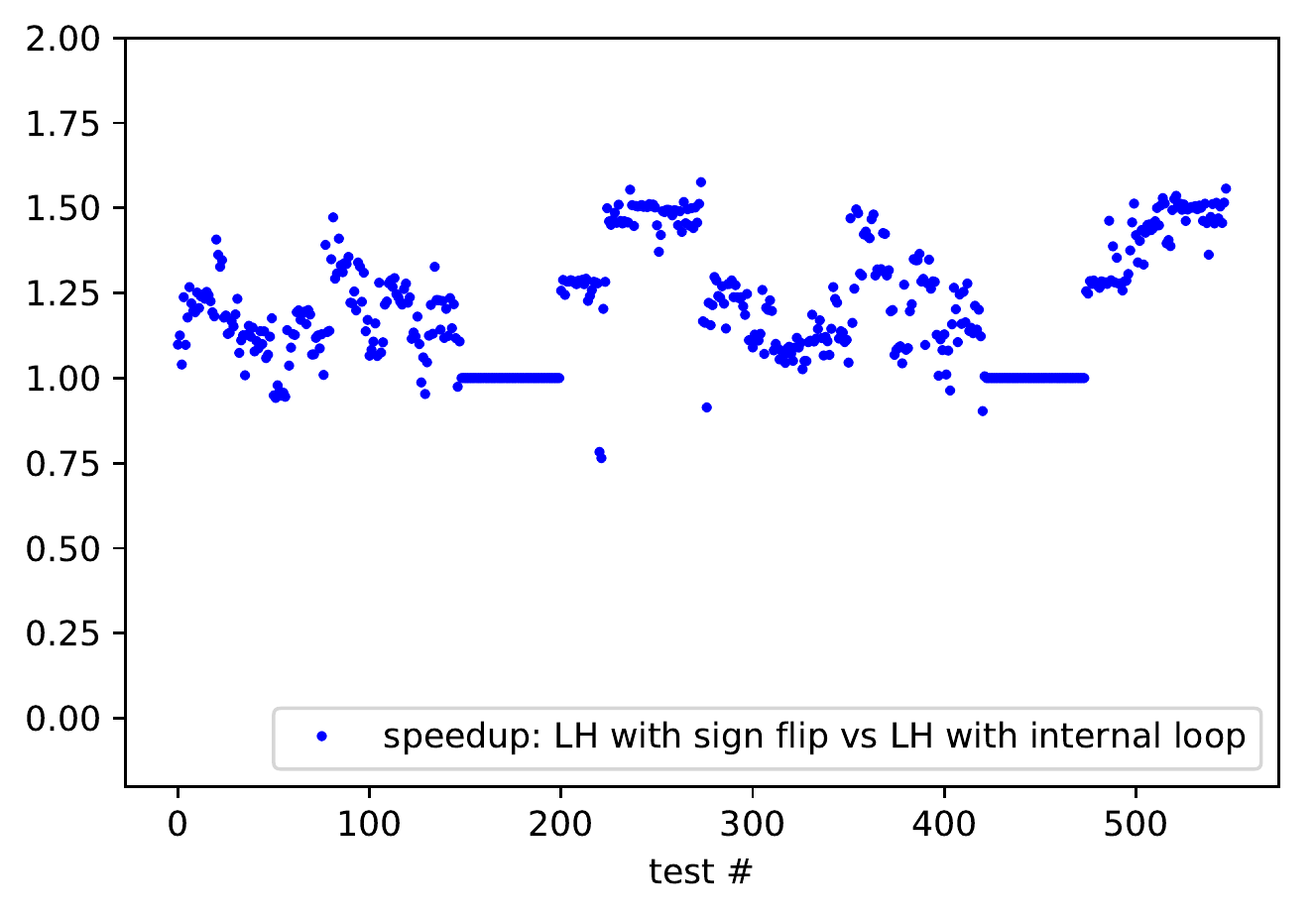}&
      \includegraphics[height=0.22\textwidth]{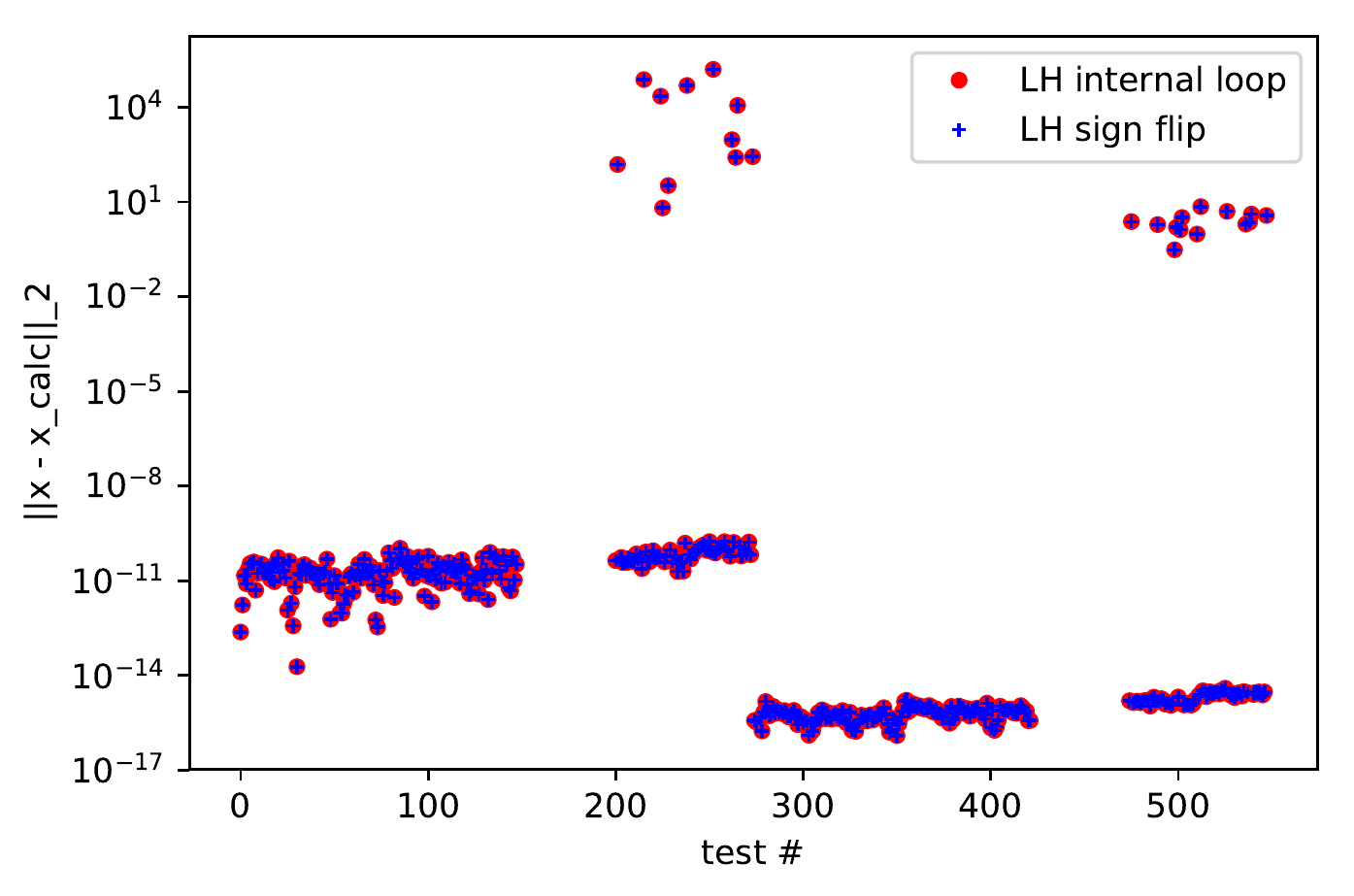}
    \end{tabular}
      \caption{Execution times (left), speedup of LH with sign flip vs LH with internal loop (center) and distance to optimum $|| \mathbf{x} - \mathbf{x}_{calc} ||_2$ (right), for the Lawson-Hanson (LH) algorithm, with the classic internal loop and with its replacement of sec. \ref{sec:positivity-trick} (“sign flip"). The vector $\mathbf{x}_{calc}$ is the solution computed by the algorithm and $\mathbf{x}$ the true solution (unique, in ERC tests).} 
      \label{fig:results_1}
\end{figure}

\subsection{Speedup of LHDM versus LH}
\label{sec:speedup_LHDM_vs_LH}

In Figure \ref{fig:results_2} and \ref{fig:results_5} we see the speedup of LHDM with respect to LH, for solution vectors at different levels of sparsity. The original dataset of \cite{AAVV_Solving_Basis_Pursuit-Heuristic_Optimality_Check_and_Solver_Comparison} has instances with very sparse solution vectors; we created a few modified datasets, starting from the original, with random generated solutions at prescribed sparsity levels, and perturbing matrices to satisfy the ERC property at each instance.
As we can notice, the speedup varies a lot with the sparsity of the solution: the more the solution is less sparse, i.e. the number of nonzeros increases, the more LHDM accelerates on LH. The Figures shows the two extreme situations: the original database with very sparse solutions and a modified database with almost dense solution vectors.

\begin{figure}[H]
\centering
    \begin{tabular}{ccc}
      \includegraphics[height=0.22\textwidth]{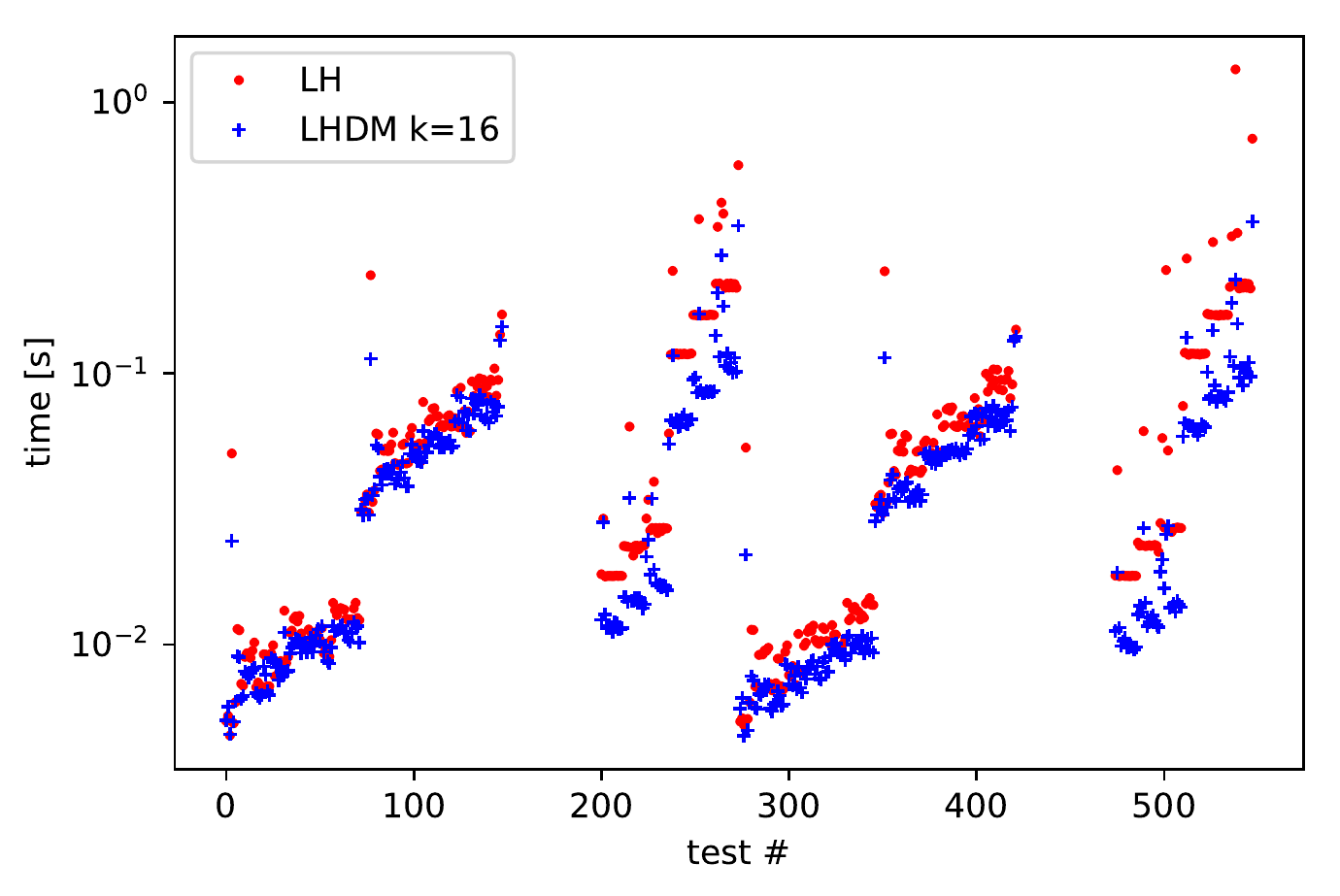}&
      \includegraphics[height=0.22\textwidth]{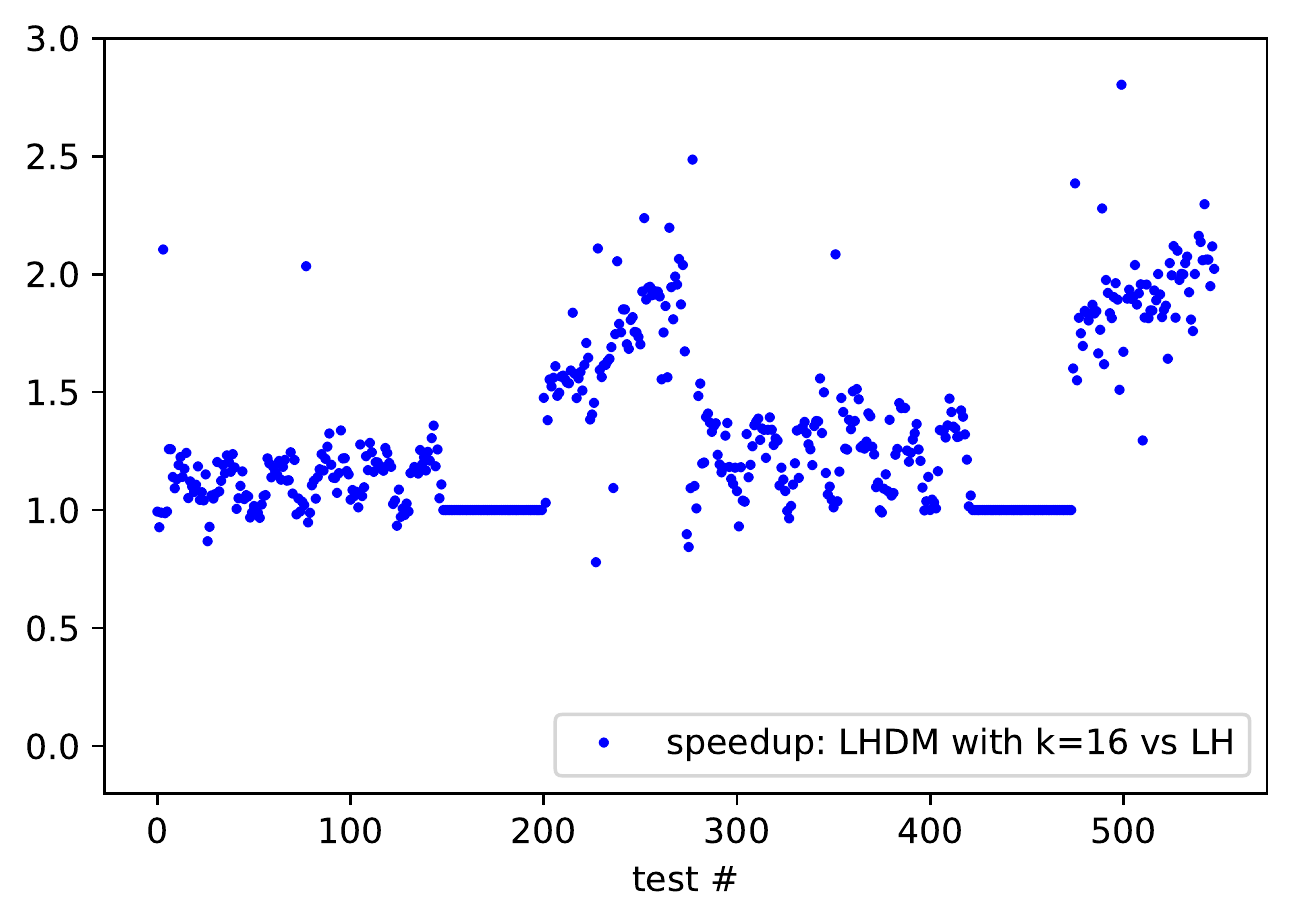}&
      \includegraphics[height=0.22\textwidth]{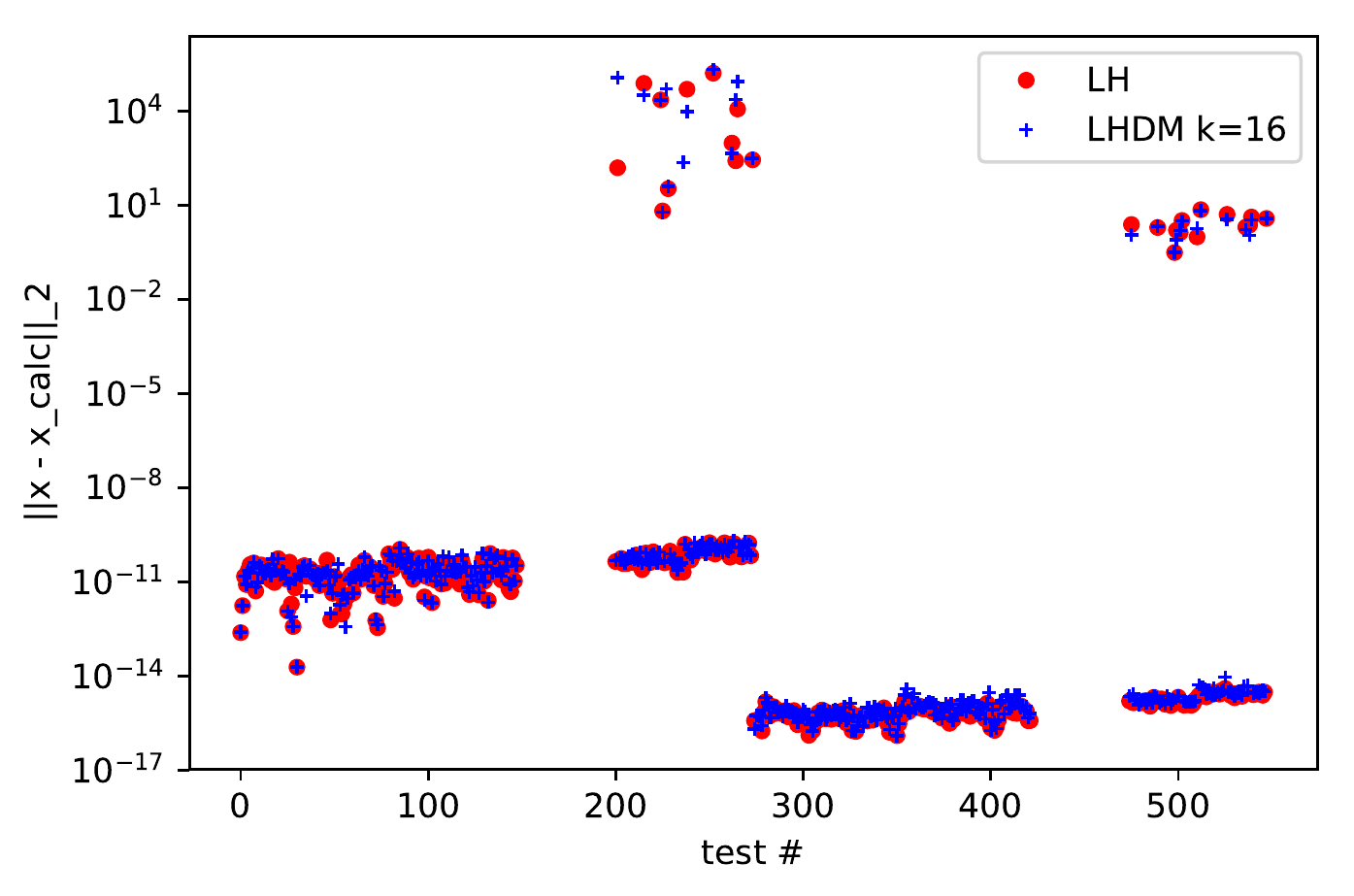}
    \end{tabular}
      \caption{Execution times (left), speedup of LHDM vs LH (center) and distance to optimum $|| \mathbf{x} - \mathbf{x}_{calc} ||_2$ (right), for the original dataset, whose instances have very sparse solution vectors. The vector $\mathbf{x}_{calc}$ is the solution computed by the algorithm and $\mathbf{x}$ the true solution (unique, in ERC tests).} 
      \label{fig:results_2}
\end{figure}

\begin{figure}[H]
\centering
    \begin{tabular}{ccc}
      \includegraphics[height=0.22\textwidth]{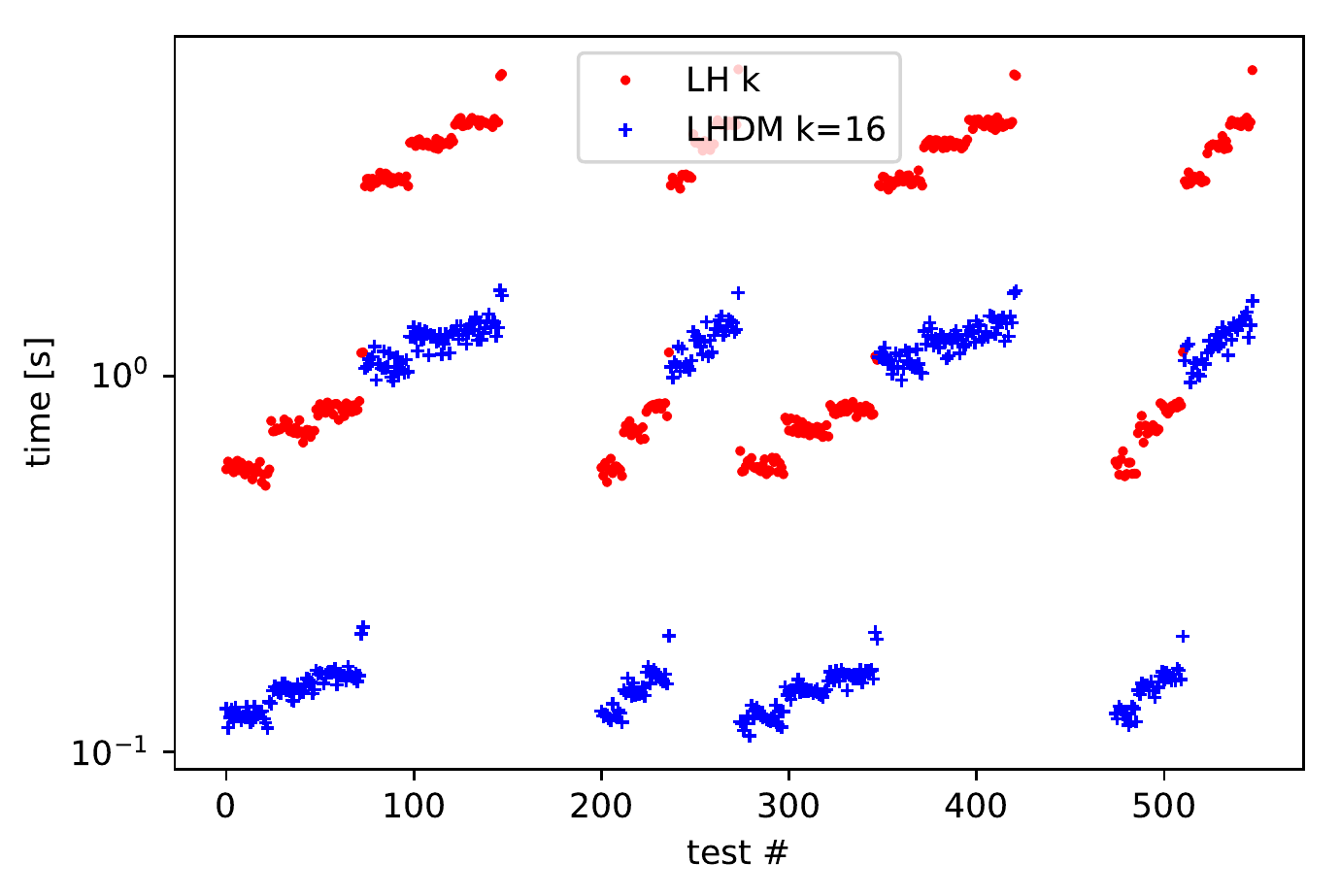}&
      \includegraphics[height=0.22\textwidth]{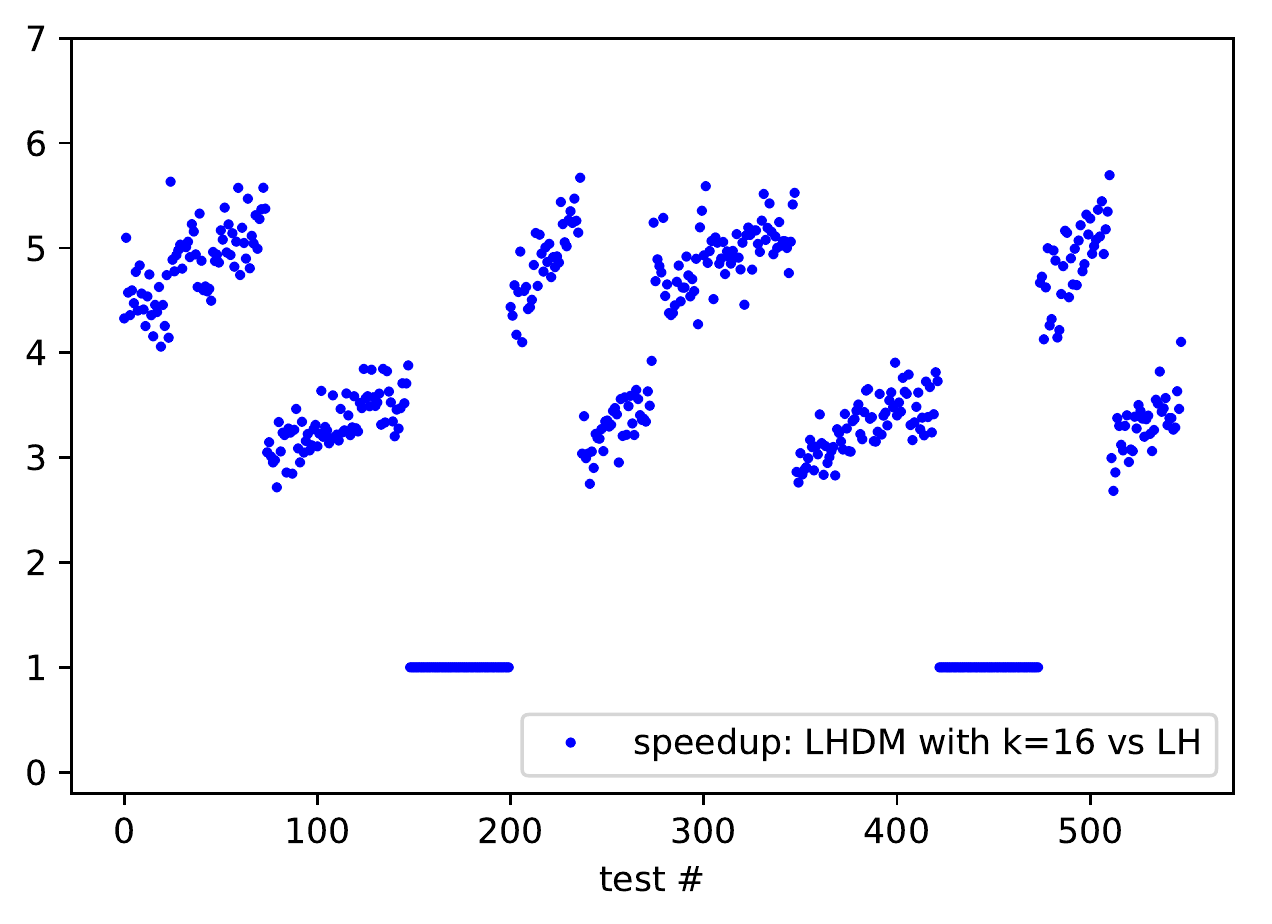}&
      \includegraphics[height=0.22\textwidth]{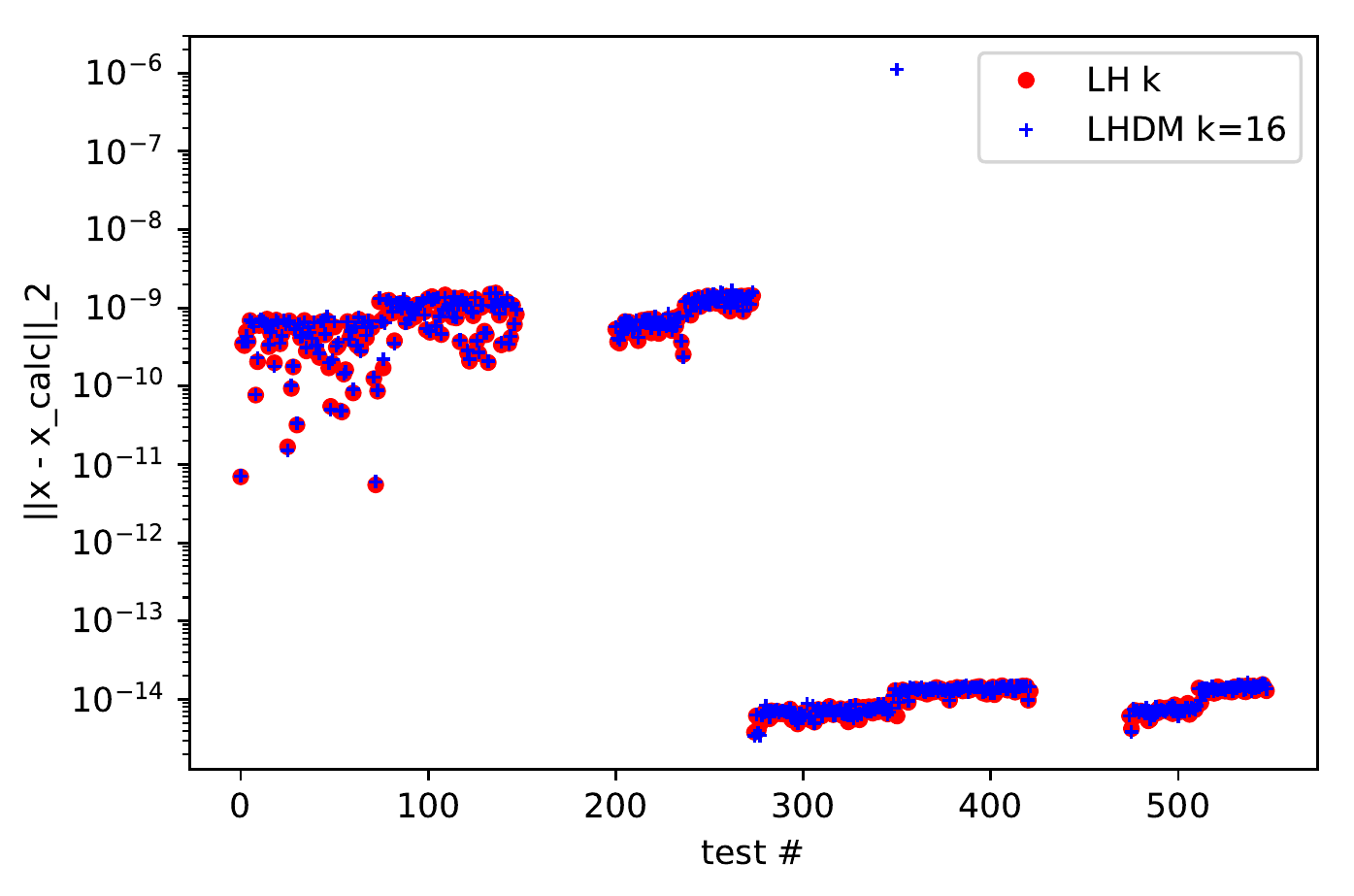}
    \end{tabular}
      \caption{Execution times (left), speedup of LHDM vs LH (center) and distance to optimum $|| \mathbf{x} - \mathbf{x}_{calc} ||_2$ (right), for the modified dataset, whose instances have almost dense solution vectors. The vector $\mathbf{x}_{calc}$ is the solution computed by the algorithm and $\mathbf{x}$ the true solution (unique, in ERC tests).} 
    \label{fig:results_5}
\end{figure}

\subsection{Comparison with the other convex relaxation solvers}

In Figure \ref{fig:comparison_database_1} we see the comparison of LHDM with the other convex relaxation solvers listed at the beginning of this section.

\begin{figure}[H]
\centering
    \begin{tabular}{ccc}
      \includegraphics[height=0.22\textwidth]{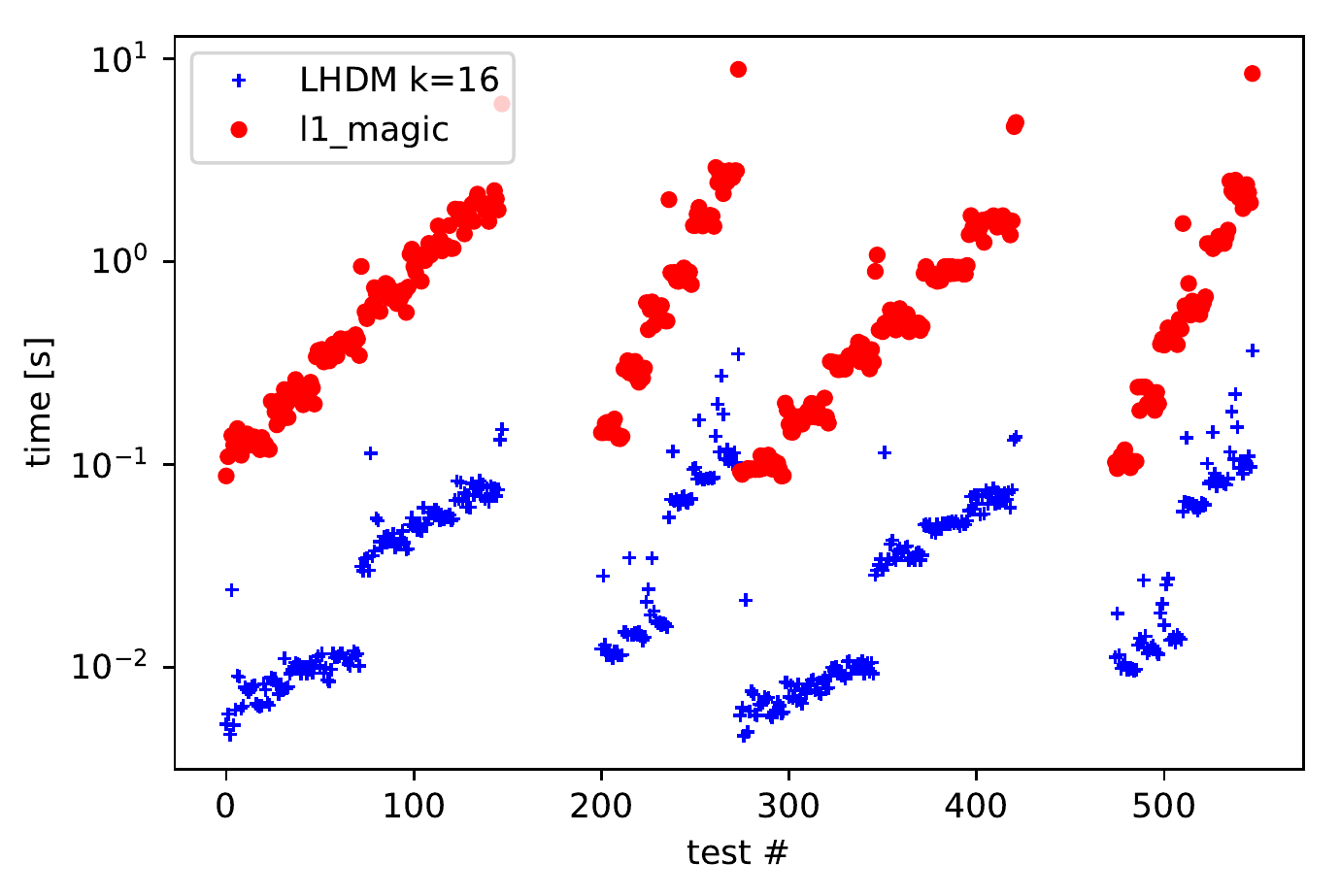}&
      \includegraphics[height=0.22\textwidth]{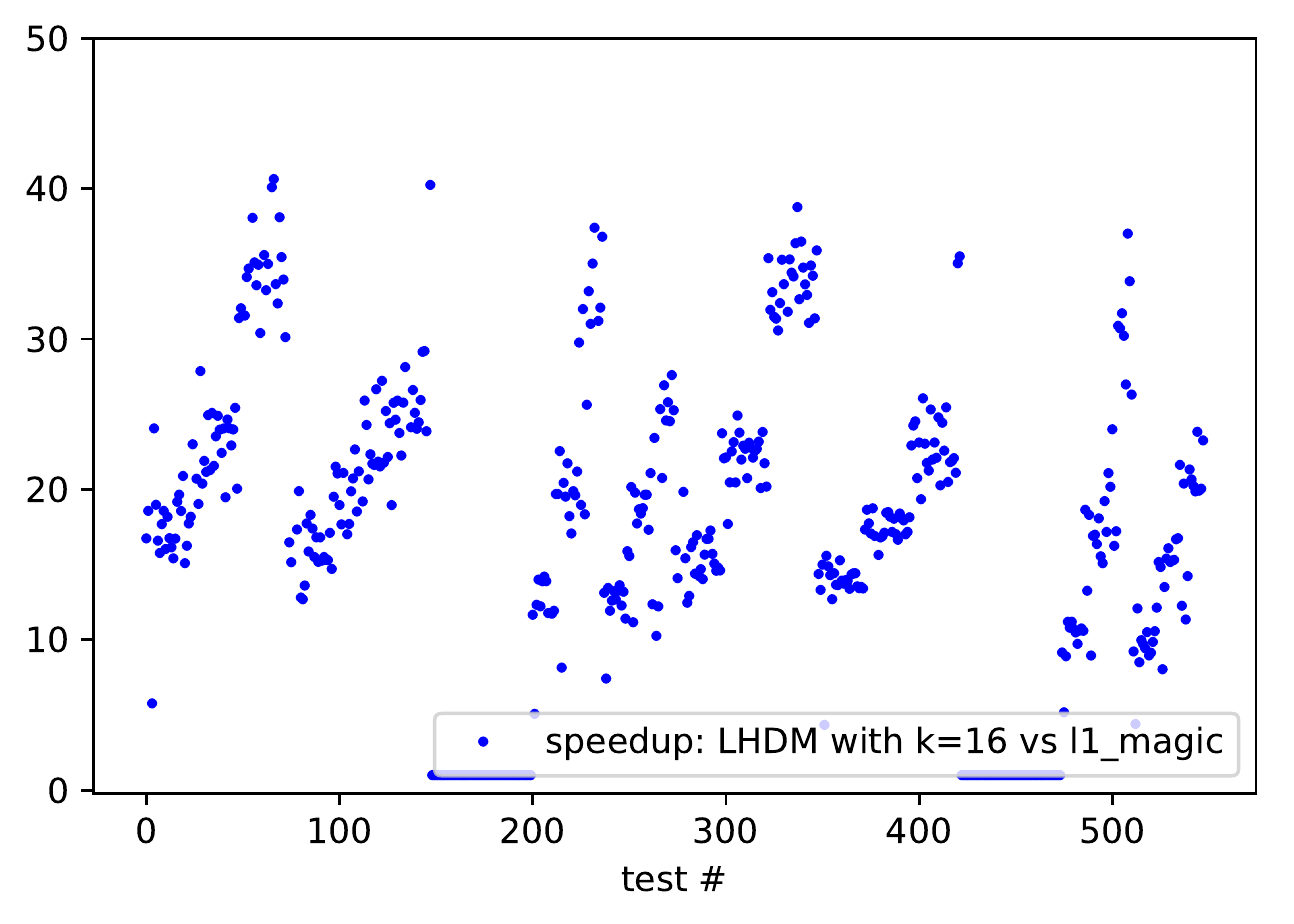}&
      \includegraphics[height=0.22\textwidth]{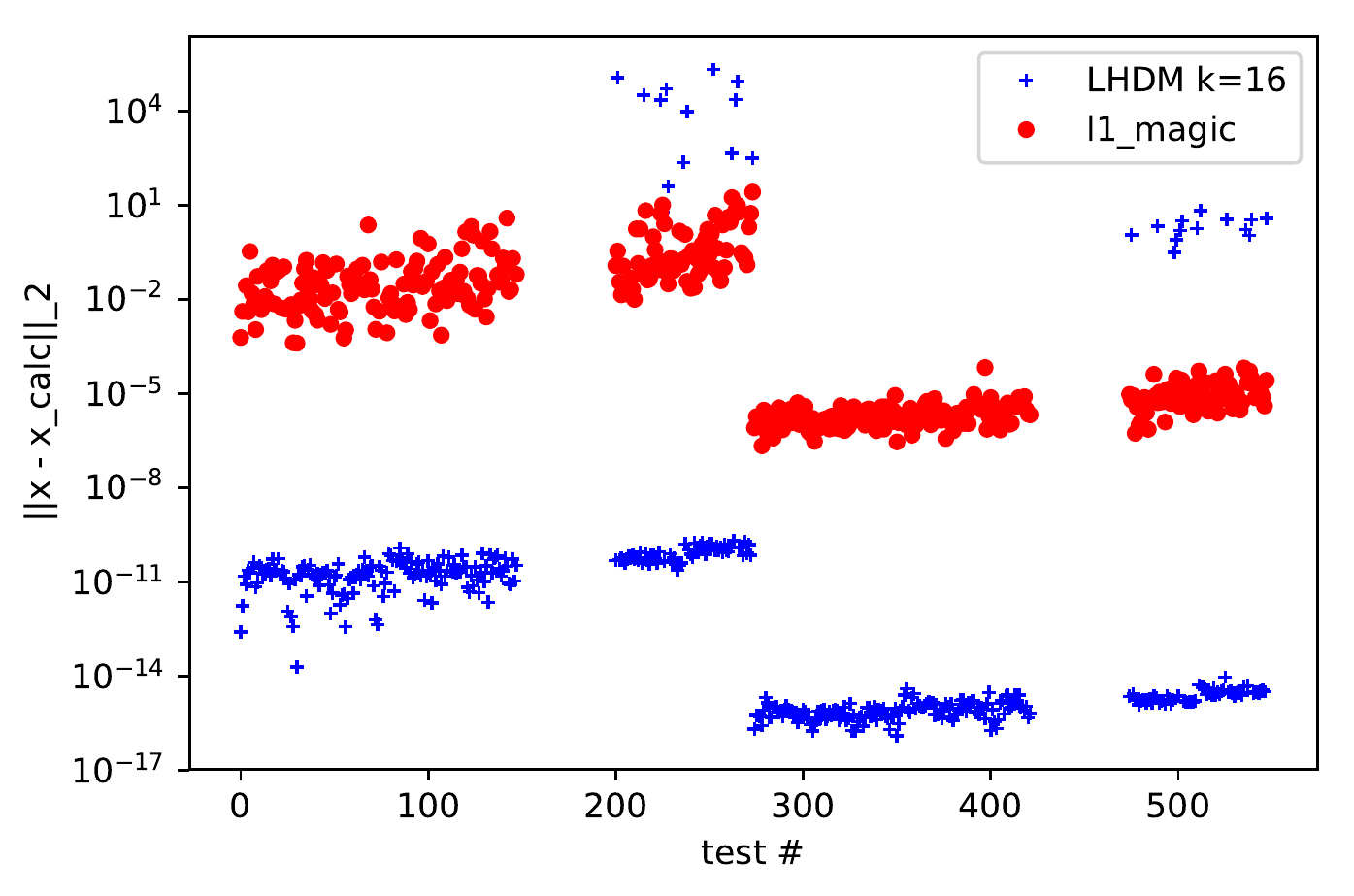}\\
      \includegraphics[height=0.22\textwidth]{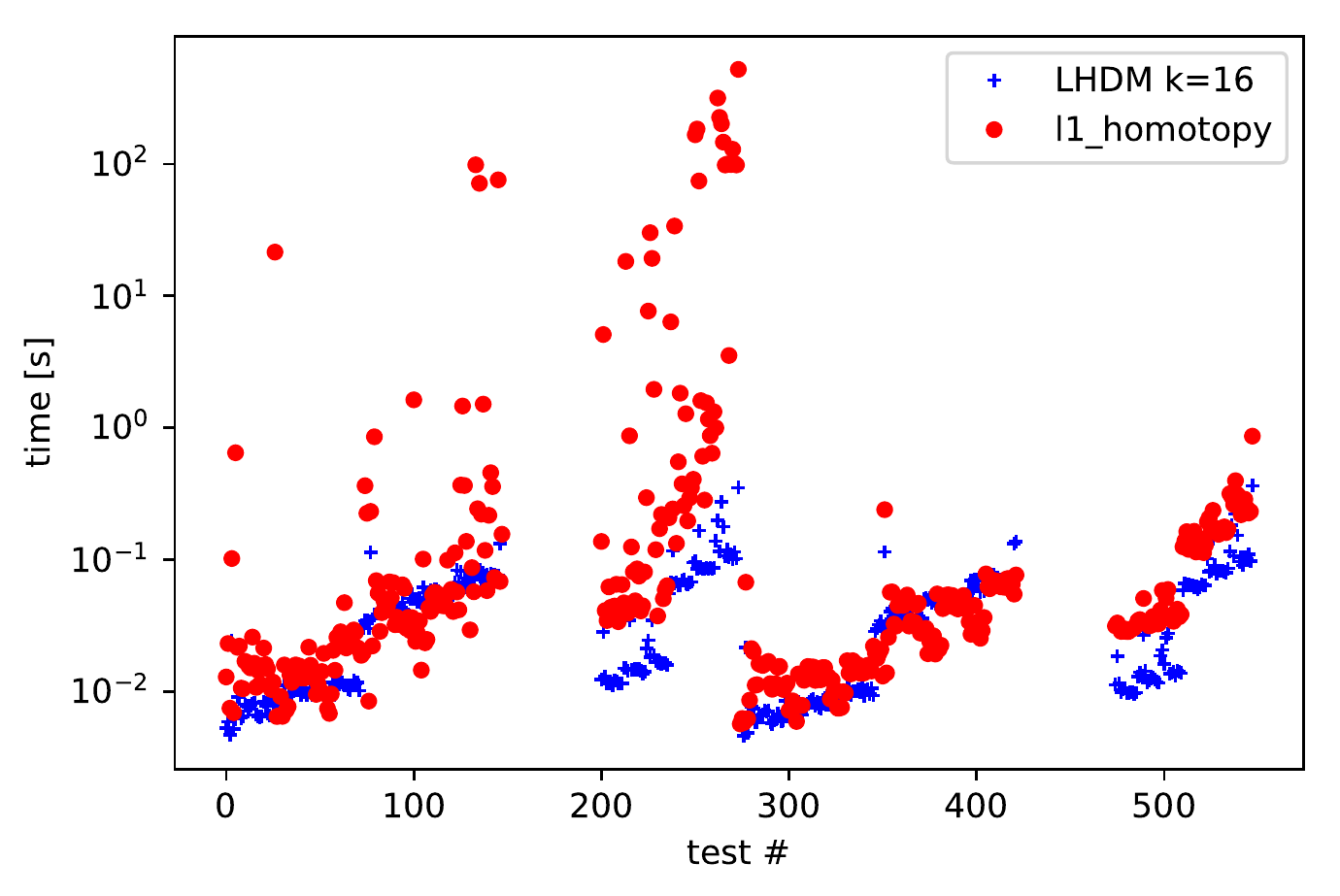}&
      \includegraphics[height=0.22\textwidth]{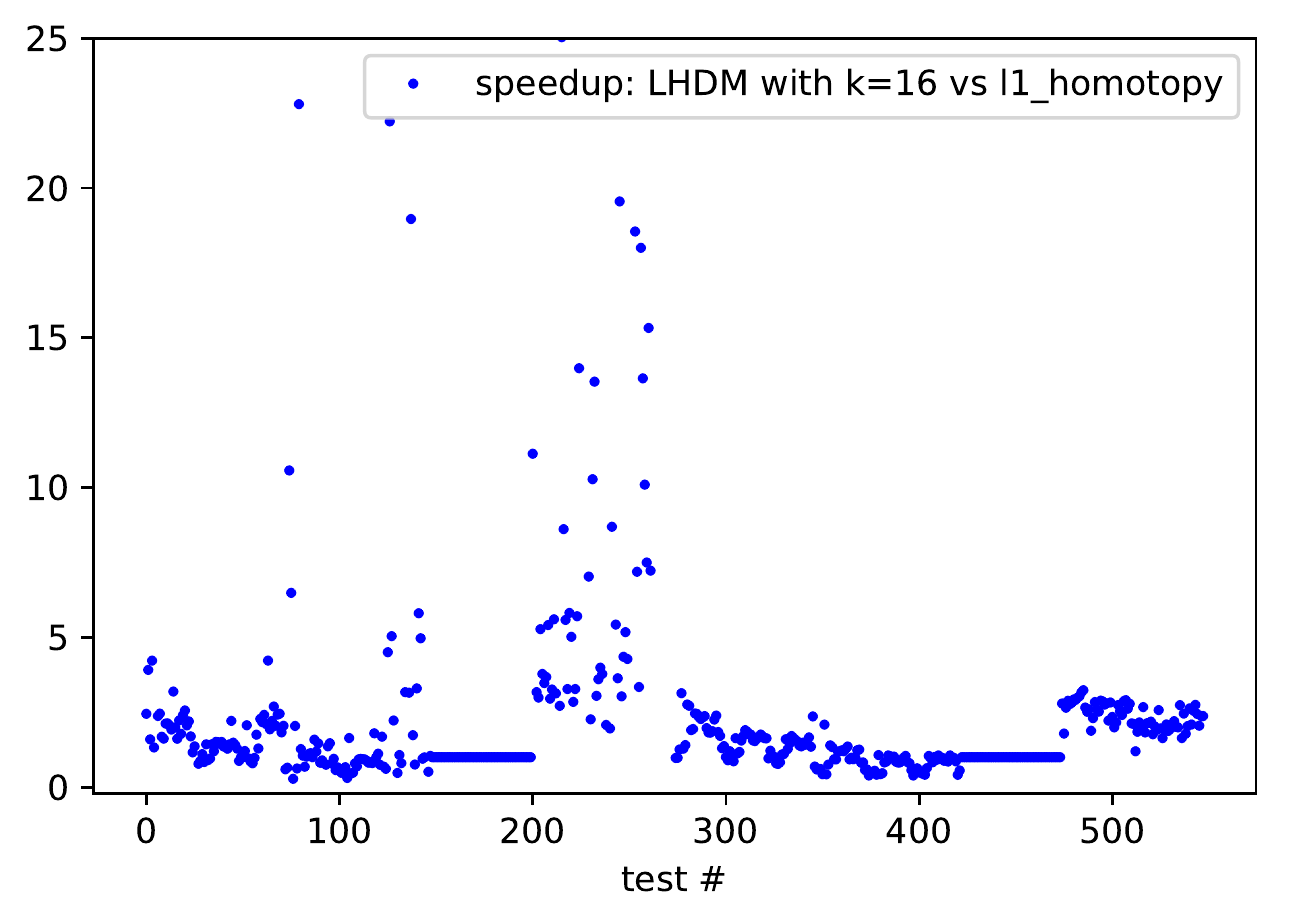}&
      \includegraphics[height=0.22\textwidth]{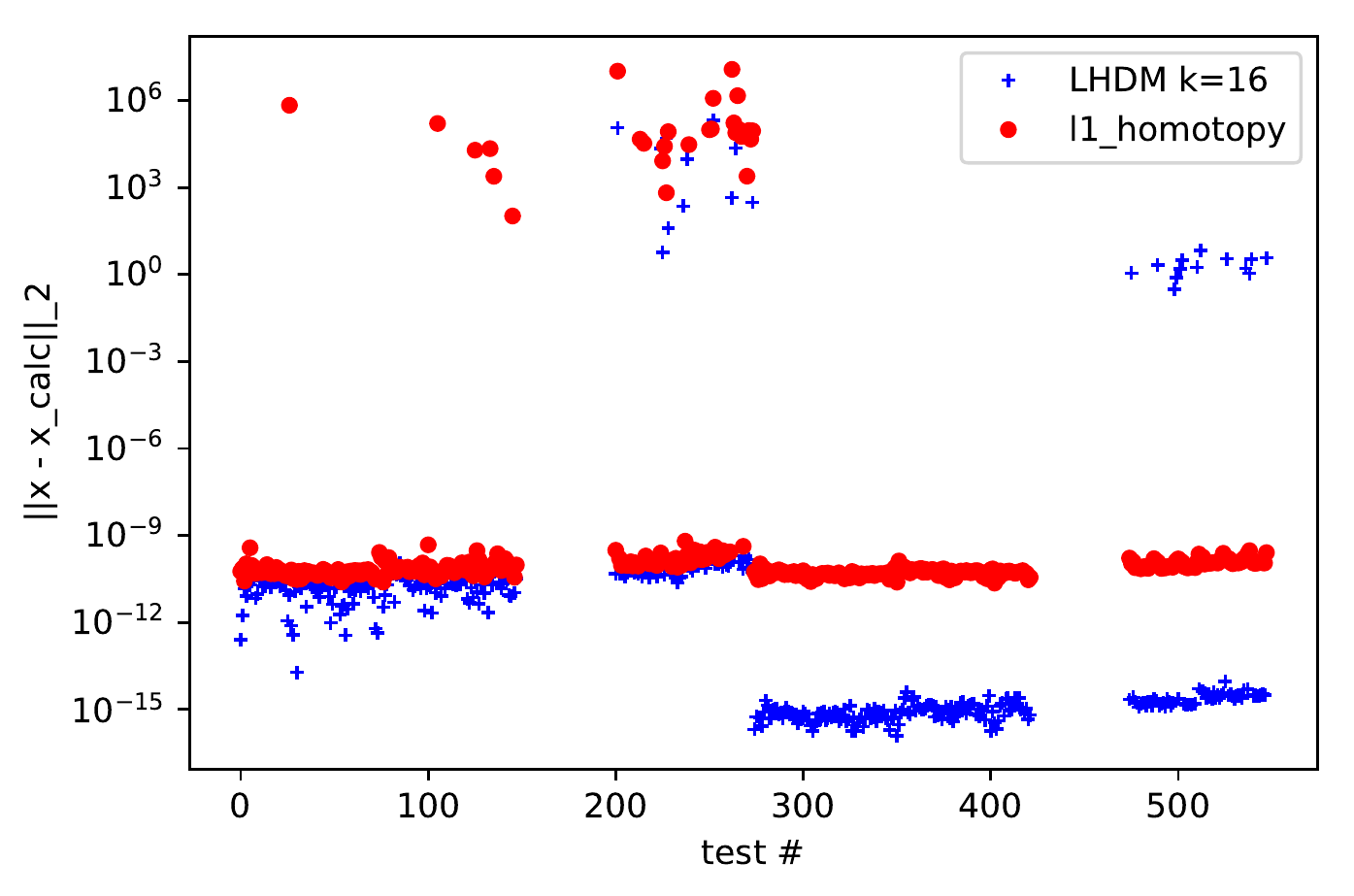}\\
      \includegraphics[height=0.22\textwidth]{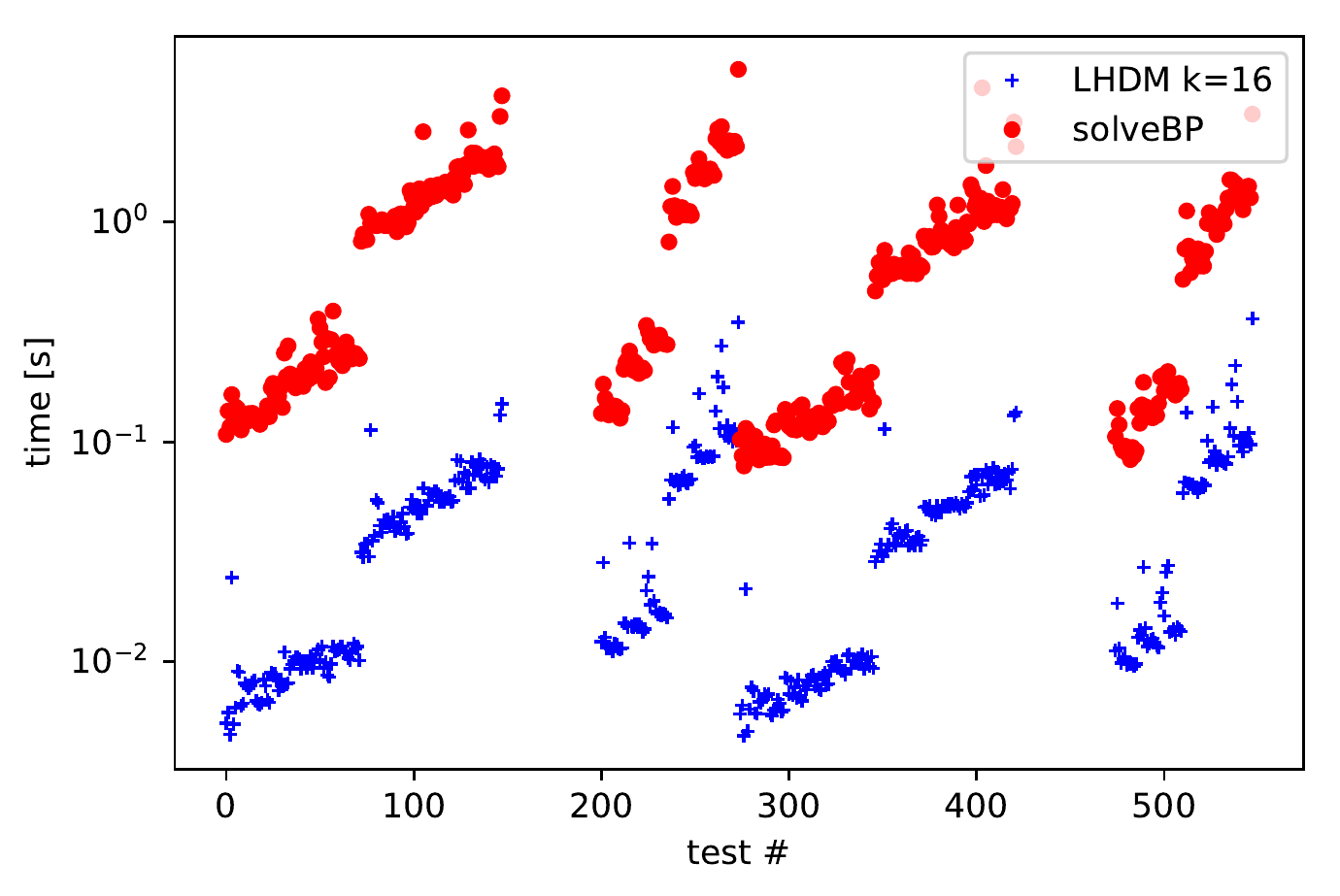}&
      \includegraphics[height=0.22\textwidth]{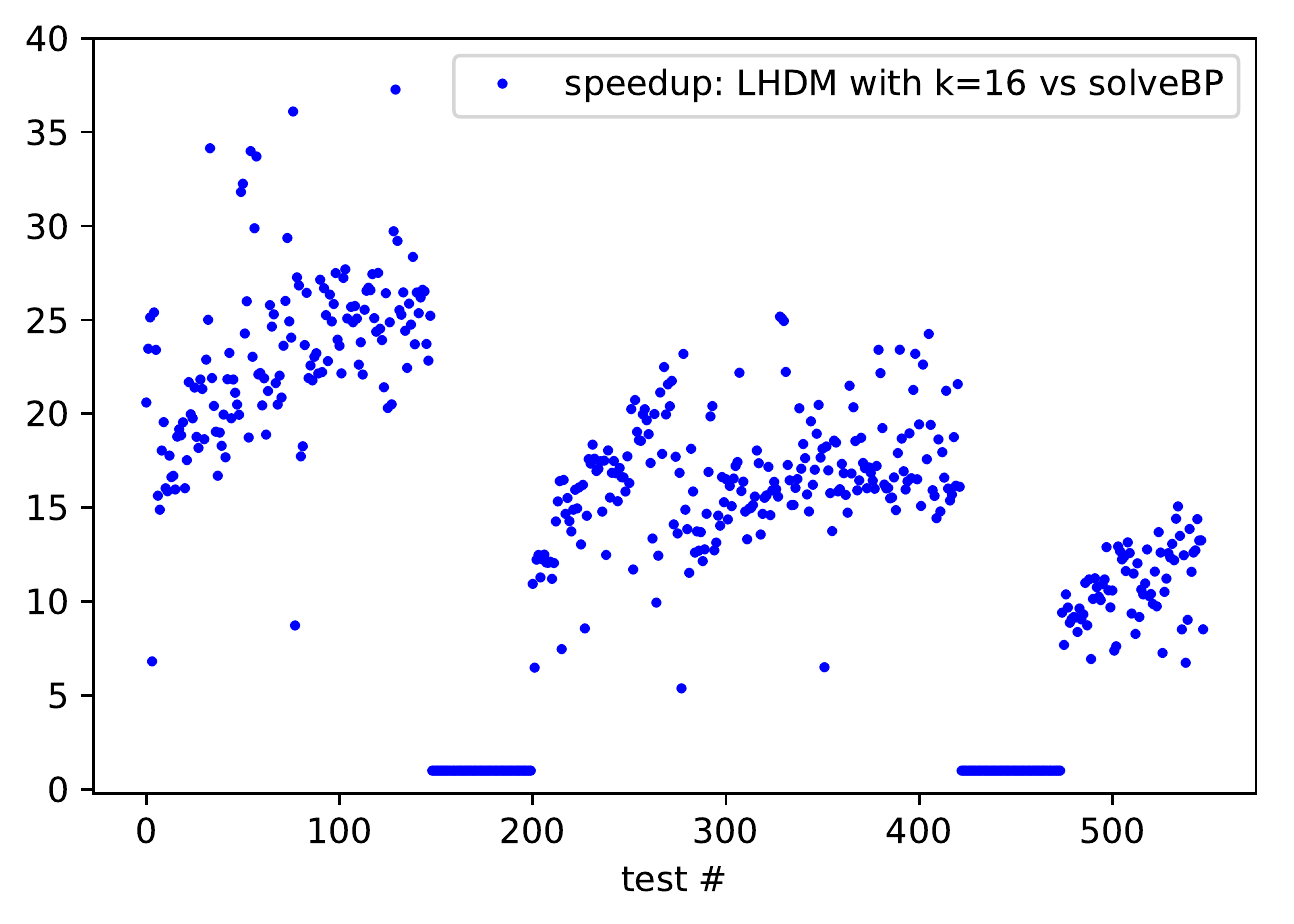}&
      \includegraphics[height=0.22\textwidth]{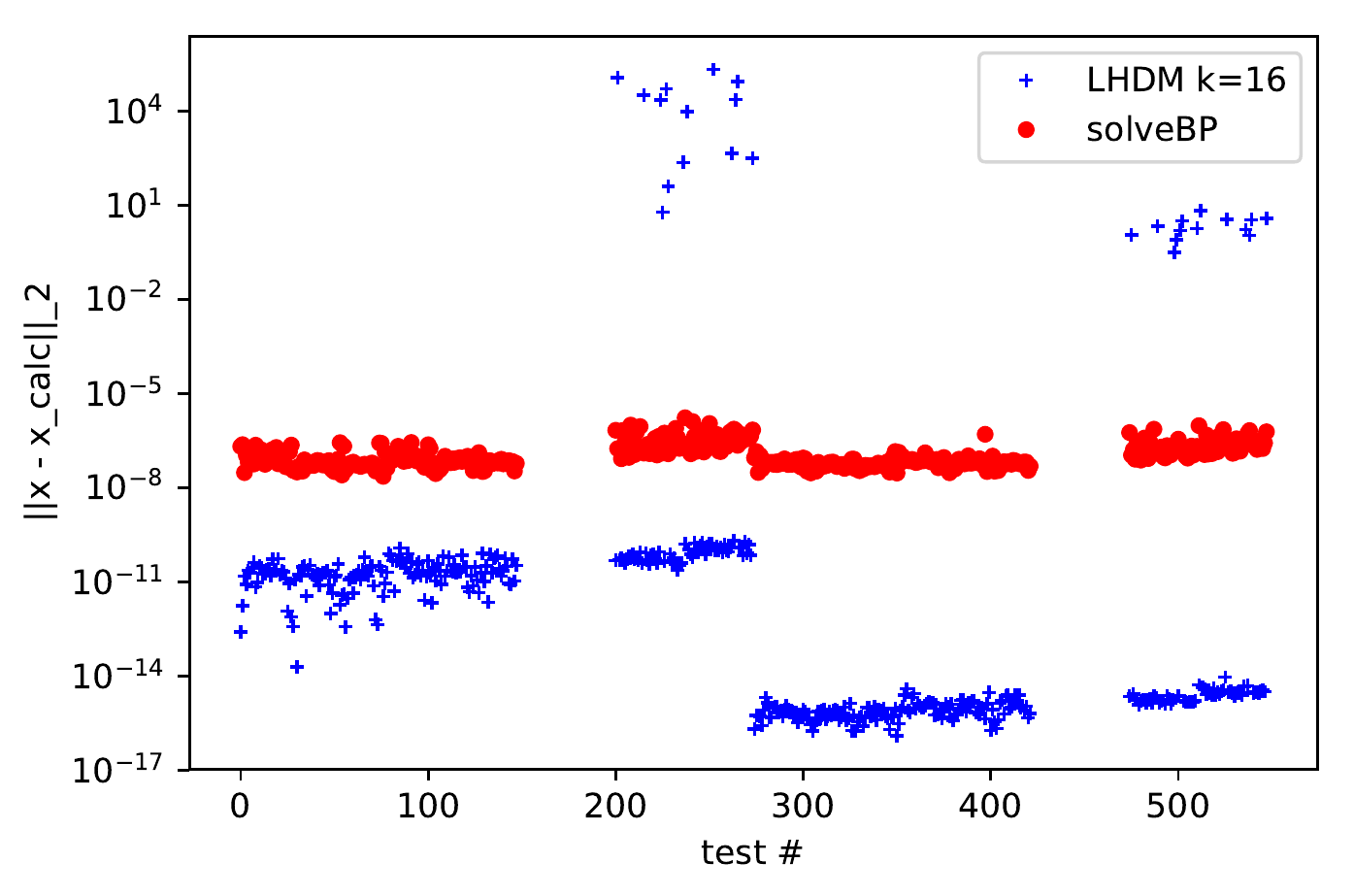}\\
      \includegraphics[height=0.22\textwidth]{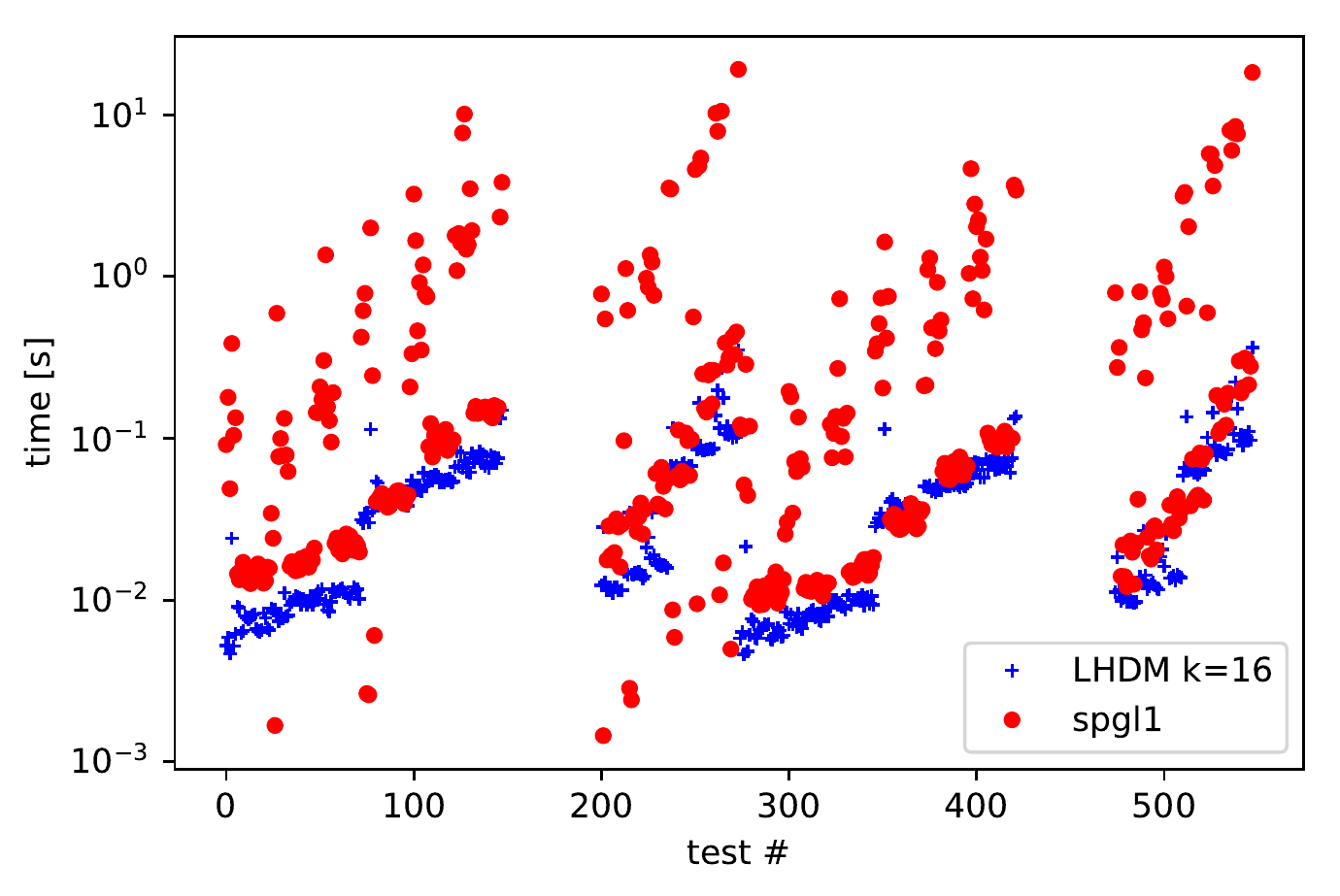}&
      \includegraphics[height=0.22\textwidth]{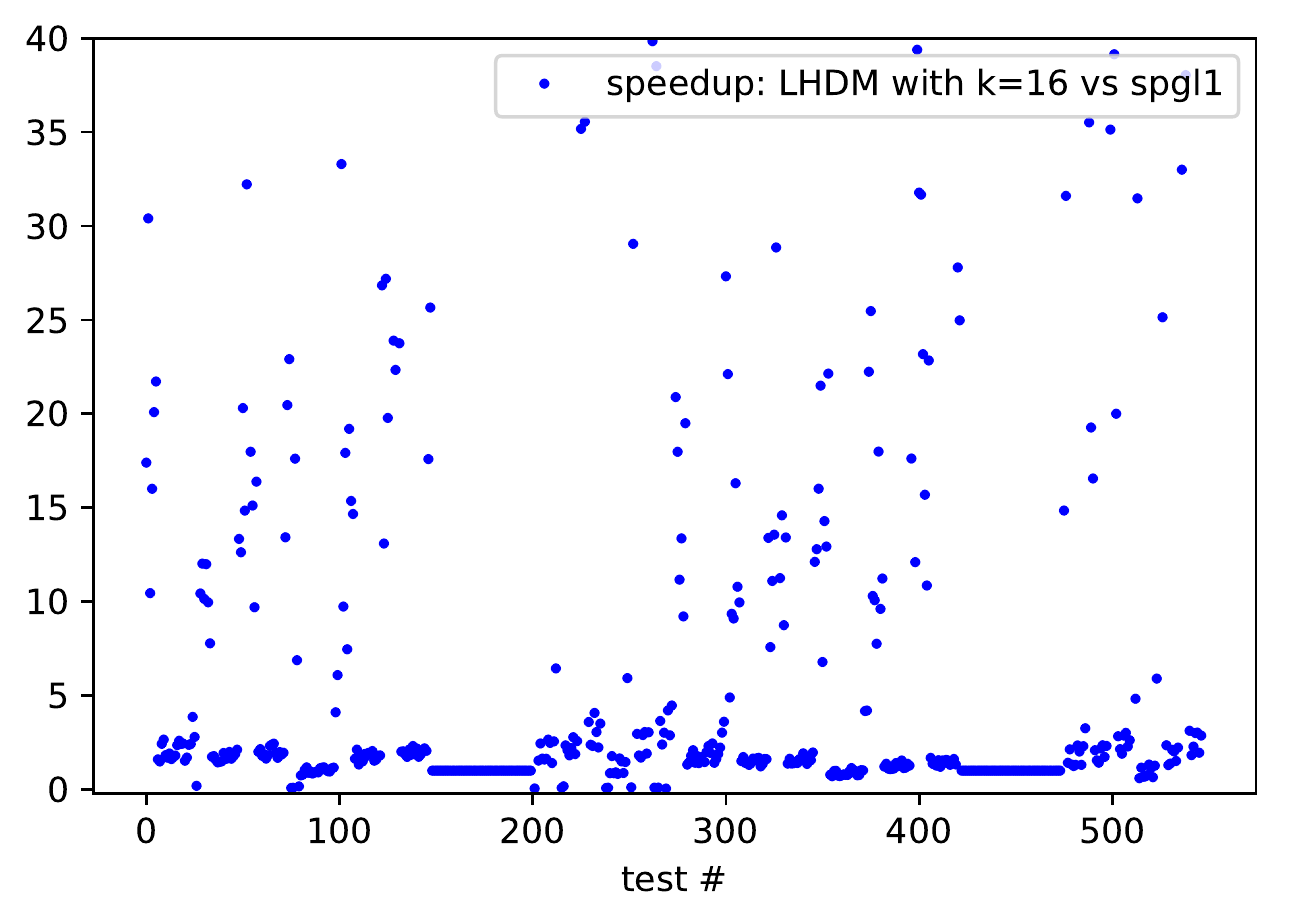}&
      \includegraphics[height=0.22\textwidth]{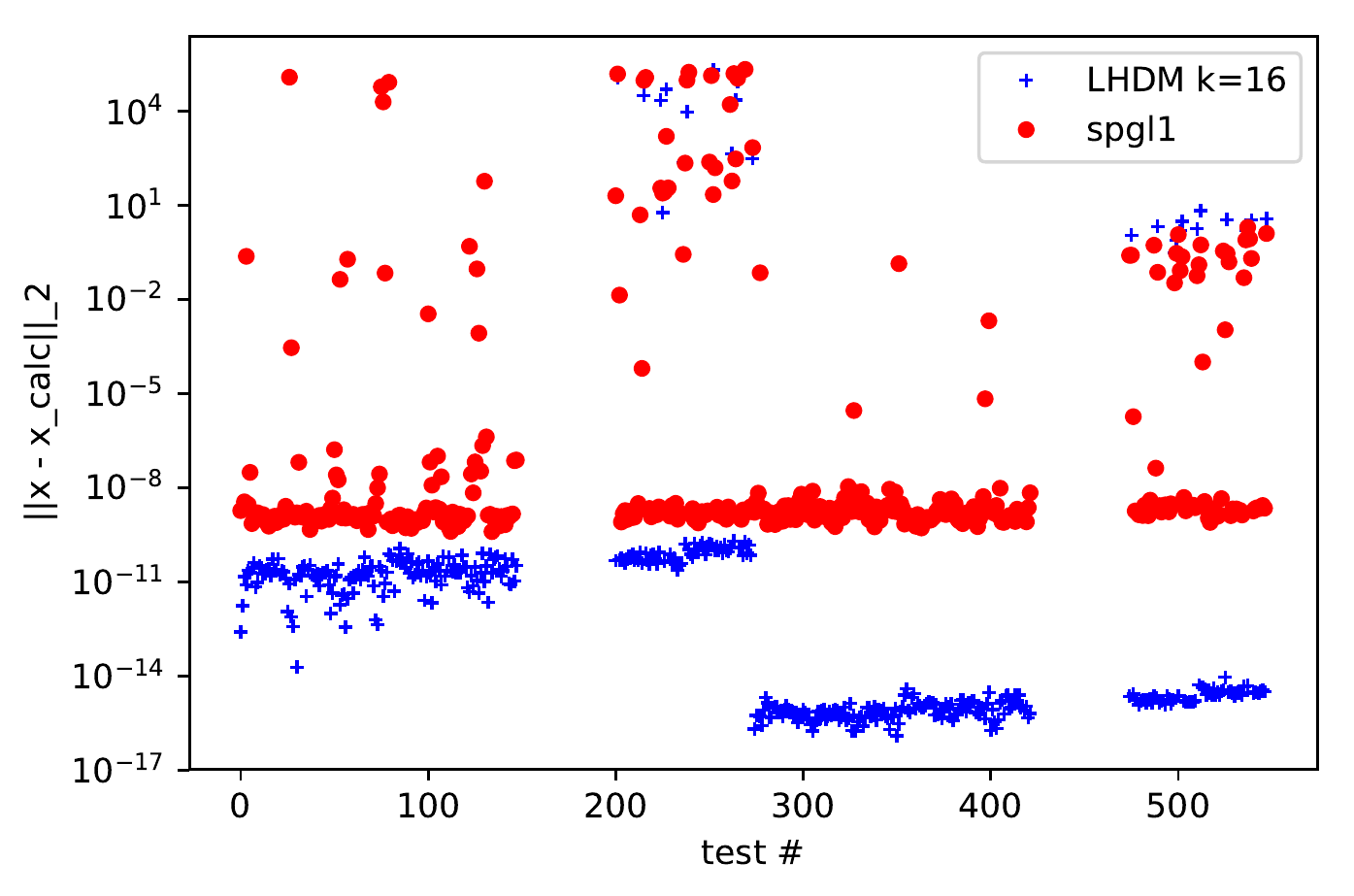}\\
      \includegraphics[height=0.22\textwidth]{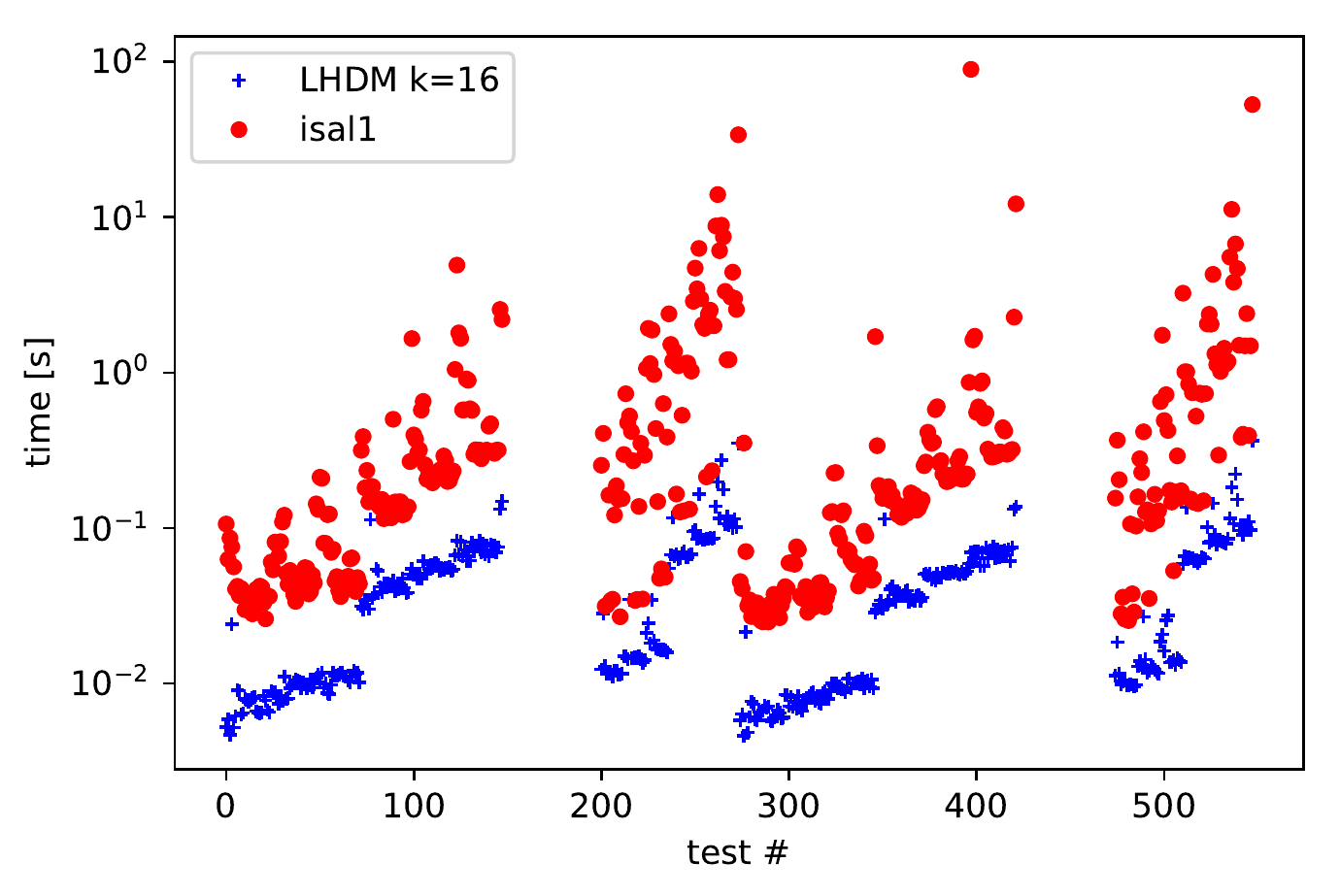}&
      \includegraphics[height=0.22\textwidth]{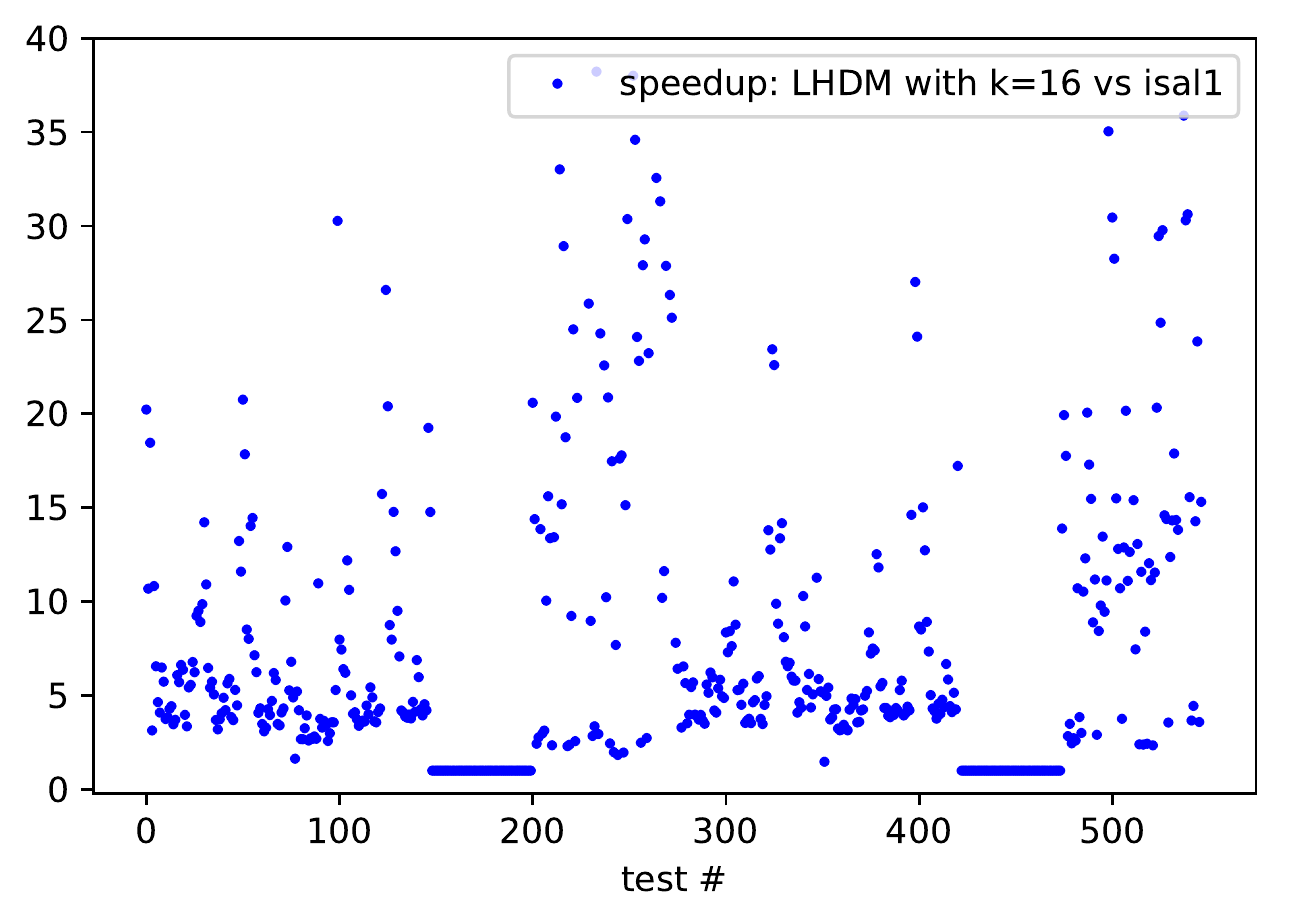}&
      \includegraphics[height=0.22\textwidth]{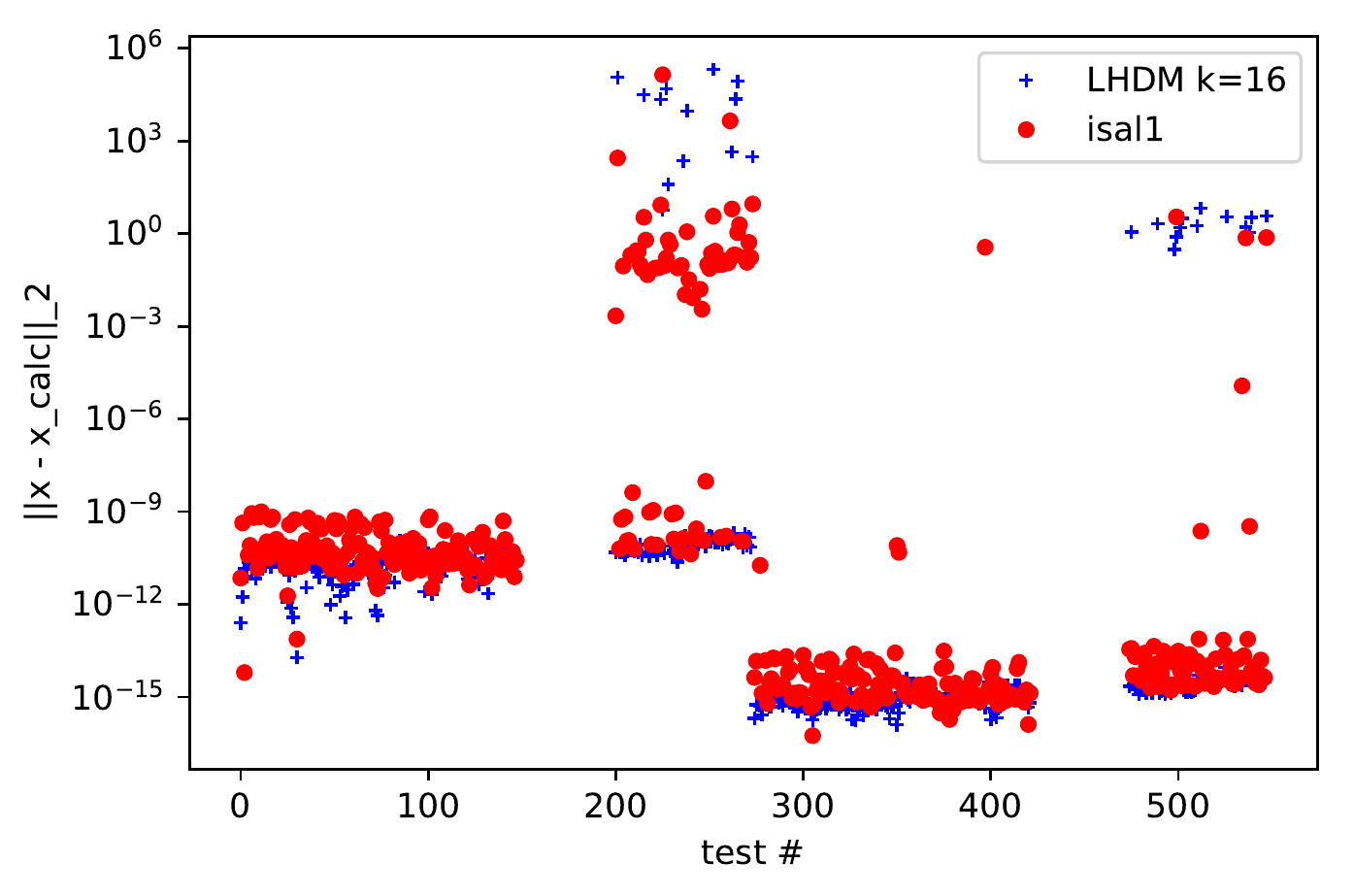}\\
      \includegraphics[height=0.22\textwidth]{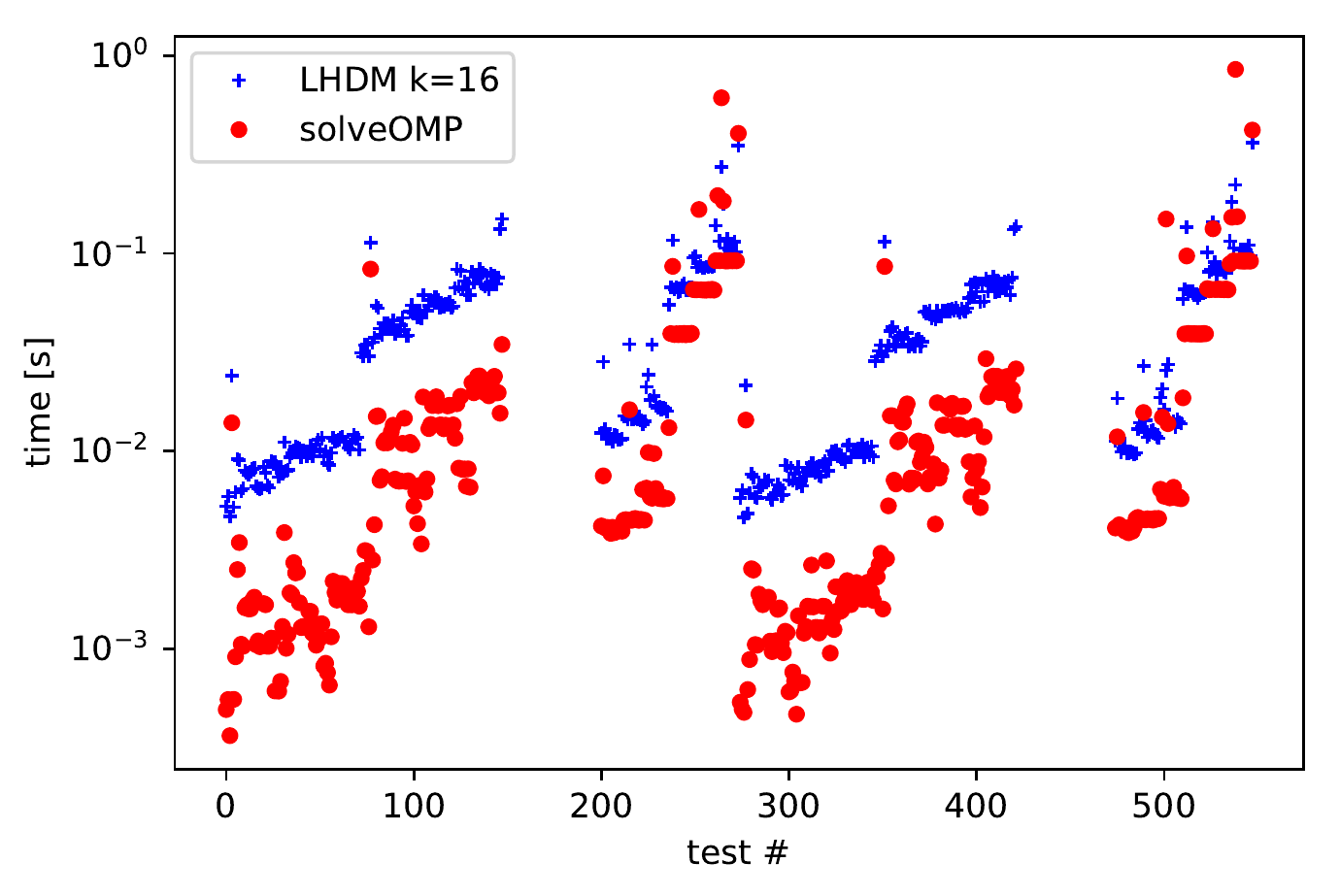}&
      \includegraphics[height=0.22\textwidth]{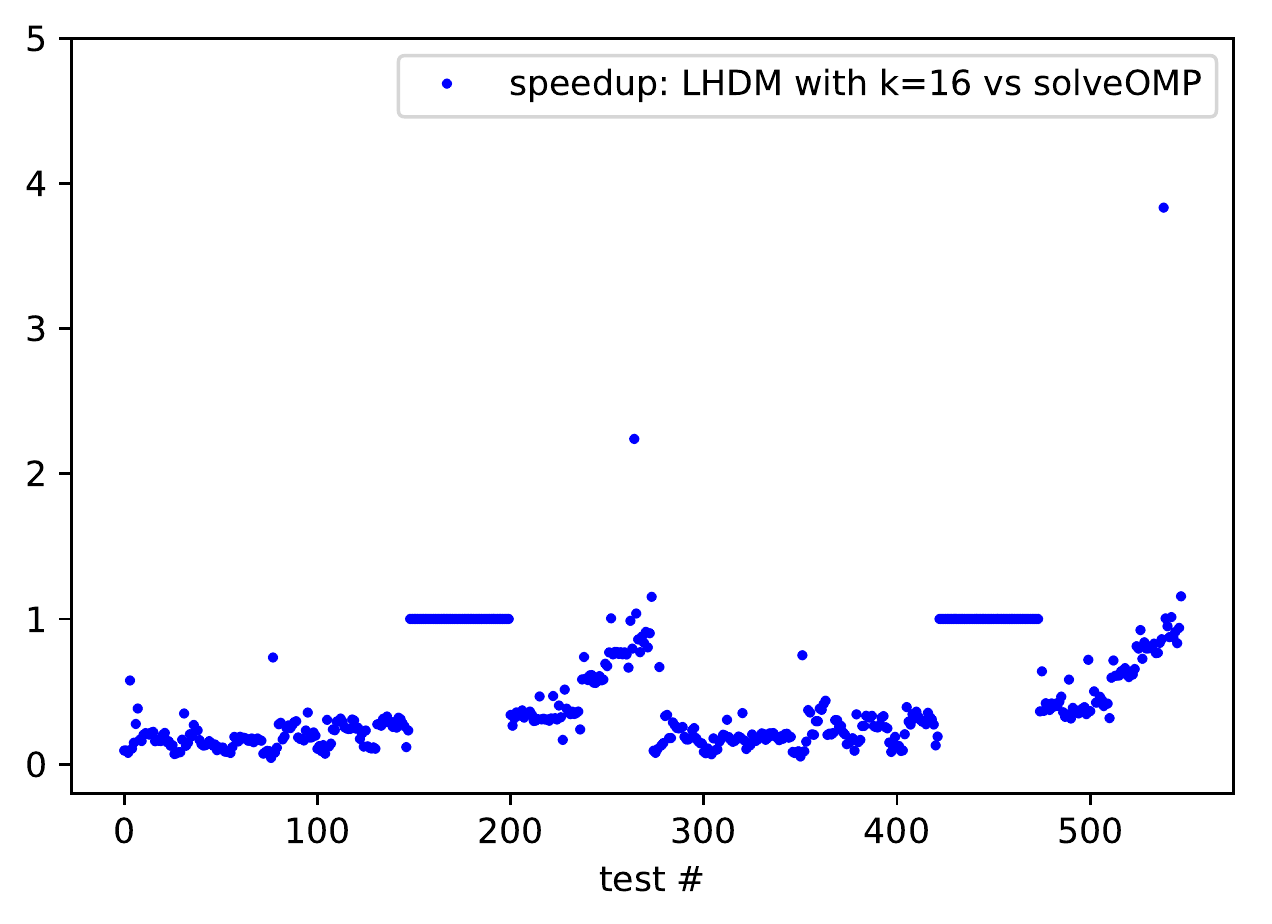}&
      \includegraphics[height=0.22\textwidth]{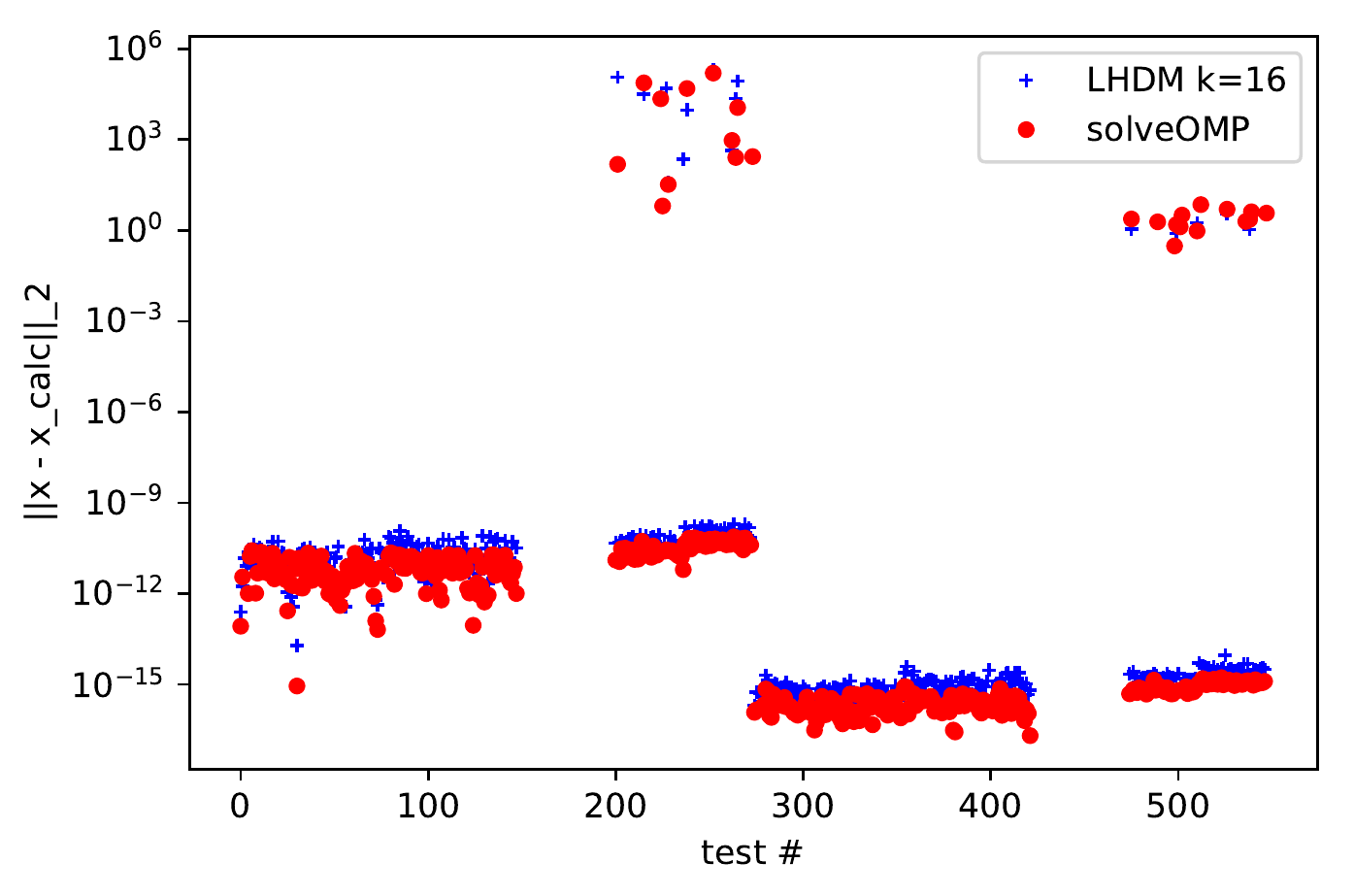}
    \end{tabular}
      \caption{Execution times (left), speedup of LHDM vs the other convex relaxation solvers listed at the beginning of this section (center) and distance to optimum $|| \mathbf{x} - \mathbf{x}_{calc} ||_2$ (right). Each row of plots corresponds to a convex relaxation solver, in this order (from the top): $l1\_magic$, $l1\_homotopy$, $solveBP$, $spgl1$, $isal1$ and $solveOMP$. The vector $\mathbf{x}_{calc}$ is the solution computed by the algorithm and $\mathbf{x}$ the true solution (unique, in ERC tests).} 
      \label{fig:comparison_database_1}
\end{figure}

We can notice that LHDM is a good performer except versus solveOMP (an implementation of the Orthogonal Matching Pursuit), that surprisingly has not been considered in \cite{AAVV_Solving_Basis_Pursuit-Heuristic_Optimality_Check_and_Solver_Comparison} for theoretical reasons, but actually outperforms all the other methods in the whole database considered. 
In the next subsection we will therefore go into a detailed comparison between LHDM and solveOMP.

\subsection{LHDM vs SolveOMP and the DM role}

To compare LHDM with SolveOMP we will perturb the database in three main directions: sparsity level, condition number of the system matrix, and nonnegativity of the solution.

Let us start with the first database modification. 
As already did in sec. \ref{sec:speedup_LHDM_vs_LH}, we created a few modified datasets, starting from the original, with random generated solutions at prescribed sparsity levels, and perturbing matrices to satisfy the ERC property at each instance.
As we can notice in Figure \ref{fig:results_4}, where a modified database with almost dense solution vectors is used, LHDM behaves quite better than SolveOMP now, with a substantial speedup (compare with Figure \ref{fig:comparison_database_1}, where the original database with very sparse solution vectors were used). If we consider also the speedup gain of LHDM vs LH with the same modified dataset, shown in Figure \ref{fig:results_5}, we can conclude that the performance gain of LHDM vs SolveOMP is due to the Deviation Maximization (DM) and not to LH alone. This is not surprising, since SolveOMP is based on Cholesky factorization, which is computationally lighter that the QR factorization on which LH is based. However, the QR is a more stable algorithm and it allows the DM to improve the performances considerably in a wide range of instances.

\begin{figure}[H]
\centering
    \begin{tabular}{ccc}
      \includegraphics[height=0.22\textwidth]{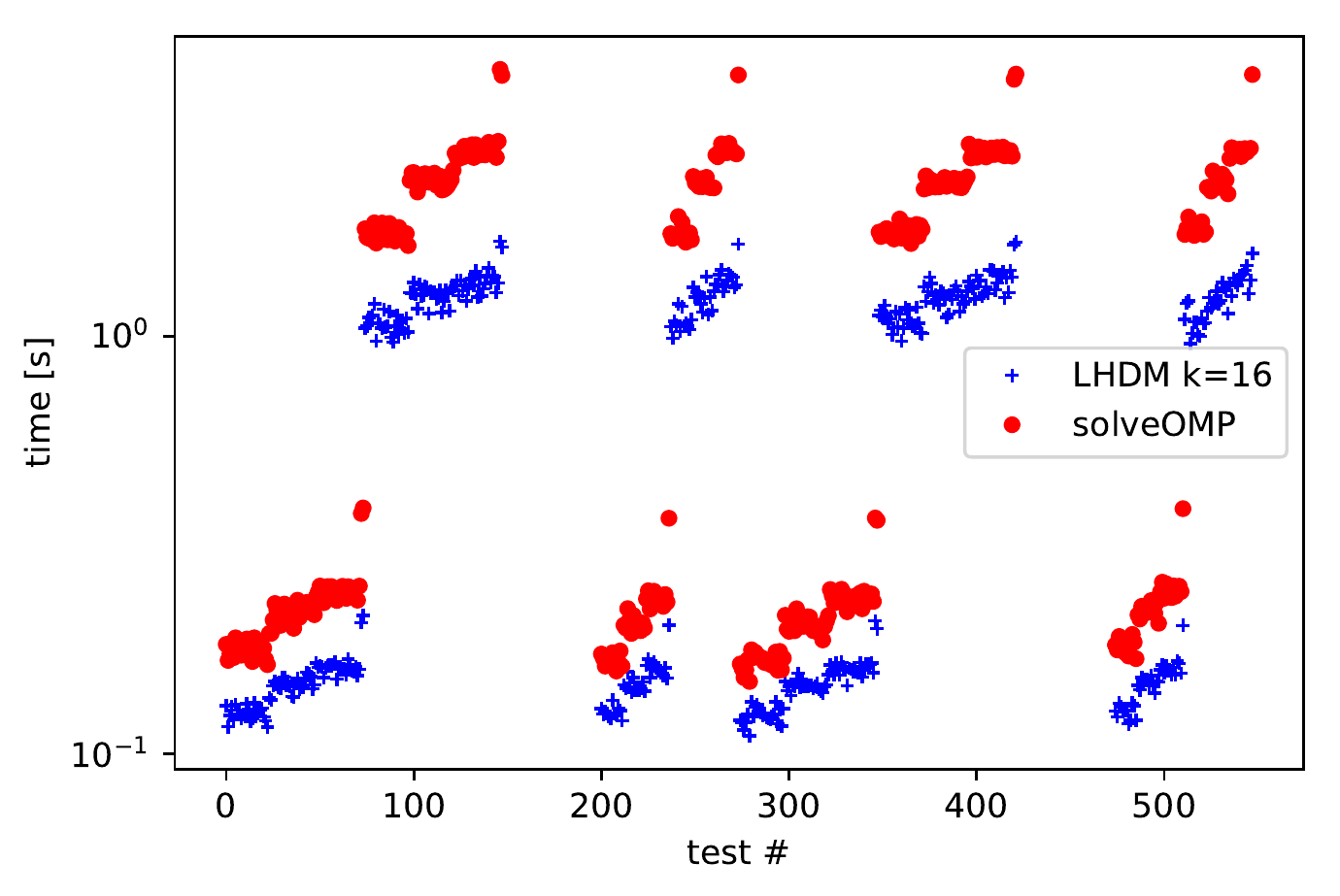}&
      \includegraphics[height=0.22\textwidth]{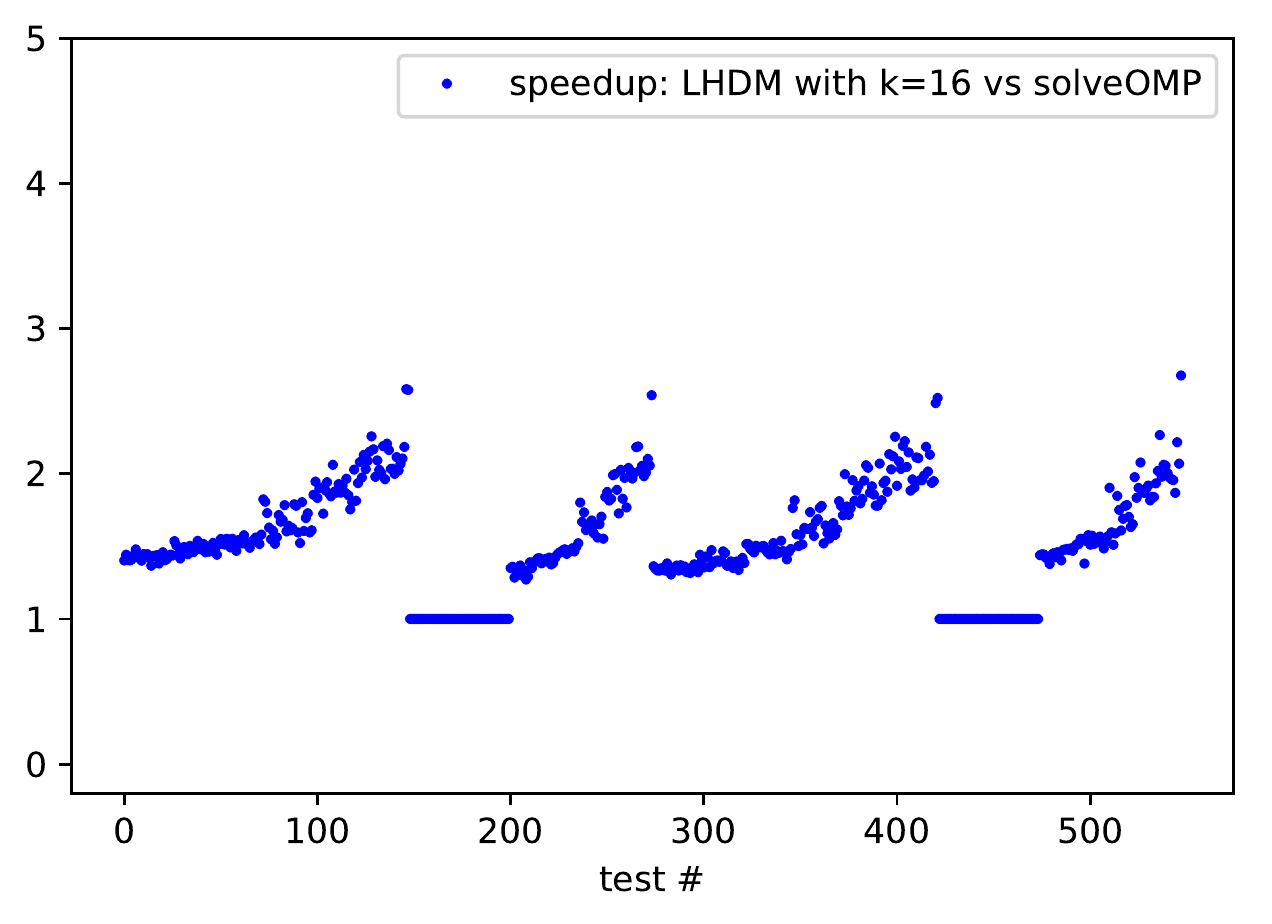}&
      \includegraphics[height=0.22\textwidth]{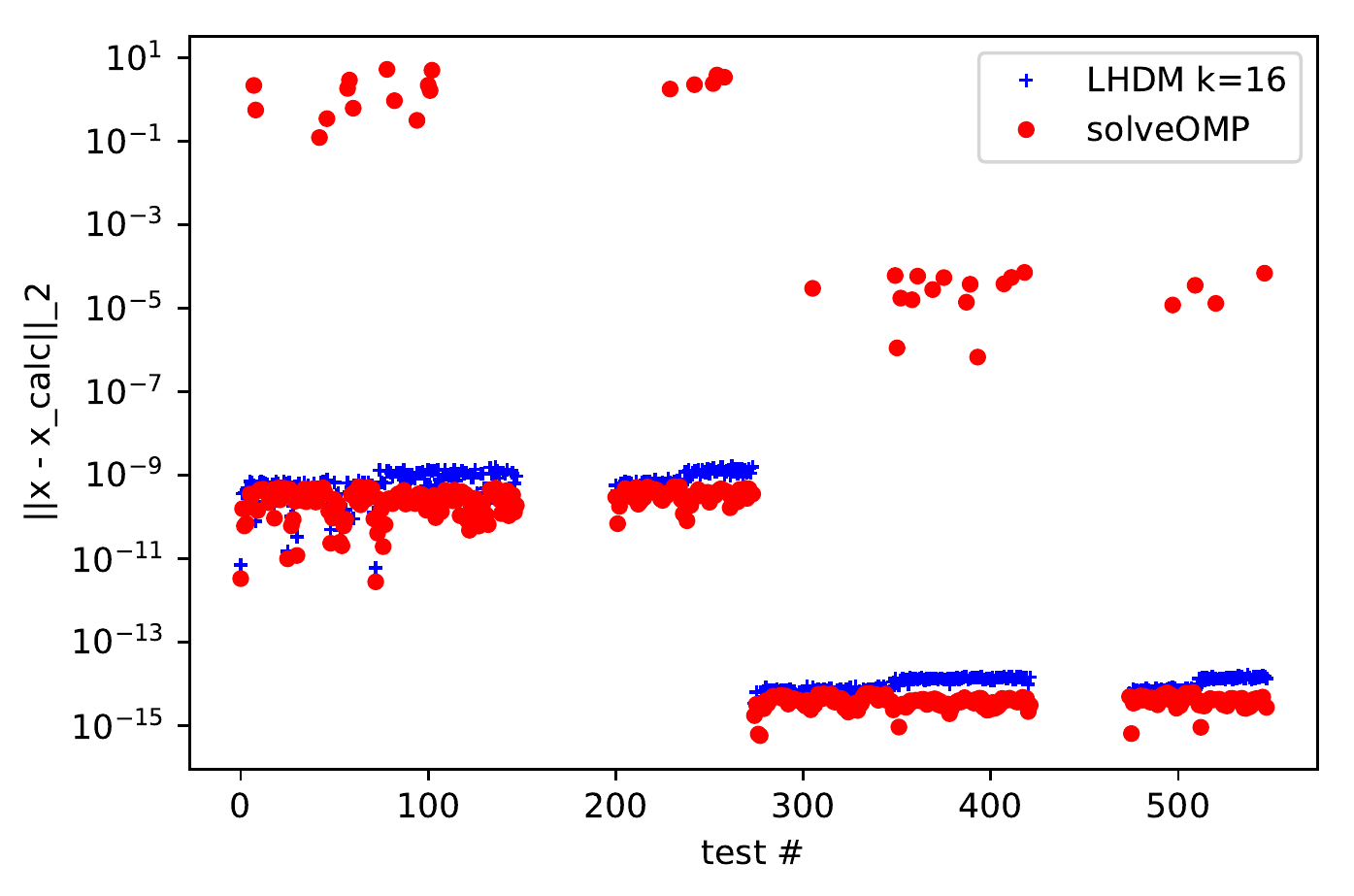}
    \end{tabular}
      \caption{Execution times (left), speedup of LHDM vs SolveOMP (center) and distance to optimum $|| \mathbf{x} - \mathbf{x}_{calc} ||_2$ (right), for the modified dataset, whose instances have almost dense solution vectors. The vector $\mathbf{x}_{calc}$ is the solution computed by the algorithm and $\mathbf{x}$ the true solution (unique, in ERC tests).} 
      \label{fig:results_4}
\end{figure}

This observation suggests us to try a modified dataset along the second direction: the matrix condition number.
The matrices of the original dataset in \cite{AAVV_Solving_Basis_Pursuit-Heuristic_Optimality_Check_and_Solver_Comparison} are very well conditioned (condition number $< 10^2$). 
We created a few modified datasets, starting from the original, with random generated matrices at prescribed condition numbers, and then perturbing the resulting matrices to satisfy the ERC property at each instance.
If the matrix condition number is moderate, say between $10^5$ and $10^6$, SolveOMP fails $\approx 70\%$ of instances, while LHDM solves all of them. Moreover, as we can notice in Figure \ref{fig:results_6}, LHDM keeps its performance gain vs LH, independently from the conditioning of the problem (compare with Figure \ref{fig:results_2}).

\begin{figure}[H]
\centering
    \begin{tabular}{ccc}
      \includegraphics[height=0.22\textwidth]{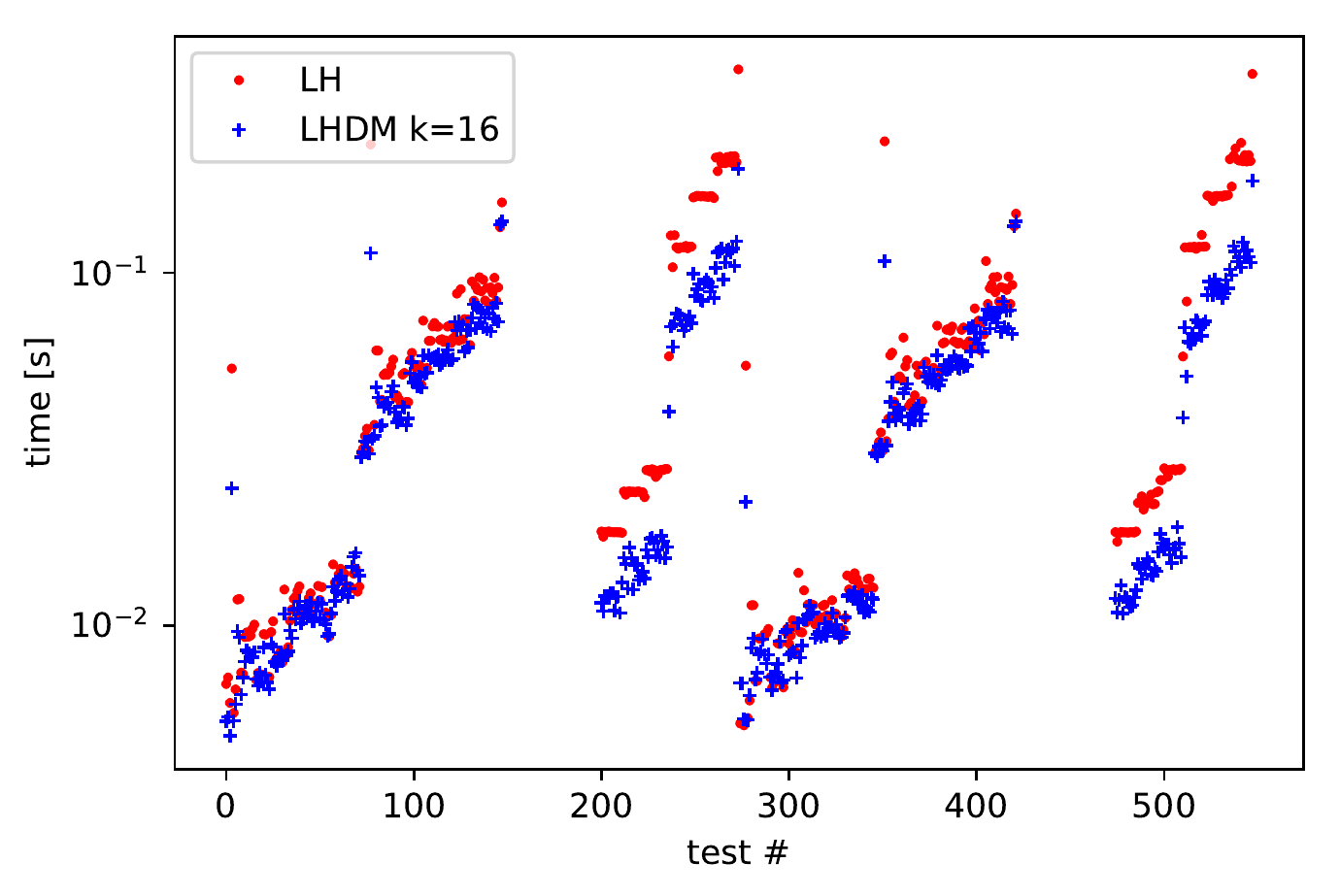}&
      \includegraphics[height=0.22\textwidth]{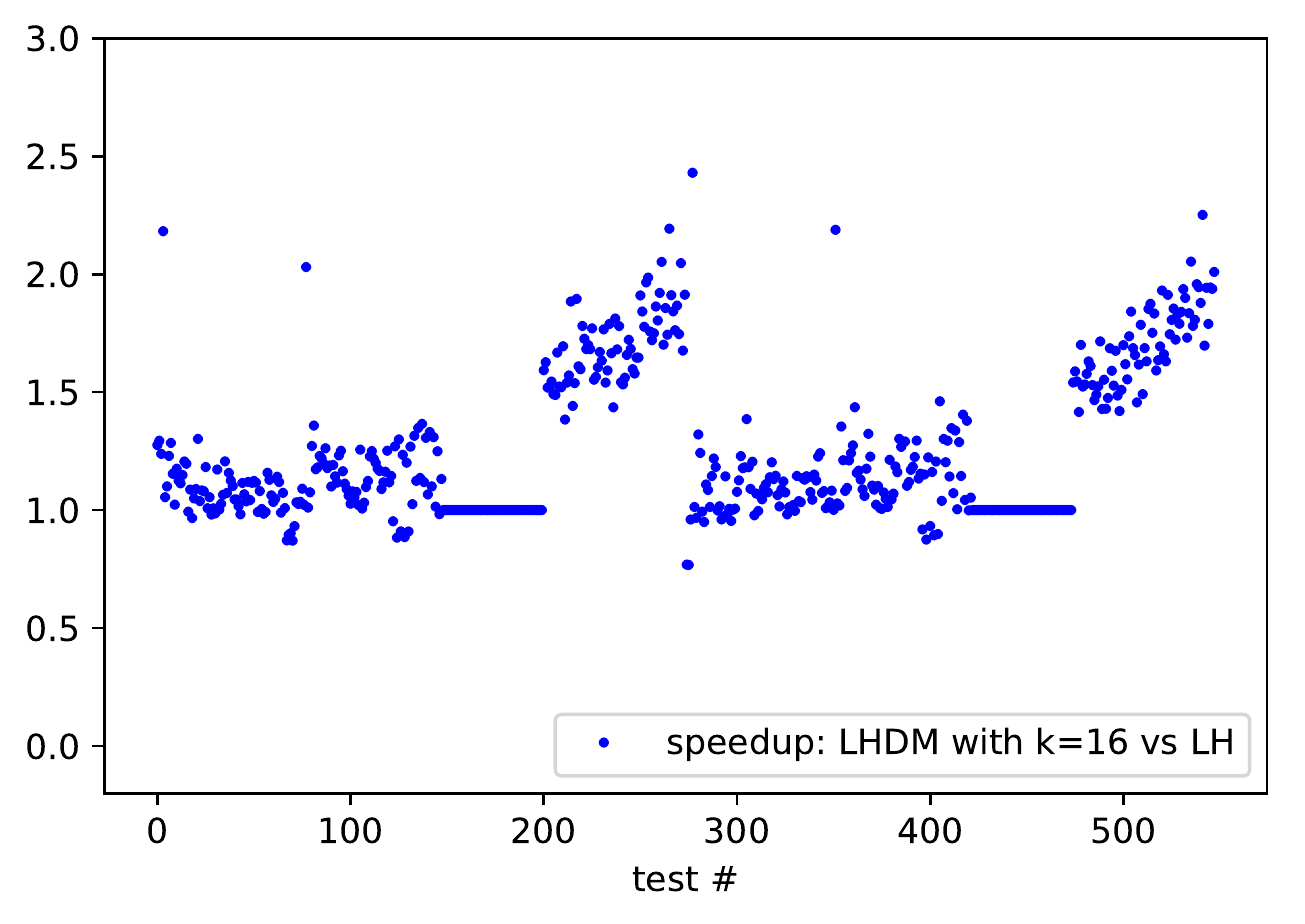}&
      \includegraphics[height=0.22\textwidth]{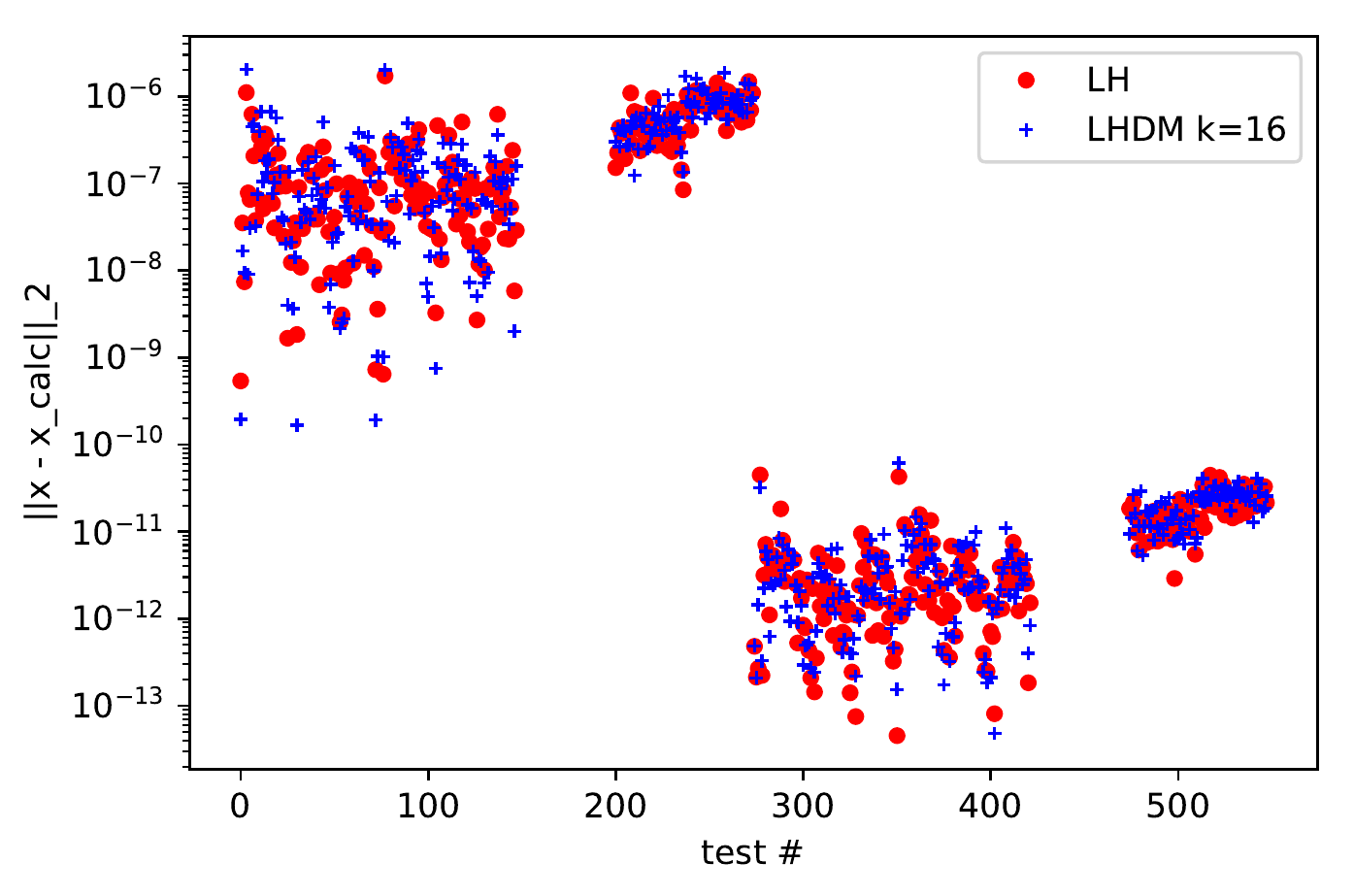}
    \end{tabular}
      \caption{Execution times (left), speedup of LHDM vs LH (center) and distance to optimum $|| \mathbf{x} - \mathbf{x}_{calc} ||_2$ (right), for the modified dataset, whose matrices have a moderate condition number (between $10^5$ and $10^6$). The vector $\mathbf{x}_{calc}$ is the solution computed by the algorithm and $\mathbf{x}$ the true solution (unique, in ERC tests).} 
    \label{fig:results_6}
\end{figure}


The third direction in which we modified the database is the nonnegativity of the solution. 
We created one modified dataset, starting from the original, by taking componentwise the absolute value of each solution vector, and then perturbing the matrices to satisfy the ERC property at each instance.
When a nonnegative solution is requested, SolveOMP does not satisfy this requirement, while LHDM does.
Moreover, as we can notice in Figure \ref{fig:results_7}, LHDM keeps its performance gain vs LH.

\begin{figure}[H]
\centering
    \begin{tabular}{ccc}
      \includegraphics[height=0.22\textwidth]{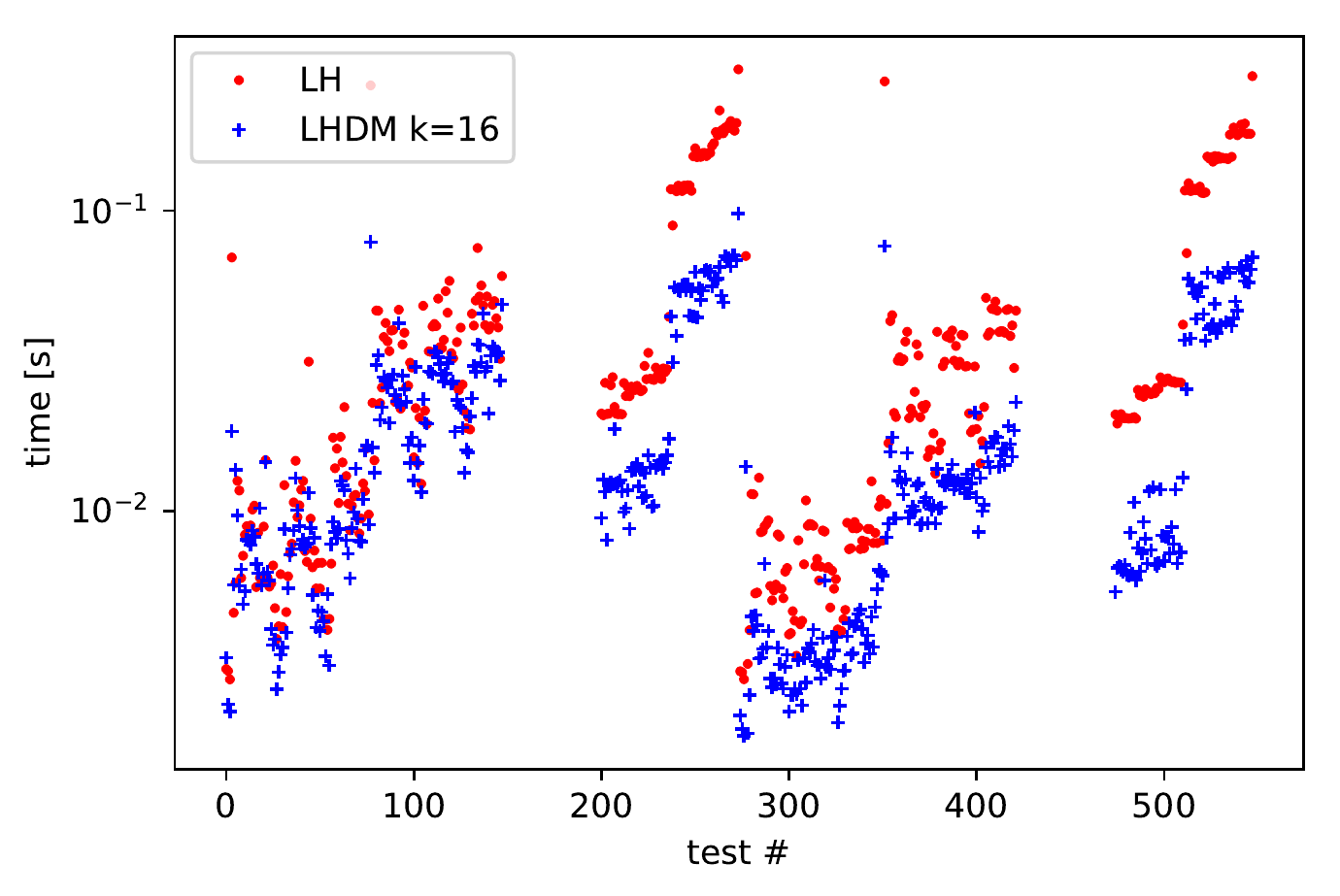}&
      \includegraphics[height=0.22\textwidth]{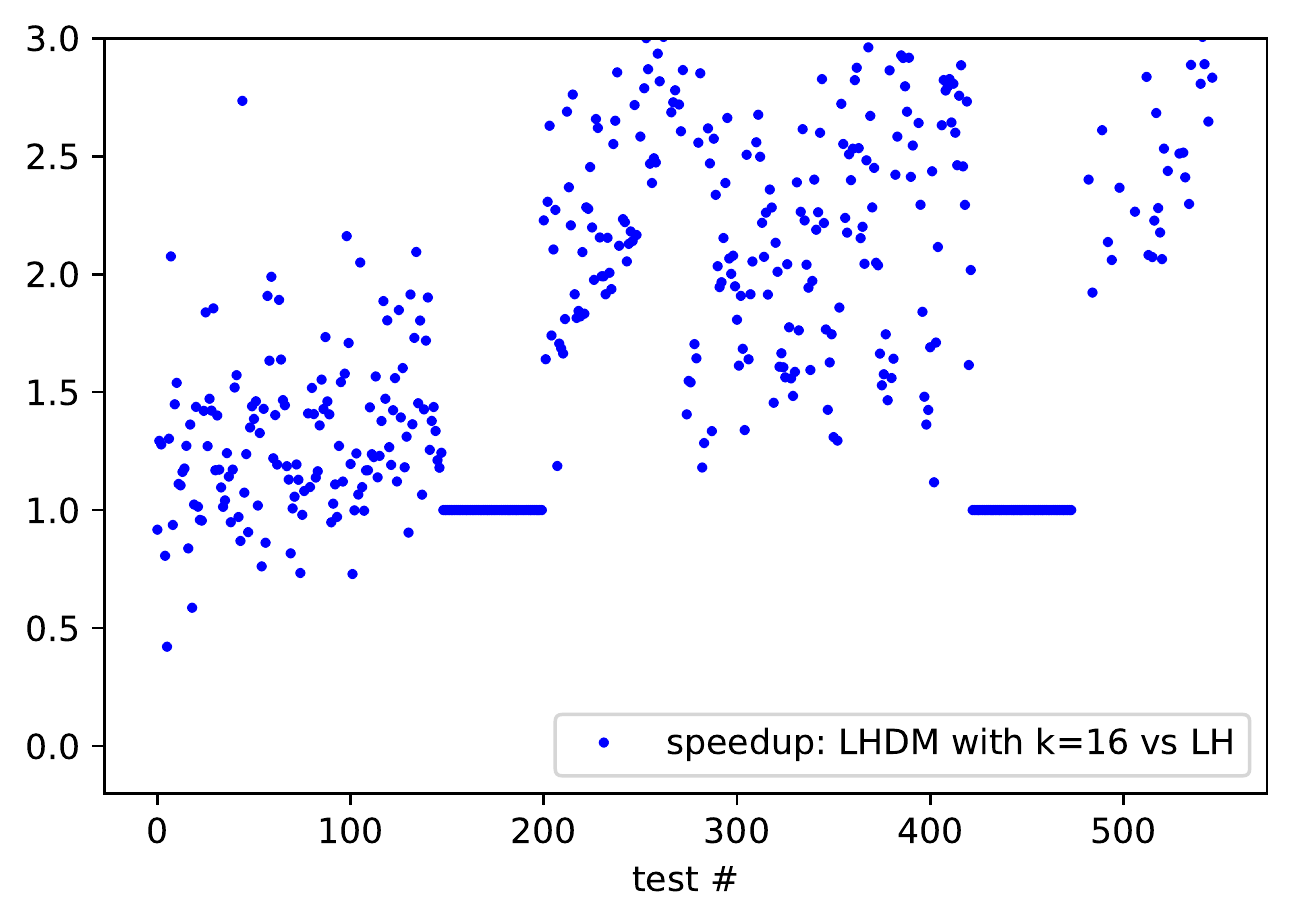}&
      \includegraphics[height=0.22\textwidth]{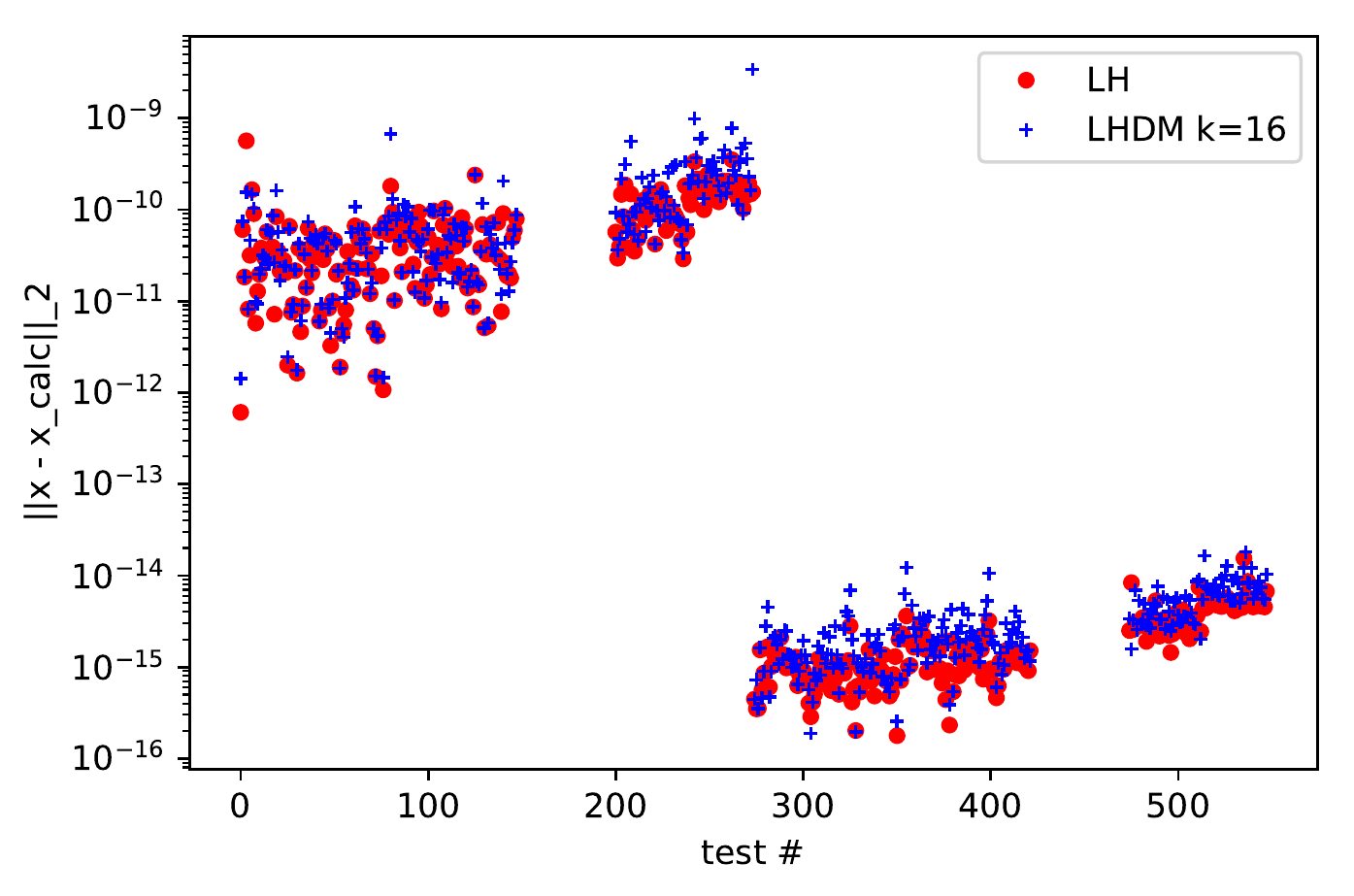}
    \end{tabular}
      \caption{Execution times (left), speedup of LHDM vs LH (center) and distance to optimum $|| \mathbf{x} - \mathbf{x}_{calc} ||_2$ (right), for the modified dataset, whose solution vectors are componentwise the absolute value of each solution vector of the original dataset. The vector $\mathbf{x}_{calc}$ is the solution computed by the algorithm and $\mathbf{x}$ the true solution (unique, in ERC tests).} 
    \label{fig:results_7}
\end{figure}

\section{Conclusions} \label{sec:conclusions}
In this work we have presented a new NNLS solver, namely the Lawson-Hanson with Deviation Maximization algorithm (LHDM), based on the “deviation maximization” columns selection strategy. Our method relies on cosine evaluation in order to select a subset of sufficiently linearly independent columns to compute a feasible descend direction and we have used it as column pivoting strategy to devise a new active set method. We have provided a theoretical analysis proving the finite convergence of LHDM, which is shown to terminate in at most $2^n$ steps, just like the standard Lawson-Hanson. This is a worst case bound, and in practice a polynomial rate of convergence is observed.

We have given also the details for an efficient implementation of LHDM, through Householder and Givens orthogonal transformations. In particular, the algorithmic clue given for the positivity-trick case in sec. \ref{sec:positivity-trick} is by itself a novel contribution, as far as we know.

Extensive numerical experiments have been carried out over a wide set of instances, confirming that LHDM yields a significant performance gain over LH with an average speedup of $3\times$ with peaks up to $6\times$. Moreover, we compared it with several $\ell_1$ solvers whose implementation is publicly available online. Numerical testing has confirmed that LHDM is competitive with $\ell_1$ solvers for sparse recovery in terms of solution quality and execution times, revealing it as a good choice for sparse recovery for a wide class of instances.

The proposed LHDM algorithm exploits BLAS-3 operations for efficiency. A thorough experimental evidence is not possible in this Matlab implementation and it requires an implementation in a compiled programming language, e.g. the C language, which is an ongoing project, for an open-source delivery. With a C implementation, the performances seen in sec. \ref{sec:numerical_experiments} should sensibly improve, like we did experience with QRDM \cite{DMa21}, included the speedup with respect to LH.

We did not cope with sparse instances, since it would require a dedicated implementation based on sparse QR with column pivoting and would be interesting to compare with randomization techniques. Finally, it would be also interesting to extend the deviation maximization as pivoting strategy to other problems which require column selection, e.g. more general constrained optimization problems such as least squares problems with linear inequality constraints.

\section*{Acknowledgments}
The authors gratefully acknowledge the doctoral grant funded by BeanTech s.r.l. ``GPU computing for modeling, nonlinear optimization and machine learning". 
This work was partially supported by the Project BIRD192932 of the University of Padova.

\bibliography{LHDM_revised_preprint}
\end{document}